\documentclass[a4paper,11pt,english,reqno]{amsart}
\usepackage{amsmath,amssymb,amsfonts,dsfont,amsthm}    
\usepackage[utf8]{inputenc}
\usepackage[english]{babel}
\usepackage[T1]{fontenc} 
\usepackage{graphicx}
\usepackage{float}
\usepackage[font=small,labelfont=bf]{caption}
\usepackage[a4paper]{geometry}
\usepackage{enumerate}

\newcommand{\defeq}{\stackrel{\mathrm{def}}{=}}
\newcommand{\prob}{\mathbf{P}}
\newcommand{\erw}{\mathds{E}}
\newcommand{\e}{\mathrm{e}}

\newcommand{\tr}{\mathrm{tr\,}}

\newcommand{\supp}{\mathrm{supp\,}}
\newcommand{\spec}{\mathrm{Spec}}
\newcommand{\Ima}{\mathrm{Im\,}}
\newcommand{\Rea}{\mathrm{Re\,}}
\newcommand{\dist}{\mathrm{dist\,}}
\newcommand{\diag}{\mathrm{diag}}
\newcommand{\mO}{\mathcal{O}}
\newcommand{\C}{\mathds{C}}
\newcommand{\R}{\mathds{R}}
\newcommand{\T}{\mathds{T}}
\newcommand{\N}{\mathds{N}}
\newcommand{\Z}{\mathds{Z}}

\newcommand{\cS}{\mathcal{S}}

\newcommand{\tth}{\widetilde{h} }
\newtheorem{thm}{Theorem}
\newtheorem{cor}[thm]{Corollary}
\newtheorem{con}[thm]{Conjecture}
\newtheorem{prop}[thm]{Proposition}
\newtheorem{lem}[thm]{Lemma}
\newtheorem{rem}[thm]{Remark}
\theoremstyle{remark}

\numberwithin{equation}{section}
\usepackage{hyperref}
\hypersetup{pdfborder=0 0 0, 
	    colorlinks=true,
	    citecolor=blue,
	    linkcolor=blue,
	    urlcolor=blue,
	    pdfauthor={Martin Vogel}
	   }
\title{Almost sure Weyl law for quantized tori}
\author{Martin Vogel}
\address[Martin Vogel]{Institut de Recherche Math{\'e}matique Avanc{\'e}e - UMR 7501, 
Universit{\'e} de Strasbourg et CNRS, 7 rue René-Descartes, 67084 Strasbourg Cedex, France.}
\email{vogel@math.unistra.fr}
 \date{}
 \keywords{Spectral theory; non-self-adjoint operators; random perturbations}
\subjclass[2010]{47A10, 47B80, 47H40, 47A55}
\date{}
\begin{document}
\maketitle
\begin{abstract}
We study the eigenvalues of the Toeplitz quantization of complex-valued 
functions on the torus subject to small random perturbations given by a 
complex-valued random matrix whose entries are independent 
copies of a random variable with mean $0$, variance $1$ and bounded fourth moment. 
We prove that the eigenvalues of the perturbed operator satisfy a Weyl law with probability 
close to one, which proves in particular a conjecture by T.~Christiansen and M.~Zworski \cite{ZwChrist10}.
 \end{abstract}
 \setcounter{tocdepth}{1}
\tableofcontents
\section{Introduction}\label{int}
In this paper we consider Toeplitz quantizations of complex-valued functions on the $2d$-dimensional  
dimensional torus $\T^{2d}=\R^{2d}/\Z^{2d}$. This quantization maps smooth functions to $N^d\times N^d$ 
matrices (in general non-selfadjoint), 
\begin{equation}\label{ii1}
		C^{\infty}(\T^{2d}) \ni p \mapsto p_N\in \mathcal{L}(\C^{N^d},\C^{N^d}).
\end{equation}
We will describe this procedure in Section \ref{main} and in detail in Section \ref{sec:SC}. However, 
as in \cite{ZwChrist10}, we observe that when $d=1$, then $\T^2= S^1_x\times S^1_{\xi}$ and 
\begin{equation}\label{ii2}
		\begin{split}
			&f=f(x) \mapsto f_N=\diag (f(l/N); l=0,\dots ,N-1)\\
			&g=g(\xi) \mapsto g_N=\mathcal{F}_N^*\,\diag (g(l/N); l=0,\dots ,N-1)\mathcal{F}_N,
		\end{split}
\end{equation}
where $\mathcal{F}_N^*= N^{-1/2}(\exp(2\pi i k\ell/N))_{0\leq k,\ell \leq N-1}$ is the discrete Fourier 
transform. In the case of $\T^2$, the operators $p_N$ are also referred to as \emph{twisted Toeplitz matrices}, 
see \cite{TrEm05,TrCh04}. 
\\
\par
As an example we may consider the \emph{Scottish flag operator} \cite{ZwChrist10,TrEm05} 
given by the symbol
\begin{equation}\label{ii3.2}
		p(x,\xi) = \cos(2\pi x) + i\cos(2\pi \xi), \quad (x,\xi)\in \T^2.
\end{equation}
From \eqref{ii2} we get that 
\begin{equation}\label{ii3.1}
		p_N = \begin{pmatrix}
			  \cos x_1 & i/2 & 0 & 0 & \dots & i/2 \\
			  i/2  & \cos x_2 & i/2  & 0 &\dots & 0\\
			  0 & i/2  & \cos x_3 & i/2  & \ddots &0 \\ 
			  \vdots&  \ddots & \ddots & \ddots & \ddots & \vdots \\ 
			  0  & \dots & 0& i/2  & \cos x_{N-1} &i/2 \\ 
			  i/2   & 0 & \dots & 0 & i/2  & \cos x_N  \\ 
		           \end{pmatrix}
\end{equation}
where $x_j = 2\pi j/N$, $j=1,\dots, N$.
\begin{figure}[ht]
  \centering
   \begin{minipage}[b]{0.49\linewidth}
  \centering
   \includegraphics[width=\textwidth]{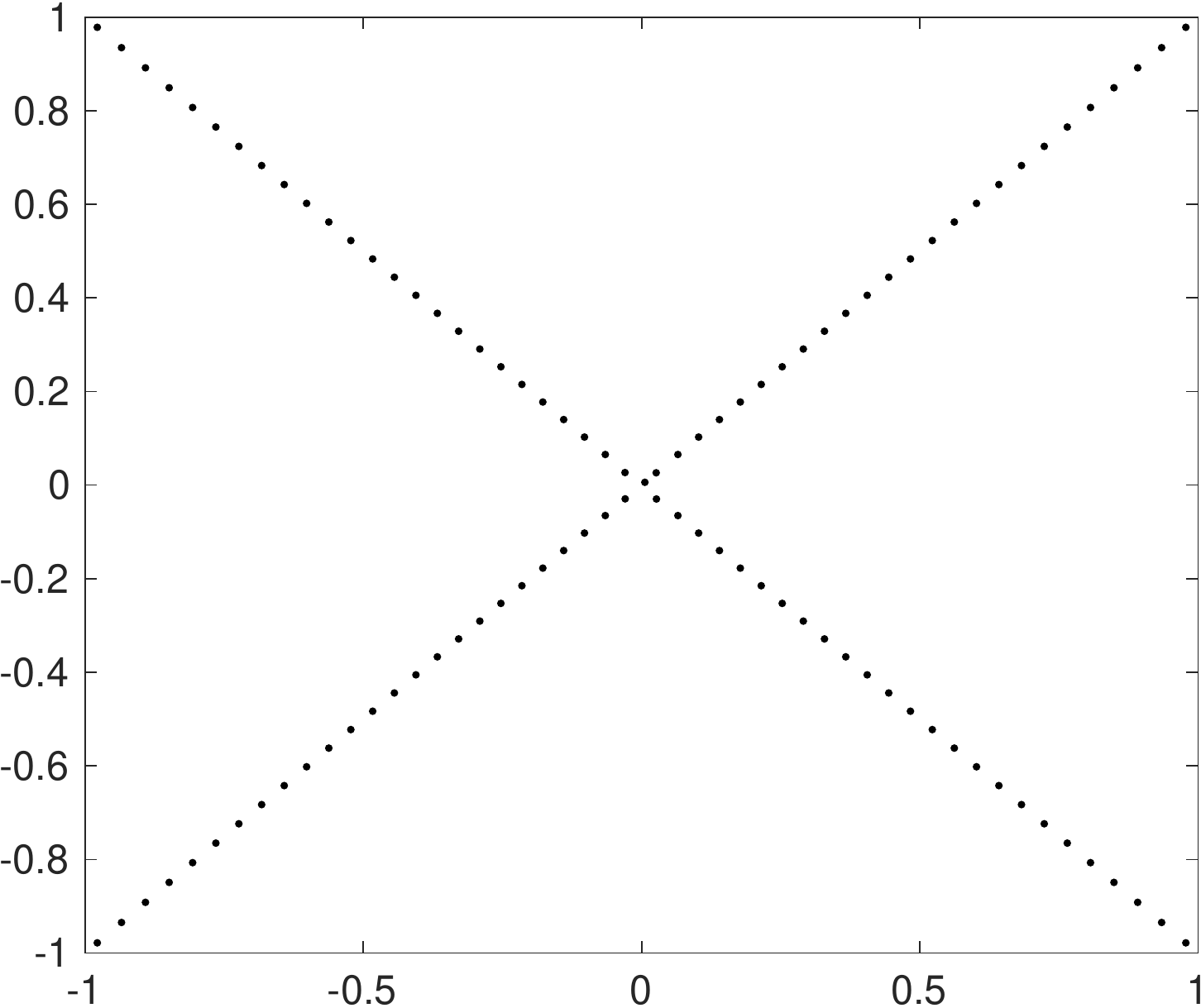}
 \end{minipage}
 \hspace{0cm}
 \begin{minipage}[b]{0.49\linewidth}
  \includegraphics[width=\textwidth]{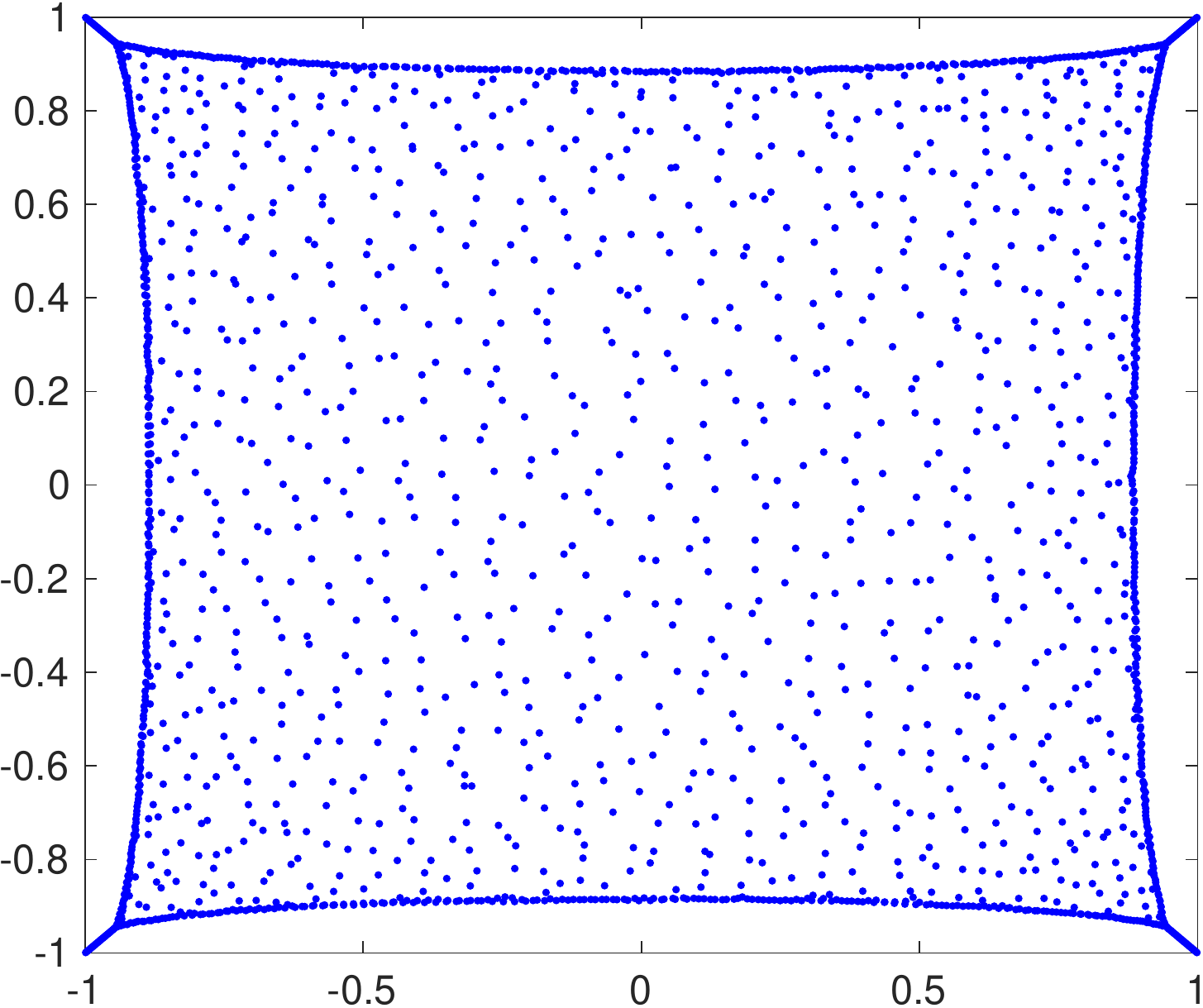}
 \end{minipage}
 \caption{The left hand side shows the spectrum of the unperturbed Scottish flag operator $p_N$ \eqref{ii3.2}, 
 	      and the right hand side shows the spectrum of a $p_N+\delta Q_N$, for $N=1000$, $Q_N$ a complex 
	      Gaussian random matrix and $\delta =10^{-12}$.}
  \label{fig1}
\end{figure}
In the recent paper \cite{ZwChrist10} Christiansen and Zworski established a Weyl law for the 
expected number of eigenvalues of small Gaussian random perturbations of $p_N$. They proved 
\begin{thm}[\cite{ZwChrist10}]
	Suppose that $f\in C^{\infty}(\T^{2d})$, and that $\Omega$ is a simply connected open set 
	with a smooth boundary, $\partial \Omega$, such that for all $z$ in a neighbourhood of 
	$\partial\Omega$, 
	\begin{equation}\label{ii3}
		\mathrm{vol}_{\T^{2d}}(\{w: |f(w)-z|^2\leq t \} = \mO(t^{\kappa}), \quad 0 \leq t \ll 1, 
	\end{equation}
	with $1/2 < \kappa \leq 1$. Let $Q_{N}$ be a complex Gaussian random $N^d\times N^d$-matrix 
	with independent and identically distributed entries $ \sim \mathcal{N}_{\C}(0,1)$. Then for any 
	$p\geq p_0 > d+1/2$ 
	\begin{equation}\label{ii4}
		\erw\left( |\spec(f_N+N^{-p}Q_N)\cap \Omega|\right) 
		= 
		N^d \mathrm{vol}_{\T^{2d}}(f^{-1}(\Omega)) + \mO(N^{d-\beta}),
	\end{equation}
	for any $\beta < (\kappa -1)(\kappa+1)$.
\end{thm}
Let us remark that the original result of \cite{ZwChrist10} is presented with $|f(w)-z|$ in \eqref{ii3} 
instead of $|f(w)-z|^2$, which then leads to $1 < \kappa \leq 2$. We modified the notation to be 
more easily comparable with the results that follow. 
\par
In Theorem \ref{thm1} and \ref{thm2} below we present a stronger result, estimating the 
probability that this asymptotic holds and providing more precise error estimates. Moreover, 
we remove the lower bound on $\kappa$ and 
simply demand it to be $>0$. Furthermore, we allow for a universal probability distribution 
in the perturbation, see Theorem \ref{thm2}. Finally, we 
remark that in our results we allow for coupling constants which may go up to the critical case of $N^{-p}$ 
with $p>d/2$ and down to being sub-exponentially small in $N$.
\\
\par
In \cite{ZwChrist10} the authors state the following 
\begin{con}[\cite{ZwChrist10}]\label{conj1}
	Suppose that \eqref{ii3} holds for all $z\in\C$ with a fixed $0<\kappa \leq 1$. Define random 
	probability measures 
	\begin{equation*}
		\mu_N = N^{-d} \sum_{\lambda \in \spec(f_N+N^{-p}Q_N)} \delta_{\lambda},
	\end{equation*}
	with $p\geq p_0 > d+1/2$. Then, almost surely 
	\begin{equation*}
		\mu_N  \rightharpoonup f_*(\sigma^n/n!),\quad N\to \infty,
	\end{equation*}
	where $\sigma = \sum_1^d d\xi_k\wedge dx_k$, $(x,\xi)\in \T^{2d}$, is the symplectic form 
	in $\T^{2d}$.
\end{con}
We prove this conjecture, see Corollary \ref{cor1} below, for general random matrix ensembles, 
and coupling constants $\delta =N^{-p}$, $p>d/2+1$. When $d/2 +1 \geq p>d/2$ we show that 
the convergence still holds in probability.
\section{Main result}\label{main}
We are interested in the Toeplitz quantization of smooth functions on the $2d$-dimensional torus 
$\T^{2d} = \R^{2d}/\Z^d$. This is related to the more general Berezin-Toeplitz quantization 
of compact symplectic K\"ahler manifolds, see \cite{BoUr03} or for instance \cite{Fl18} for an introduction.  
A \emph{symbol} $p\in C^{\infty}(\T^{2d})$ can be identified 
with a smooth periodic function on $\R^{2d}$. Hence $p$ is in the 
symbol class $S(1)$, i.e. the class of smooth functions $a\in C^{\infty}(\R^{2d})$ such that 
 for any $\beta\in \N^{2d}$ there exists a constant $C_{\beta}>0$ 
such that 
\begin{equation}\label{i1}
		|\partial^{\beta} a(\rho )| \leq C_{\beta}.
\end{equation}
We let $h\in ]0,1]$ denote the semiclassical parameter. A symbol $a\in S(1)$ may 
depend on $h$, in which case we demand that the constants in the estimates \eqref{i1} are uniform  
with respect to $h$. The $h$-\emph{Weyl quantization} of such a symbol $a$ is given by 
the linear operator
\begin{equation*}
	a^w(x,hD_x;h) u (x) \defeq \frac{1}{(2\pi h)^d}
	\iint \e^{\frac{i}{h}\xi(x-y) }a\left(\frac{x+y}{2}, \xi ;h\right) u(y)dyd\xi, 
\end{equation*}
acting on a Schwartz function $u\in \cS(\R^d)$. Here, the integral with respect to $\xi$ is 
to be seen as an oscillatory integral. The operator $a^w$ is a continuous linear map  
$\cS\to \cS$, $\cS'\to \cS'$ and a bounded linear map $L^2\to L^2$, see 
for instance in \cite{DiSj99,Ma02,Zw12}.
\\
\par
We denote by $\mathcal{H}_h^d$ the space of tempered distributions 
$u\in\cS'(\R^d)$ which are $\Z^d$-translation invariant in position and in frequency, 
more precisely 
\begin{equation*}
	u(x+n) = u(x), \quad \mathcal{F}_h(u)(\xi + n) = \mathcal{F}_h(u)(\xi), 
	\quad \forall n \in\Z^d. 
\end{equation*}
Here $\mathcal{F}_h$ denotes the semiclassical Fourier transform, see \eqref{qt2} below 
for a definition. 
The space $\mathcal{H}_h^d$ is $\neq \{0\}$ if and only if $h = 1/(2\pi N)$, for some $\N\ni N>0$, in which 
case $\dim \mathcal{H}_h^d = N^d$, and we can identiy $\mathcal{H}_h^d \simeq \ell^2(\Z^d/ N\Z^d) 
\simeq \C^{N^d}$. 
\\
\par
When $p\in C^{\infty}(\T^{2d})$, possibly $h$ dependent in the above sense, then $p^w(x,hD_x)$ 
maps $\mathcal{H}_h^d$ into itself, and the restriction 
\begin{equation*}
	p_N \defeq p^w(x,hD_x)\!\upharpoonright_{\mathcal{H}_{h}^d } :
	\mathcal{H}_h^d \longrightarrow \mathcal{H}_h^d, \quad h = \frac{1}{2\pi N},
\end{equation*}
defines a quantization 
\begin{equation*}
	C^{\infty}(\T^{2d}) \ni p \mapsto p_N \in \mathcal{L}(\C^{N^d},\C^{N^d}).
\end{equation*}
The matrix elements of $p_N$ (see \eqref{qf8} below for details) are given by 
 \begin{equation*}
	(p_N)_{m,j} = \sum_{n,r\in \Z^d}\widehat{p}(n,j-m-rN) \,\e^{\frac{i\pi}{N}(j+m)\cdot n}(-1)^{n\cdot r}, 
	\quad m,j \in (\Z/N\Z)^d.
\end{equation*}
where $\widehat{p}$ is the Fourier transform of $p$. 
\\
\par
Let $h = 1/(2\pi N)$, $\N\ni N>0$, and suppose that for $p\in C^{\infty}(\T^{2d}) $ 
there exist $ p_{\nu}\in C^{\infty}(\T^{2d}) $, $\nu\in\N$, so that 
\begin{equation}\label{i2.1}
	p(\rho;h) \sim p_0(\rho) + h p_1(\rho) + \dots \text{ in } S(1), 
\end{equation}
meaning that  $p - \sum_0^M h^{\nu}p_\nu \in h^{M+1} S(1)$ for all $M\in \N$. We 
call $p_0$ the \emph{principal symbol} of $p$. 
\\
\par
The aim of this paper is to study the eigenvalue distribution of 
$p_N+ \delta Q_{\omega}$ for $\delta$ in a suitable range and 
for $Q_{\omega}$ in a suitable ensemble of $N^d\times N^d$ random 
matrices. 
\\
\par
Let $\Omega\Subset \C$ be an open relatively compact simply connected set 
with a uniformly Lipschitz boundary $\partial \Omega$, see Section \ref{sec:LipBD} 
below for a precise definition. For $z$ in 
a neighbourhood of $\partial\Omega$ (denoted by $\text{neigh}(\partial \Omega)$) 
and $0\leq  t \ll1 $ we set 
\begin{equation}\label{i2}
		V_z(t) = \mathrm{Vol}\{\rho \in \T^{2d}; |p_0(\rho) -z |^2 \leq t \}.
\end{equation}
We suppose that 
\begin{equation}\label{i3}
		\exists \kappa \in ]0,1], \text{ such that } V_z(t) = \mO(t^{\kappa}), 
		\text{ uniformly for } z \in \text{neigh}(\partial \Omega), ~ 0 \leq t \ll 1.
\end{equation}
The first result concerns the case of a perturbation by a complex Gaussian 
random matrix. 
\begin{thm}\label{thm1}
Let $p\in C^{\infty}(\T^{2d})$ satisfy \eqref{i2.1} and let $N\geq 2$. 
Let $\Omega\Subset \C$ be an open relatively compact simply connected set 
with a uniformly Lipschitz boundary $\partial \Omega$, so that \eqref{i3} holds. 
Let $Q_{\omega}$ be a complex Gaussian random $N^d\times N^d$-matrix with independent 
and identically distributed entries, i.e. 
\begin{equation}\label{t0}
	Q_{\omega} = (q_{i,j}(\omega))_{1\leq i,j\leq N^d}, \quad 
	q_{i,j}(\omega) \sim \mathcal{N}_{\C}(0,1) ~(iid).
\end{equation}
Let $N^{-1} \ll \alpha \ll 1$, let $C>0$ be sufficiently large, and let  
\begin{equation}\label{t1}
\varepsilon \gg  \alpha^{\kappa}\log \frac{CN^{d/2}} {\delta\alpha^2} + \delta N^{\frac{d}{2}}\alpha^{-1/2}, 
	\quad 0<\delta \ll N^{-d/2}\alpha^{1/2}.
\end{equation}
Then, 
\begin{equation*}
		\bigg|\#(\spec(p_N +\delta Q_{\omega})\cap \Omega) - {N^d} \int_{p_0^{-1}(\Omega)} d\rho \bigg| 
		\leq \mO(N^d) \left( \int_{p_0^{-1}(\partial\Omega+D(0,r))}d\rho
		 + \frac{\varepsilon}{r}
		 + r^{\kappa}\right)
\end{equation*}
for $0 < r \ll 1$, with probability 
\begin{equation*}
	\geq 1 -  \mO(r^{-1}) \left( N^{d/2}  \alpha^{-2}\delta^{-1}
	 \exp\left(\alpha^{-\kappa}(2\delta N^{\frac{d}{2}}\alpha^{-1/2} - \varepsilon/C)\right) +\e^{-N^d}\right).
\end{equation*}
\end{thm}
This Weyl law shows that the eigenvalues of the small random perturbations of 
$p_N$ roughly equidistribute in $\Sigma = p_0(\T^{2d})$ the numerical range of 
the principal symbol of the operator $p_N$. This is illustrated in Figure \ref{fig1}.  
In Corollaries \ref{cor2}, \ref{cor2.2}, 
\ref{cor2.1}, \ref{cor3}, below we provide some special cases of Theorem \ref{thm1}. 
\par
A similar result to Theorem \ref{thm1} for 
semiclassical pseudo-differential operators on $\R^{d}$ has been proven by 
Hager \cite{Ha06} and Hager-Sj\"ostrand \cite{HaSj08}. Condition \eqref{i3} appears 
as well in the works of Christiansen-Zworski \cite{ZwChrist10} and Hager-Sj\"ostrand \cite{HaSj08}, 
and is used there, as well as in our work, to control the number of small singular 
values of $(p_N-z)^*(p_N-z)$ for $z$ in a neighbourhood of $\partial\Omega$. 
\par
As observed in \cite{ZwChrist10} 
for real analytic $p$ condition \eqref{i3} always holds for some $\kappa >0$. Similarly, when $p$ is 
real analytic and such that $p(\T^{2d})\subset\C$ has non-empty interior, 
then 
\begin{equation}\label{ke1}
	\forall z\in\partial\Omega: ~
	dp \!\upharpoonright_{p^{-1}(z)} \neq 0 \quad \Longrightarrow \quad 
	\eqref{i2} \text{ holds with } \kappa > 1/2.
\end{equation}
For smooth $p\in C^{\infty}(\T^{2d})$ we have that when for every $z\in\partial\Omega$ 
\begin{equation}\label{ke2}
\begin{split}
	&dp, d\overline{p} \text{ are linearly independent at every point of } p^{-1}(z),  \\
	&\text{then }\eqref{i2} \text{ holds with } \kappa = 1.
\end{split}
\end{equation}
Observe that $dp$ and $d\overline{p}$ are linearly independent at $\rho$ when 
 $\{p,\overline{p}\}(\rho)\neq 0$, 
 where $\{a,b\} = \partial_\xi a \cdot \partial_x b -  \partial_x a \cdot \partial_\xi b$ 
denotes the Poisson bracket on $\T^{2d}$. Observe that in dimension $d=1$ 
the condition $\{p,\overline{p}\}\neq 0$ on $p^{-1}(z)$ is equivalent to 
$dp$, $d\overline{p}$ being linearly independent at every point of $p^{-1}(z)$. 
However, in dimension $d>1$ this cannot in hold general as the integral of 
$\{p,\overline{p}\}$ with respect to the Liouville measure on $p^{-1}(z)$ vanishes 
on every compact connected component of $p^{-1}(z)$, see \cite[Lemma 8.1]{MeSj02}. 
Furthermore, condition \eqref{ke2} cannot hold 
when $z\in \partial\Sigma$, $\Sigma = p_0(\T^{2d})$. However, some iterated Poisson 
bracket may not be zero there. For example, it was observe in \cite[Example 12.1]{HaSj08} 
that 
\begin{equation}\label{ke3}
	\forall \rho \in p^{-1}(\partial \Omega): ~\{p,\overline{p}\}(\rho) \neq 0 \text{ or }
	\{p,\{p,\overline{p}\}\} (\rho) \neq 0, 
	\text{ then }
	\eqref{i2} \text{ holds with } \kappa = \frac{3}{4}.
\end{equation}
The Poisson bracket has another important role in this context: the link to 
\emph{spectral instability}. Indeed 
\begin{equation}\label{ke4}
	\{p,\overline{p}\}(\rho_0) \neq 0, ~ \rho_0 = p(z_0), 
	\quad \text{then } \|(p-z_0)^{-1}\| \geq C_k N^{k}, ~\forall k>0.
\end{equation}
Moreover, the approximate eigenvector $u_N$, with $\|u_N\|_{\ell^2} =1$, also called 
a \emph{quasimode}, causing this resolvent growth at $z_0$ can be microlocalized 
at $\rho_0$, i.e. for any smooth function $\chi$ on $\T^{2d}$ vanishing near $\rho_0$, we have that 
$\| \chi_N u_N\| =\mO(N^{-\infty})$. This has been shown for $\T^2$ by 
Chapman-Trefethen \cite{TrCh04}, see also \cite{TrEm05}, and for the  
Berezin-Toeplitz quantization of compact symplectic K\"ahler manifolds 
by Borthwick-Uribe \cite{BoUr03}. In these works the condition $\{p,\overline{p}\}(\rho_0) \neq 0$ 
is usually referred to as the \emph{twist} resp. \emph{anti-twist condition},  
depending on the sign of $\{\Rea p, \Ima p\}(\rho_0)$. 
The construction of such quasimodes goes back to a classical result of 
H\"ormander concerning the local solvability of partial differential equations, 
see \cite{Zw01} and the references therein. We refer the reader also to 
\cite{Da99b,NSjZw04} for recent results.
\\
\par
In Figure \ref{fig1} we see that the eigenvalues of the perturbed operator 
do not fully reach the boundary of $\Sigma$. It is expected that there, 
under some non-degeneracy condition, such as \eqref{ke4}, we have 
improved resolvent bounds on $\|(p_N -z_0)^{-1}\|$, $z_0\in\partial \Sigma$. 
Indeed, Dencker, Sj\"ostrand and Zworski prove in \cite[Theorem 1.4]{NSjZw04} that 
under the condition that at least some  
$k$-fold iterated Poisson bracket of the real and imaginary part of the principal 
symbol $p_0$ does not vanish on $p^{-1}(z_0)$, $z_0\in\partial \Sigma$,  
the Weyl quantization of the bounded smooth symbol $p\sim p_0 + hp_1+\dots$, 
satisfies the $L^2$ resolvent bound 
\begin{equation*}
	\| (p^w -z_0)^{-1}\| \leq C h^{-\frac{k}{k+1}}.
\end{equation*}
We suspect (but have not investigated further) that the proof of \cite[Theorem 1.4]{NSjZw04} can be carried over to our situation, with $h =1/2\pi N$, which would then give an explanation of the 
absence of eigenvalues of $p_N+\delta Q_{\omega}$ in Figure \ref{fig1} 
in a small $N$-dependent neighbourhood of the boundary $\partial\Sigma$ 
(away from the edge points of the square $\Sigma$).
\\
\par
Before we turn to the case of more general perturbations, let us 
discuss some special cases of Theorem \ref{thm1}. Notice that when 
$\kappa >1/2$ then \eqref{i3} implies that 
\begin{equation}\label{ke6}
	\int_{p_0^{-1}(\partial\Omega+D(0,r))}d\rho = \mO(r^{2\kappa -1}).
\end{equation}
One can easily see that $r = \varepsilon^{\frac{1}{2\kappa}}$ minimizes (up to a constant) 
the error term in Theorem \ref{thm1}, and it becomes 
\begin{equation*}
	\mO(N^d \varepsilon^{\frac{2\kappa -1}{2\kappa}}).
\end{equation*}
Taking $\alpha=CN^{-1}$ and $\varepsilon = C_0 N^{-\kappa}(\log N )^2$, 
for some sufficiently large constants $C,C_0>1$, one obtains from Theorem \ref{thm1} 
the following 
\begin{cor}\label{cor2}
Under the assumptions of Theorem \ref{thm1}, we let $\kappa \in ]1/2,1]$ and 
set for $p \geq (d+1)/2+\kappa$  
\begin{equation*}
	\delta = \frac{1}{C} N^{-p},
\end{equation*}
for some sufficiently large $C>0$. Then, 
\begin{equation*}
		\bigg|\#(\spec(p_N +\delta Q_{\omega})\cap \Omega) - {N^d} \int_{p_0^{-1}(\Omega)} d\rho \bigg| 
		\leq \mO\!\left(N^{d-\kappa+1/2} (\log N)^{(2\kappa-1)/\kappa}\right)
\end{equation*}
with probability $\geq 1 -  \mO(N^{-\infty})$.
\end{cor}
Notice that at the price of increasing the error term of the eigenvalue counting 
estimate by a factor $N^{\beta}$, with $\beta \in]0,1[ $, one can obtain the above result 
with probability $\geq 1 - \e^{-N^{\beta'}/C}$, for some $\beta' \in ]0,1[$.
\\
\par
For $p_0\in]0,\kappa]$, we set $\alpha=CN^{-p_0/\kappa}$ and $\varepsilon = C_0 N^{-p_0}(\log N )^2$, 
for some sufficiently large constants $C,C_0>1$. Then, one gets from Theorem \ref{thm1} 
the following 
\begin{cor}\label{cor2.2}
Under the assumptions of Theorem \ref{thm1}, we let $\kappa \in ]1/2,1]$ and for $p_0 \in ]0,\kappa]$  
set 
\begin{equation*}
	\delta = \frac{1}{C} N^{-(d+1)/2-p_0},
\end{equation*}
for some sufficiently large $C>0$. Then, 
\begin{equation*}
		\bigg|\#(\spec(p_N +\delta Q_{\omega})\cap \Omega) - {N^d} \int_{p_0^{-1}(\Omega)} d\rho \bigg| 
		\leq \mO\!\left(N^{d-p_0(2\kappa-1)/(2\kappa)} (\log N)^{(2\kappa-1)/\kappa}\right)
\end{equation*}
with probability $\geq 1 -  \mO(N^{-\infty})$.
\end{cor}
Taking $\alpha=CN^{-1}$ and $\varepsilon = C_0 N^{\beta-\kappa}$, 
for some sufficiently large constants $C,C_0>1$, and $\beta\in]0,\kappa[$, one
 obtains form Theorem \ref{thm1} the following 
\begin{cor}\label{cor2.1}
Under the assumptions of Theorem \ref{thm1}, we let $\kappa \in ]1/2,1]$ and $\beta\in]0,\kappa[$. 
Set
\begin{equation*}
	\delta = \e^{-N^{\beta}},
\end{equation*}
then, 
\begin{equation*}
		\bigg|\#(\spec(p_N +\delta Q_{\omega})\cap \Omega) - {N^d} \int_{p_0^{-1}(\Omega)} d\rho \bigg| 
		\leq \mO\!\left(N^{d+(\beta-\kappa)(2\kappa-1)/(2\kappa)} \right)
\end{equation*}
with probability 
\begin{equation*}
	\geq 1 -  \mO\!\left(\e^{-N^{\beta}/C}\right).
\end{equation*}
\end{cor}
When $0 < \kappa \leq 1/2$, then, without any additional 
assumptions on the behaviour of the integral \eqref{ke6}, we only know that  
\begin{equation*}
	\int_{p_0^{-1}(\partial\Omega+D(0,r))}d\rho= o(1), \quad r \to 0,
\end{equation*}
since by the Morse-Sard theorem $p_0^{-1}(\partial\Omega)$ has Lebesgue 
measure $0$, so the regularity of the Lebesgue measure shows the 
above convergence. In this situation, the best we can have is an 
error term of order 
\begin{equation*}
	o(N^d).
\end{equation*}
Similarly to Corollary \ref{cor2}, \ref{cor2.1}, we get that 
\begin{cor}\label{cor3}
Under the assumptions of Theorem \ref{thm1}, we let $\kappa \in ]0,1]$ and for $p > d/2$  
set
\begin{equation*}
	\delta = \frac{1}{C} N^{-p},
\end{equation*}
for some sufficiently large $C>0$. Then, 
\begin{equation}\label{ee1}
		\bigg|\#(\spec(p_N +\delta Q_{\omega})\cap \Omega) - {N^d} \int_{p_0^{-1}(\Omega)} d\rho \bigg| 
		\leq o(N^d)
\end{equation}
with probability $\geq 1 -  \mO(N^{-\infty})$. Moreover, when 
\begin{equation*}
	\delta = \e^{-N^{\beta}}, \quad  \text{for some } \beta\in]0,\kappa[,
\end{equation*}
then \eqref{ee1} holds with probability $\geq 1 -  \mO(\e^{-N^{\beta}/C})$.
\end{cor}

The next result concerns the case of a perturbation by an iid matrix. 
\begin{thm}\label{thm2}
Let $p \in C^{\infty}(\T^{2d})$ satisfy \eqref{i2.1}.  
Let $\Omega\Subset \C$ be an open relatively compact simply connected set 
with a uniformly Lipschitz boundary $\partial \Omega$, so that \eqref{i3} holds. 
Let $Q_{\omega}$ be a random $N^d\times N^d$-matrix whose entries are independent 
copies of a random variable $q$ satisfying
\begin{equation*}
	\erw[ q ] =0, \quad  \erw[|q|^2] =1, \quad \erw[ |q|^4] < +\infty.
\end{equation*}
For $\delta_0 >0$ and some sufficiently large $C>0$, let 
\begin{equation*}
	\delta  = \frac{1}{C} N^{ -d/2 - \delta_0 },
\end{equation*}
and for $\tau \in [0,1[$, set 
\begin{equation*}
		\varepsilon =  N^{-\min(\delta_0,1)\tau \kappa}\log N+N^{-\tau \delta_0 /2},
\end{equation*}
Then, 
\begin{equation*}
	\bigg|\#(\spec(p_N+\delta Q_{\omega})\cap \Omega) - {N^d} \int_{p_0^{-1}(\Omega)} d\rho \bigg| 
	    \leq \mO(N^d) \left(
		\int_{p_0^{-1}(\partial\Omega+D(0,r))}d\rho+ \frac{\varepsilon}{r} +r^{\kappa}
		\right),
\end{equation*}
for $0 < r \ll 1$, with probability 
\begin{equation*}
	\geq  1 -  \mO(r^{-1})N^{-( 1-\tau) \delta_0}.
\end{equation*}
\end{thm}
Similarly to Corollaries \ref{cor2}, \ref{cor2.2}, \ref{cor2.1}, \ref{cor3}, one can use 
Theorem \ref{thm2} to get precise error estimates in the various situations. %
\\
\par
As a consequence of Theorem \ref{thm2}, or of Theorem \ref{thm1} in the Gaussian 
case, we obtain the following result providing a positive response to Conjecture 
\ref{conj1} by \cite{ZwChrist10}. 
\begin{cor}\label{cor1} 
Let $Q_{\omega}$ and $\delta>0$ be as in Theorem \ref{thm2} and 
assume that \eqref{i3} holds uniformly for all $z\in \C$. Set
\begin{equation*}
		\mu_N = N^{-d} \sum_{\lambda \in \spec(p_N +{\delta}Q_{\omega})} \delta_{\lambda}.
\end{equation*}
Then, for $\delta_0 >1$
\begin{equation*}
		\mu_N  \rightharpoonup (p_0)_*(d\rho), \quad \text{almost surely},
\end{equation*}
and for $ \delta_0 \in ]0,1]$, 
\begin{equation*}
		\mu_N  \rightharpoonup (p_0)_*(d\rho), \quad \text{in probability.}
\end{equation*}
\end{cor}
We remark than in the case of $\T^{2d}$ the measure induced by the symplectic volume form 
$\sigma^n/n!$ given in Conjecture \ref{conj1} is equal to the Lebesgue measure $d\rho$ on $\T^{2d}$. 
\subsection{Related results}
The case of Toeplitz matrices given by symbols on $\T^2$ of the form $\sum_{n\in\Z} a_n\e^{in\xi}$ 
was studied in a series of recent works by Davies and Hager \cite{DaHa09}, Guionnet, Wood and Zeitouni 
\cite{GuMaZe14}, Basak, Paquette and Zeitouni \cite{BPZ18, BPZ18b}, Sj\"ostrand and the author 
of the current paper \cite{SjVo19a,SjVo19b}. Such symbols amount to the case of symbols which are 
constant in the $x$ variable. In these works the non-selfadjointness of the problem does however not 
come from the symbol itself but from boundary conditions destroying the periodicity of the symbol 
in $x$ by allowing for a discontinuity. Nevertheless, these works show that by adding some small random 
noise the limit of the empirical eigenvalues measure $\mu_N$ of the perturbed operator converges 
in probability (or even almost surely in some cases) to $p_*(d\rho)$. 
\par
In \cite{BPZ18} the authors treated in particular the special case of upper triangular banded 
twisted Toeplitz matrices given by symbols of the form 
\begin{equation*}
	\widetilde{p}(x,\xi) =  \sum_{n=0}^{N_+}f_n(2\pi x)\e^{-2\pi i n \xi}, \quad (x,\xi)\in\T^2
\end{equation*}
where $f_n$ is only assumed to be a H\"older continuous function and can have a 
discontinuity. They showed through quite different methods from ours that 
the $\mu_N$ converges weakly in probability to the measure 
\begin{equation*}
		\widetilde{\mu} = \widetilde{p}_*(d\rho).
\end{equation*}
Thus we recover this result of \cite{BPZ18} (at least in the smooth periodic setting) with 
Corollary \ref{cor1}. This suggests that the results of Corollary \ref{cor1} also 
hold in the case of general twisted Toeplitz matrices with band entries defined by $C^1$ 
functions which are defined on a compact interval with non-periodic boundary conditions. 
\subsection{Outline of this paper}
In Section \ref{sec:SC} we recall some fundamental notions and results from standard semiclassical 
calculus needed in this paper. We use this to describe the procedure to quantize 
complex-valued function on dilated tori $\T^{2d}_{\alpha}$, generalizing the 
quantization approach presented in \cite{ZwChrist10,NoZw07}.
\par
In Section \ref{sec:FC} we build on the theory developed in Section \ref{sec:SC} 
and present some functional calculus and estimates on the number of small 
singular values of Toeplitz quantizations, adapting the approach of \cite{HaSj08}.
\par
In Section \ref{sec:GP} we set up a Grushin problem providing upper bounds on 
the log determinant of the perturbed operator  $\log |\det(p_N+\delta Q_{\omega}-z)|$. 
\par
In Section \ref{sec:PRM} we provide probabilistic lower bounds on the number of small 
singular values of the perturbed operator $(p_N+\delta Q_{\omega}-z)$ which yields   
probabilistic lower bounds on $\log|\det(p_N+\delta Q_{\omega}-z)|$. This, 
together with the estimates of Section \ref{sec:GP}, 
can also be seen as a form of concentration inequality for $\log|\det(p_N+\delta Q_{\omega}-z)|$. 
\par
In Section \ref{sec:c1} we provide proofs of Theorems \ref{thm1} and 
\ref{thm2} by combining the estimates on $\log|\det(p_N+\delta Q_{\omega}-z)|$ 
with a theorem providing estimates on the number of zeros of holomorphic functions 
with exponential growth in Lipschitz domains. 
\par
In Section \ref{sec:WC}, we provide a proof of Corollary \ref{cor1} and in Appendix 
\ref{app1} we prove a complex version of a result due to Sankar, Spielmann and 
Teng \cite[Lemma 3.2]{SaTeSp06}, see also \cite[Theorem 2.2]{TaVu10}.
\subsection{Notation}
In this paper we frequently use the following notation: when we write $a \ll b$, we mean that 
$Ca \leq  b$ for some sufficiently large constant $C>0$. The notation $f = \mO(N)$ 
means that there exists a constant $C>0$ (independent of $N$) such that $|f| \leq C N$. 
When we want to emphasize that the constant $C>0$ depends on some parameter 
$k$, then we write $C_k$, or with the above big-O notation $\mO_k(N)$. 
\par
When we write $f= \mO(N^{-\infty})$, then we mean that for every $M\in\N$, there exists 
a constant $C_M>0$, depending on $M$, such that $|f| \leq C_M N^{-M}$. Similarly we will  
also use the notation $f=\mO(h^{\infty})$, with $h\in]0,1]$.
\par
When we write $f= o(1)$, as $N\to \infty$, then we mean that $f \to 0$ as $N\to \infty$. 
Similarly  $f= o(N)$, as $N\to \infty$, means that $N^{-1}f \to 0$ as $N\to \infty$. 
\\
\\
\paragraph{\textbf{Acknowledgments}} 
The author is very grateful to Maciej Zworski for suggesting this project, to L\'aszl\'o Erd\H{o}s for 
some very enlightening discussions and to the Institute of Science and Technology, Austria, where 
a part of this paper has been written, for providing a welcoming and stimulating environment. 
The author was supported by a CNRS Momentum grant. 
\section{Semiclassical calculus}\label{sec:SC}
In this section we begin by reviewing some basic notions and properties of semiclassical calculus in $\R^d$, as can be found for instance in \cite{DiSj99,Ma02,Zw12}. Afterwards, we will review the Toeplitz quantizaton of functions on the torus, as presented in \cite{ZwChrist10,NoZw07}, which roughly speaking consist in restricting the semiclassical calculus to periodic symbols and to function spaces given by tempered distributions which are both periodic in space and in semiclassical frequency. 
\subsection{Semiclassical quantization}\label{sec:SQ}
Until further notice we let $h\in]0,1]$ denote the semiclassical parameter. 
We call $m\in C^{\infty}(\R^{2d},]0,\infty[)$ an \emph{order function}, if there 
exist $C_0,N_0>0$ such that 
\begin{equation}\label{sc0}
	m(\rho) \leq C_0 \langle \rho - \mu \rangle^{N_0} m(\mu),
\end{equation}
where $\langle \rho - \mu \rangle = (1 + |\rho - \mu|^2)^{1/2}$. We define the symbol class 
\begin{equation}\label{sc0.1}
	S(m) \defeq \{ p\in C^{\infty}(\R^{2d}); ~
	 \forall \alpha\in \N^{2d} :|\partial^{\alpha} p(\rho)| \leq C_{\alpha} m(\rho)\}.
\end{equation}
A symbol $p=p(\rho; h)\in S(m)$ may depend on $h\in ]0,1]$, in which case we assume that the symbol  
estimates \eqref{sc0.1} hold uniformly with respect to $h$. If a symbol $p$ is of the form 
$p(\rho;h) = p_0(\rho) + hr(\rho;h)$
with $r\in S(m)$, then we call $p_0$ the \emph{principal symbol} of $p$. We say that $p\in S(m)$ 
has the asymptotic expansion 
\begin{equation}\label{sc1}
	p \sim p_0 + h p_1 + \dots \text{ in } S(m), \quad p_j \in S(m),
\end{equation}
when $ p - \sum_0^N h^j p_j \in h^{N+1} S(m)$ for any $N\in \N$. In fact, given symbols $p_j\in S(m)$ 
we can always find a symbol $p\in S(m)$ by Borel summation, such that \eqref{sc1} holds. 
\\
\par
The $h$-Weyl quantization of a symbol $p \in S(m) $, acting on a function $u\in \mathcal{S}(\R^d)$ in Schwartz space, is given by 
\begin{equation}\label{sc1.1}
	p^w(x,hD_x)u(x) = \frac{1}{(2\pi h)^d} \iint_{\R^d} \e^{\frac{i}{h}\xi(x-y)} p\left( \frac{x+y}{2},\xi;h\right) u(y) dy d\xi,
\end{equation}
where the integral with respect to $\xi$ is to be seen as an oscillatory integral. Integration by parts shows 
that that the operator $p^w(x,hD_x)$ is continuous $ \mathcal{S}(\R^d)\to \mathcal{S}(\R^d) $, and 
continuous $\mathcal{S}'(\R^d)\to \mathcal{S}'(\R^d)$ by duality. Moreover, it can be shown that when 
$m$ is bounded then $p^w(x,hD_x)$ is bounded $L^2(\R^d)\to L^2(\R^d)$. 
\\
\par
Given $a,b\in S(m)$, then 
\begin{equation}\label{sc2}
	a^w \circ b^w = c^w, \quad \text{where } c = a\#_h b ~\in S(m).
\end{equation}
Here, the product $\#_h$ is the bilinear continuous map
\begin{equation}\label{sc3}
	S(m_1)\times S(m_2) \ni (a,b) \mapsto a\#_h b 
	\defeq \e^{\frac{ih}{2} \sigma(D_x,D_{\xi};D_y,D_{\eta})} a(x,\xi) b(y,\eta) \big|_{y=x,\eta=\xi}~
	\in S(m_1m_2),
\end{equation}
where $\sigma$ denotes the symplectic form on $\R^{2d}$, so 
$\sigma(D_x,D_{\xi};D_y,D_{\eta})= D_{\xi}D_y - D_xD_{\eta}$. 
We have the following asymptotic expansion 
\begin{equation}\label{sc4}
	a\#_h b \sim \sum_0^{\infty} \frac{1}{k!} \left( 
	\frac{ih}{2}\sigma(D_x,D_{\xi};D_y,D_{\eta})
	\right)^k a(x,\xi) b(y,\eta) \big|_{y=x,\eta=\xi} \in S(m_1m_2).
\end{equation}
\par
A symbol $a\in S(m)$ is called \emph{elliptic} if there 
exists a constant $C>0$, independently of $h>0$, such that 
\begin{equation}\label{sc5}
	|a| \geq \frac{1}{C} m.
\end{equation}
\subsection{Quantization of the torus}
We essentially follow the approach of \cite{ZwChrist10,NoZw07} who considered the 
case $\alpha=1$ of the subsequent. For $\alpha >0$, we define the torus
\begin{equation}\label{qt1}
	\T^{2d}_{\alpha} := \R^{2d}/ (\alpha^{-1/2}\Z)^{2d}.
\end{equation}
When $\alpha=1$ then we will write $\T^{2d}=\T^{2d}_{1}$.  
We define the semiclassical Fourier transform by 
\begin{equation}\label{qt2}
	(\mathcal{F}_hu)(\xi) \defeq \frac{1}{(2\pi h)^{d/2}} \int \e^{\frac{i}{h}\xi x }u(x) dx, \quad u \in \mathcal{S}(\R^d), 
\end{equation}
which maps $\cS \to \cS$ continuously and can be extended to a continuous map $\cS' \to \cS'$, 
mapping $L^2\to L^2$ unitarily.
\par
Let $h \ll \alpha \leq 1$ and set $\widetilde{h} = h/\alpha$. We define the space 
$\mathcal{H}_{\tth,\alpha}^d \subset \cS'(\R^d)$ of tempered distributions $u\in\cS'$ which 
are both $\alpha^{-1/2}\Z^d$-periodic in position and in frequency, i.e. 
\begin{equation}\label{qt2.1}
	u(x+ \alpha^{-1/2}n) = u(x), \quad (\mathcal{F}_{\tth}u)(\xi+ \alpha^{-1/2}n) =  (\mathcal{F}_{\tth}u)(\xi), 
	\quad \forall n\in \Z^d.
\end{equation}
When $\alpha=1$ we will simply write $\mathcal{H}_{h}^d = \mathcal{H}_{\tth,1}^d $. The following result 
was stated in the case $\alpha = 1 $ in \cite{ZwChrist10,NoZw07}.
\begin{lem}\label{lem:qt1}
	Let $h \ll \alpha \leq 1$. Then, $\mathcal{H}_{\tth,\alpha}^d \neq \{0\}$ if and only if $h = (2\pi N)^{-1}$ for some $N\in\N^*=\N\backslash\{0\}$, 
	in which case $\dim \mathcal{H}_{\tth,\alpha}^d = N^d$ and 
	\begin{equation}\label{qt3}
		 \mathcal{H}_{\tth,\alpha}^d 
		 = \mathrm{span}\left\{ Q_k^{\alpha}=\frac{1}{(\alpha^{1/2}N)^{d/2}} \sum_{n\in\Z^d} 
		 \delta( x - \alpha^{-1/2}( n + k/N)); k\in (\Z/ N\Z)^d\right\}.
	\end{equation}
\end{lem}
\begin{proof} To ease the notation we write $\widehat{u}^{\tth}= (2\pi \tth)^{d/2}\mathcal{F}_{\tth}u$, $u \in \cS'$. 
Recall the Poisson summation formula (for instance from \cite[Section 7.2]{Ho83}) 
\begin{equation}\label{qt4}
		 \sum_{g\in\Z^d}\widehat{\psi}^{\tth}(g) = 
		 \left(\frac{2\pi \tth }{a}\right)^d \sum_{g\in\Z^d}\psi\left(\frac{2\pi \tth }{a} g\right), \quad \psi \in \cS, ~ 0\neq a \in \R.
\end{equation}
Let $\phi \in C^{\infty}_c(\R^d;\R)$ be such that $\sum_{g\in\Z^d} \phi(x-g) =1$. Suppose that $u\in\cS'$ is 
$\alpha^{-1/2}\Z^d$-periodic in position, as in \eqref{qt2.1}. Then for $\psi\in\cS$
\begin{equation}\label{qt5}
	\begin{split}
		\langle \widehat{u}^{\tth} , \psi \rangle 
		= \langle u , \widehat{\psi}^{\tth}\rangle
		&= \langle u , \sum _g\widehat{\psi}^{\tth}(\cdot + g)\phi  \rangle \\
		& = (2\pi \tth \alpha^{1/2})^d \langle u , \sum _g\psi(2\pi \tth \alpha^{1/2}g) \e^{-2\pi i \alpha^{1/2}g \cdot}\phi  \rangle,
	\end{split}
\end{equation}
where the "$\cdot$" holds the place of the variable in which the distribution acts. In the last equation we applied 
\eqref{qt4} with $a = \alpha^{-1/2}$ to $\psi(y)\e^{-\frac{i}{\tth} yx}$, whose $\tth$-Fourier transform $y\to\xi$ is 
given by $\widehat{\psi}^{\tth}(\xi + x)$. Hence, 
\begin{equation}\label{qt6}
		\begin{split}
		&\widehat{u}^{\tth} = (2\pi \tth \alpha^{1/2})^d\sum _g c_g\delta_{2\pi \tth \alpha^{1/2} g }, \\
		& c_g = \langle u ,  \e^{-2\pi i\alpha^{1/2}g \cdot} \phi \rangle = \langle u, \e^{-2\pi i \alpha^{1/2}g \cdot} \rangle_{\T^d_{\alpha}},
		\end{split}
\end{equation}
where in the last line we see $u\in \mathcal{D}'(\T^d_{\alpha})$ as a distribution 
on $\T^{2d}_{\alpha}$, and $\delta_{x_0}=\delta(x-x_0)$ 
denotes the Dirac measure at $x_0$. 
\par
Since $\widehat{u}^{\tth}$ is $\alpha^{-1/2}\Z^d$-translation invariant by \eqref{qt2.1}, it follows that 
\begin{equation}\label{qt7}
		\widehat{u}^{\tth} = \delta_{\alpha^{-1/2}g} *\widehat{u}^{\tth}, \quad \forall g\in\Z^d,
\end{equation}
where $*$ denotes the convolution. Hence, the condition $h= (2\pi N)^{-1}$ for some $N\in \N^*$ is necessary 
for $u \neq 0$, as the supports of both sides of \eqref{qt7} have to match which is equivalent to the condition 
$\Z^d = (2\pi \tth \alpha)^{-1} \Z^d + \Z^d$. 
\par
On the other hand,  suppose that $h= (2\pi N)^{-1}$ for some $N\in \N^*$. The translation invariance \eqref{qt2.1} 
of  $\widehat{u}^{\tth}$ implies that 
$u(x) = \e^{-\frac{i}{\tth}\alpha^{-1/2}nx}u(x)$, $n\in\Z^d$. From \eqref{qt6} we then get $c_g = c_{g + Nl}$ 
for all $l\in\Z^d$. Hence, 
 \begin{equation}\label{qt8}
 \begin{split}
		\widehat{u}^{\tth} &= (2\pi \tth \alpha^{1/2})^d\sum _{j\in(\Z/ N\Z)^d} c_{j} 
		\sum_{l\in\Z^d}\delta_{\alpha^{-1/2}(l + N^{-1}j)} \\ 
		& = (2\pi \tth \alpha^{1/2})^d \sum _{j\in(\Z/ N\Z)^d} c_{j} 
		\left(\delta_{\alpha^{-1/2}N^{-1}j}  * \sum_{l\in\Z^d}\delta_{l\alpha^{-1/2}} \right)
\end{split}
\end{equation}
Notice that when $0\neq a \in \R$ and $u_a = \sum_{g\in\Z^d} \delta_{ag}$, then 
$\widehat{u}^{\tth}_a = (2\pi \tth/a)^d u_{2\pi \tth /a}$, see for instance 
\cite[Theorem 7.2.1]{Ho83}. This, together with 
Fourier inversion formula yields 
 \begin{equation}\label{qt9}
 \begin{split}
		u(x) & = N^{-d}\sum _{j\in(\Z/ N\Z)^d} c_{j} 
		\sum_{l\in\Z^d}\exp\left(\frac{i}{\tth} N^{-1} \alpha^{-1/2} j\cdot x \right) 
		\delta(x-2\pi \tth \alpha^{1/2} l).\\
		& = N^{-d}\sum _{j , k \in(\Z/ N\Z)^d} c_{j} \exp\left(2\pi i N^{-1} j\cdot k \right) 
		\sum_{l'\in\Z^d}
		\delta(x-\alpha^{-1/2}(l' + N^{-1}k)).
\end{split}
\end{equation}
Thus, the condition of the lemma is also sufficient and \eqref{qt3} follows as well. 
\end{proof}
Finally, we remark that the Fourier transform $\mathcal{F}_{\tth}$ maps 
$ \mathcal{H}_{\tth,\alpha}^d $ into 
$ \mathcal{H}_{\tth,\alpha}^d $, and can be represented, in the basis \eqref{qt3}, by 
 \begin{equation}\label{qt10}
	(\mathcal{F}_{\tth})_{nm} = \frac{\e^{-\frac{2\pi i}{N} n\cdot m}}{N^{d/2}}, \quad n,m \in (\Z/ N\Z)^d.
\end{equation}
\subsection{Quantizing functions on the torus}\label{sec:qf}
For $h \ll\alpha\leq  1$ as before, we define an order function on $\T^{2d}_{\alpha}$ as follows: let $m\in C^{\infty}(\T^{2d}_{\alpha}; ]0,\infty[)$ be 
such that there exist constant $C_0,N_0 >0$, independently of $\alpha$, such that 
 \begin{equation}\label{qf1}
 \begin{split}
		m(\rho) &\leq C_0 (1 + |\rho-\mu|_{\T^{2d}_{\alpha}}^2)^{N_0/2} m(\mu), \\ 
			    &\defeq C_0 \langle\rho-\mu\rangle_{\T^{2d}_{\alpha}}^{N_0} m(\mu)
		 \quad \forall \rho,\mu \in \T^{2d}_{\alpha},
\end{split}
\end{equation}
where $ |\rho-\mu|_{\T^{2d}_{\alpha}}:= \inf_{\gamma \in \alpha^{-1/2}\Z^{2d}} | \rho - \mu +\gamma|$. 
When seeing $m$ as a periodic function in $C^{\infty}(\R^{2d}; ]0,\infty[)$, it follows by using the natural 
projection $\R^{2d} \to \T^{2d}_{\alpha}$ that $m$ is an order function on $\R^{2d}$ in the sense of \eqref{sc0}. 
\\
\par
With this notion of order function we define the symbol class
 \begin{equation}\label{qf2}
		S(m,\alpha) \defeq \{ a \in C^{\infty}(\T^{2d}_{\alpha});~
		 \forall \beta \in \N^{2d}~ \exists\,C_{\beta}>0 : ~| \partial^{\beta}_{\rho}a(\rho)| \leq C_{\beta} m(\rho) \},
\end{equation}
where the constant $C_{\beta}>0$ is independent of $\alpha>0$.
\\
\par
Identifying a symbol in $S(m,\alpha)$ with a periodic function in $C^{\infty}(\R^{2d})$, 
we see that $S(m,\alpha) \subset S(m) $. We will use this identification frequently in the 
sequel whenever convenient. Hence, the quantization procedure discussed in 
Section \ref{sec:SQ} applies to symbols $p\in S(m,\alpha)$ and it follows immediately 
from \eqref{sc1.1} by conjugation with the unitary operator $\tau_{\gamma} \e^{i x\cdot \mu/\tth }$, 
$\gamma,\mu\in\alpha^{-1/2}\Z^d$, where $\tau_{\gamma}u(x) = u(x-\gamma)$, that 
 \begin{equation}\label{qf3}
			p^w(x,\tth D_x ) : \mathcal{H}_{\tth,\alpha}^d  \longrightarrow \mathcal{H}_{\tth,\alpha}^d, 
\end{equation}
where $ \mathcal{H}_{\tth,\alpha}^d\subset \cS'$. 
\par
Therefore, we define for $h=\frac{1}{2\pi N}$, $0<N\in \N$, $\tth = h/\alpha$
 \begin{equation}\label{qf4}
		p_{N,\alpha} \defeq p^w(x,\tth D_x )\!\upharpoonright_{\mathcal{H}_{\tth,\alpha}^d }
		~\in \mathcal{L}(\mathcal{H}_{\tth,\alpha}^d ,\mathcal{H}_{\tth,\alpha}^d )
\end{equation}
When $\alpha = 1 $ we will simply write $p_N = p_{N,1}$. Notice 
that $1_{N,\alpha} = \mathrm{Id}_{\mathcal{H}_{\tth,\alpha}^d }$.
\\
\par
It follows from \eqref{sc3} that $a\#_{\tth}b$ is periodic when $a,b\in S(m)$ are periodic. Hence, 
the composition formula \eqref{sc2} applies to symbols $(a,b) \in S(m_1,\alpha)\times S(m_2,\alpha)$ 
and we get
\begin{equation}\label{qf5}
		a_N \circ b_N = c_N, \quad \text{with } c = a\#_{\tth}b \in  S(m_1m_2,\alpha), \quad h = \frac{1}{2\pi N}, \tth = \frac{h}{\alpha}.
\end{equation}
The following result determining the Hilbert space structure of $\mathcal{H}_{\tth,\alpha}^d$ was stated 
in the case $\alpha=1 $ in \cite{ZwChrist10}.
\begin{lem}\label{lem:qf1}
There exists a unique (up to a multiplicative constant) Hilbert space structure 
on $\mathcal{H}_{\tth,\alpha}^d$ for which all 
$f_{N,\alpha} :  \mathcal{H}_{\tth,\alpha}^d  \to \mathcal{H}_{\tth,\alpha}^d$ 
with $f\in C^{\infty}(\T^{2d}_{\alpha};\R)$ are self-adjoint. One can choose the 
constant so that the basis 	\eqref{qt3} is orthonormal. This implies that the Fourier 
transform on $ \mathcal{H}_{\tth,\alpha}^d$ \eqref{qt10} is unitary. 
\end{lem}
\begin{proof}
	We can essentially follow the proof of \cite[Lemma 2.4]{ZwChrist10}, which we present here in an adapted version 
	for the readers convenience. Let $(\bullet| \bullet)_0$ denote the scalar product 
	on $\mathcal{H}_{\tth,\alpha}^d$ for which the basis $\{Q_j^{\alpha}\}_{j\in (\Z/N\Z)^d}$ \eqref{qt3} is orthonormal. 
	We write the operator $f^w(x,\tth D_x)$ on $\mathcal{H}_{\tth,\alpha}^d$ explicitly in that basis using the Fourier 
	expansion of $f$: 
	\begin{equation}\label{qf6}
		\begin{split}
			&f(x,\xi) =  \sum_{n,m\in\Z^d} \widehat{f}(n,m) \,\e^{2\pi i \alpha^{1/2} ( x\cdot n + \xi\cdot m)}, \\
			&\widehat{f}(n,m) = \alpha^d \int_{\T^{2d}_{\alpha}} f(x,\xi) \,\e^{-2\pi i \alpha^{1/2} ( x\cdot n + \xi\cdot m)} dxd\xi.
		\end{split}
	\end{equation}
	Integration by parts shows that 
	\begin{equation}\label{qf6.1}
		\widehat{f}(n,m) = \mO_k(1)\langle(n,m)\rangle ^{-k} \alpha^{d-k/2} \sum_{|\beta|\leq k}
		\| \partial^{\beta}f \|_{L^1(\T^{2d}_{\alpha})}.
	\end{equation}
	Write $L_{n,m}(x,\xi) = n\cdot x + m\cdot \xi$, so that 
	\begin{equation}\label{qf7}
		f^w(x,\tth D_x) =  \sum_{n,m\in\Z^d} \widehat{f}(n,m) \,\e^{2\pi i \alpha^{1/2} L_{n,m}^w(x,\tth D_x) }.
	\end{equation}
	Since $\e^{2\pi i \alpha^{1/2} L_{n,m}^w(x,\tth D_x) } = \e^{i\pi \alpha^{1/2}n\cdot x }\tau_{-2\pi \tth \alpha^{1/2}m} \e^{i\pi \alpha^{1/2}n \cdot x}$, 
	we check directly that 
	\begin{equation*}
		\e^{2\pi i \alpha^{1/2} L_{n,m}^w(x,\tth D_x) } Q^{\alpha}_j =
		\e^{\frac{i\pi}{N}(2n\cdot j-n\cdot m)}Q^{\alpha}_{j-m},
	\end{equation*}
	where $j-m$ is meant mod $N$. Here $\tau_au(x) = u(x-a)$ is the shift operator. 
	Consequently, 
	\begin{equation}\label{qf8}
	\begin{split}
		&f_{N,\alpha} Q^{\alpha}_j = \sum_{m\in (\Z/N\Z)^d} F_{m,j} Q_m^{\alpha}, \\
		& F_{m,j} = \sum_{n,r\in \Z^d}\widehat{f}(n,j-m-rN) \,\e^{\frac{i\pi}{N}(j+m)\cdot n}(-1)^{n\cdot r}.
	\end{split}
	\end{equation}
	Notice that $F_{m,j}$ depends on $\alpha$, although not explicitly denoted here.
	\par
	Since $\overline{\widehat{f}(n,m)}=\widehat{\overline{f}}(-n,-m)$, we get that 
	\begin{equation}\label{qf9}
	\begin{split}
		\overline{F}_{j,m} &= \sum_{n,r\in \Z^d}\widehat{\overline{f}}(-n,-(m-j-rN)) \,\e^{-\frac{i\pi}{N}(j+m)\cdot n}(-1)^{n\cdot r} \\
		&= \sum_{n,r\in \Z^d}\widehat{\overline{f}}(n,j-m-rN) \,\e^{\frac{i\pi}{N}(j+m)\cdot n}(-1)^{n\cdot r}.
	\end{split}
	\end{equation}
	We see that $\overline{F}_{j,m} = F_{m,j}$ for real-valued $f$. For such an $f$ we see that 
	$f_{N,\alpha}$ is self-adjoint for the inner product  $(\bullet|\bullet)_0$ and that the map 
	$f \mapsto (F_{m,j})_{j,m \in (\Z/N\Z)^d}$ is onto, from $C^{\infty}(\T^{2d}_{\alpha}) $ to 
	the space of Hermitian matrices. 
	\\
	\par
	Any other metric on $\mathcal{H}_{\tth,\alpha}^d$ can be written as 
	$(u|v) = (Bu|v)_0 = (u|Bv)_0$. If $(f_{N,\alpha}u|v) = (u|f_{N,\alpha}v)$ for 
	real-valued $f$, then $B$ and $f_{N,\alpha}$ commute for all such $f$, and 
	hence for all Hermitian matrices. This shows that $B=c\,\mathrm{Id}$, as claimed. 
\end{proof}
From now on we equip $\mathcal{H}_{\tth,\alpha}^d$ with 
 the inner product $(\bullet|\bullet)_0$, for which the basis \eqref{qt3} is 
orthonormal, and drop the subscript. Furthermore, we use this basis to identify 
\begin{equation}\label{qf10}
	\mathcal{H}_{\tth,\alpha}^d \simeq \ell^2( (\Z/N\Z)^d) \simeq \C^{N^d}.
\end{equation}
Using \eqref{qf8}, \eqref{qf6.1}, we deduce the following result which was presented in the 
case $\alpha=1$ in \cite[Lemma 2.5]{ZwChrist10}. 
\begin{prop}\label{prop:qf1}
	Let $f\in S(m,\alpha)$, then 
	\begin{equation*}
		\tr f_{N,\alpha} = (N\alpha)^d \int_{\T^{2d}_{\alpha}} f(\rho)d\rho + r_N,
	\end{equation*}
	where for every $k\in\N$, there exists a constant $C_{k,d}>0$, depending only on $k$ and 
	the dimension $d$, such that 
	\begin{equation*}
		| r_N | \leq C_{k,d} N^{d-k} \alpha^{d-k/2}\sum_{|\beta|\leq \max(2d+1,k)}
		\| \partial^{\beta} f \|_{L^1(\T^{2d}_{\alpha})}.
	\end{equation*}
\end{prop}
We end this section with a boundedness result.
\begin{prop}\label{prop:qf2}
	Let $ N^{-1} \ll \alpha \leq 1$ and $p \in S(1,\alpha)$. Then, there exists a constant $C>0$,  
	independent of $N$ and $\alpha$, such that
	\begin{equation*}
			\| p_{N,\alpha} \|_{\ell^2 \to \ell^2} \leq C.
	\end{equation*}
\end{prop}
\begin{proof}0. In the case when $CN^{-\rho} \leq \alpha \leq 1$ with $\rho\in ]0,1[$, we 
			may follow the proof of Proposition 2.7 in \cite{ZwChrist10} with the 
			obvious modifications. We therefore present here only  the proof 
			in the critical case $\alpha = N^{-1}C$, $C\gg1$. 
			\\
			\par
			1. We begin by constructing a partition of unity of $\R^{2d}$ comprised out 
			    of $\alpha^{-1/2}\Z^{2d}$ periodic smooth functions. Indeed let $\ell \in [1,2]$ 
			    be such that $\alpha^{-1/2} \ell^{-1} =: M \in \N$. Let $K=[0,\ell]^d$ and let $X\Subset \R^{2d}$ 
			    be some small open relatively compact $\alpha$-independent neighbourhood of $K$. 
			    There exists a $\psi\in C^{\infty}_c(X;[0,1])$ such that $\psi \equiv 1$ on $K$ and 
			    $\partial_{\rho}^{\eta}\psi(\rho) =\mO_{\eta}(1)$, uniformly in $\alpha$, for any $\eta \in \N^{2d}$. Set 
			    \begin{equation*}
				\widetilde{\psi}(\rho) \defeq \sum_{g \in \Z^{2d}} \psi(\rho -\ell g) \geq 1,
			    \end{equation*}
			  and notice that $\partial_{\rho}^{\eta}\widetilde{\psi}(\rho) =\mO_{\eta}(1)$, uniformly in 
			  $\alpha$, for any $\eta \in \N^{2d}$. Setting
			      \begin{equation*}
				\phi(\rho) \defeq \frac{\psi(\rho)}{\widetilde{\psi}(\rho)},
			    \end{equation*}
			   we see that $\phi \in C^{\infty}_c(\R^{2d};[0,1])$ with $\supp \phi \Subset B(0,C_0)$, 
			   for some constant $C_0>0$, independent of $\alpha$, and with 
			   $\partial_{\rho}^{\eta}\phi(\rho) =\mO_{\eta}(1)$, uniformly in 
			  $\alpha$, for any $\eta \in \N^{2d}$. Moreover, 
			  \begin{equation}\label{bd0.0}
				\sum_{g \in \Z^{2d}} \phi(\rho -\ell g) = 1, \quad \forall \rho \in \R^{2d}.
			\end{equation}
			Write $\Z^{2d}_{M}=  (\Z/(M\Z))^{2d}$ and set 
			\begin{equation*}
				\phi_{\beta}(\rho) \defeq \sum_{g\in\Z^{2d}} \phi\left(\rho -\ell \beta - \alpha^{-1/2}g \right). 
			\end{equation*}
			Notice that $\alpha^{-1/2} = \ell M$. Clearly $\phi_{\beta}\in C^{\infty}(\R^{2d};[0,1])$ is a 
			$\alpha^{-1/2}\Z^{2d}$-periodic function, such that 
			 $\partial_{\rho}^{\eta}\phi_{\beta}(\rho) =\mO_{\eta}(1)$, uniformly in $\alpha$, for any 
			 $\eta \in \N^{2d}$. By \eqref{bd0.0} we get that 
			\begin{equation*}
				\sum_{\beta \in \Z^{2d}_{M}}
				\phi_{\beta}(\rho) = 1.
			\end{equation*}
			For $\beta \in \Z^{2d}_{M}$ we define 
			\begin{equation*}
				p_{\beta} \defeq \phi_\beta\, p.
			\end{equation*}
			It is well know (see for instance \cite[Section 7]{DiSj99}) 
			that the composition \eqref{sc3} can be written as the oscillatory integral
			\begin{equation}\label{bd0}
				(\overline{p}_{\beta_1}\#_{\tth} p_{\beta_2})(\rho)  
				= 
				\frac{C_d}{\tth^{2d}}\int_{\R^{4d}} \e^{-\frac{i}{\tth}\varphi(w) }
						\overline{p}_{\beta_1}(\rho - w_1) p_{\beta_2}(\rho -w_2)dw
			\end{equation}
			where $\varphi(w)= 2\sigma(w_1,w_2)$ with $|\det \varphi'' |= 2^{4d}$ and with signature 
			$\mathrm{sign}(\varphi'')=0$. 
			\par
			We split the integral \eqref{bd0} into two parts, $I_1(\rho)+I_2(\rho)$, 
			by using the cut-off functions $\chi(w)$ and 
			$(1-\chi(w))$, where $\chi \in C^{\infty}_c(\R^{4d};[0,1])$ is equal to $1$ on $B(0,1)$
			and equal to $0$ on $\R^{4d}\backslash B(0,2)$. 
			\\
			\par
			2. For $\gamma \in \N^{2d}$ we take the $\partial^{\gamma}_\rho$ derivative of $I_1$ 
			and obtain  
			\begin{equation*}
				\partial^{\gamma}_{\rho}I_1(\rho)
				= 
				\frac{C_d}{\tth^{2d}}\int_{\R^{4d}} \e^{-\frac{i}{\tth}\phi(w) }\chi(w)
						\partial^{\gamma}_{\rho}(\overline{p}_{\beta_1}(\rho - w_1) 
						p_{\beta_2}(\rho -w_2))dw.
			\end{equation*}
			The method of stationary phase, see for instance \cite[Proposition 5.2]{DiSj99}, yields that 
			\begin{equation}\label{bd1}
			\begin{split}
				&\partial^{\gamma}_{\rho}I_1(\rho) \\
				&= \mO_d(1)
				\left( \partial^{\gamma}_{\rho}(\overline{p}_{\beta_1}(\rho) 
						p_{\beta_2}(\rho))
					+\tth \sum_{|\eta|\leq 4d+3}
						\sup_{|w|\leq 2} 
						\left|\partial^{\eta}_w\chi(w)
						\partial^{\gamma}_{\rho}(\overline{p}_{\beta_1}(\rho - w_1) 
						p_{\beta_2}(\rho -w_2))\right|
				\right).
			\end{split}
			\end{equation}
			Notice that the terms on the right hand side are equal to $0$ unless $|w|\leq 2$ and 
			\begin{equation}\label{bd2.0}
				\begin{split}
					&|\rho - w_1 - \ell \beta_1|_{\T^{2d}_{\alpha}}
					=\inf\limits_{g\in\Z^{2d}}|\rho - w_1 - \ell\beta_1-\alpha^{-1/2}g| \leq C_0, \\
					&|\rho - w_2 - \ell\beta_2|_{\T^{2d}_{\alpha}} \leq C_0.
				\end{split}
			\end{equation}
			Then, 
			\begin{equation}\label{bd2}
			\begin{split}
				|\ell\beta_1 - \ell\beta_2|_{\T^{2d}_{\alpha}} 
				&\leq 
				|\rho - w_1 - \ell\beta_1|_{\T^{2d}_{\alpha}}  +|\rho - w_2 - \ell\beta_2|_{\T^{2d}_{\alpha}} +
				|w_1|_{\T^{2d}_{\alpha}} +|w_2|_{\T^{2d}_{\alpha}} \\
				& \leq 2C_0 +4 ,
			\end{split}	
			\end{equation}
			and similarly 
			\begin{equation}\label{bd3}
			\begin{split}
				\left|\rho - \frac{\ell\beta_1+\ell\beta_2}{2}\right|_{\T^{2d}_{\alpha}}
				&\leq 
				\frac{1}{2}|\rho - w_1 - \ell\beta_1|_{\T^{2d}_{\alpha}}  +
				\frac{1}{2}|\rho - w_2 - \ell\beta_2|_{\T^{2d}_{\alpha}} +
				\frac{1}{2}|w_1|_{\T^{2d}_{\alpha}} +\frac{1}{2}|w_2|_{\T^{2d}_{\alpha}} \\
				& \leq C_0 +2.
			\end{split}	
			\end{equation}
			Since all derivatives of $p\in S(1,\alpha)$ are bounded, we deduce from \eqref{bd1}, 
			\eqref{bd2}, \eqref{bd3} that 
			for any $K\in\N$, $\gamma \in \N^{2d}$, $\beta_1,\beta_2 \in \Z^{2d}_{M}$
			\begin{equation}\label{bd4}
				|\partial^{\gamma}_{\rho}I_1(\rho)| \leq 
				\mO_{d,\gamma,K} (1) \langle \ell\beta_1-\ell\beta_2\rangle_{\T^{2d}_{\alpha}}^{-K}
				\left\langle\rho-\frac{\ell\beta_1+\ell\beta_2}{2}\right\rangle_{\T^{2d}_{\alpha}}^{-K}.
			\end{equation}
			3. Next, we turn to the second part of \eqref{bd0} 
			\begin{equation*}
				I_2(\rho)
				= 
				\frac{C_d}{\tth^{2d}}\int_{\R^{4d}} \e^{-\frac{i}{\tth}\varphi(w) }(1-\chi(w))
						\overline{p}_{\beta_1}(\rho - w_1) 
						p_{\beta_2}(\rho -w_2)dw.
			\end{equation*}
			Since the integrand is equal to $0$ when $|w|\leq 1$, and $|\nabla \varphi| = 2|w|$, 
			we set $^tL= |\nabla \varphi |^{-2} \nabla \varphi \cdot  \frac{\tth}{i}\nabla $, and obtain from integration 
			by parts that for any $\gamma \in \N^{2d}$ and $T\in \N$ sufficiently large, 
			\begin{equation}\label{bd4.1}
			\begin{split}
				|\partial_{\rho}^{\gamma}I_2(\rho)|
				&=
				\left| 
				\frac{C_d}{\tth^{2d}}\int_{\R^{4d}} \e^{-\frac{i}{\tth}\varphi(w) }
				 L^T\big((1-\chi(w))
						\partial_{\rho}^{\gamma}(\overline{p}_{\beta_1}(\rho - w_1) 
						p_{\beta_2}(\rho -w_2))\big)dw\right| \\
				& \leq \mO_{d,\gamma,T}(1) \tth^{T-2d} 
					\int_{|w|\geq 1} \frac{f_{\beta_1}(\rho - w_1) f_{\beta_2}(\rho - w_2)}{|w|^T} dw,
			\end{split}
			\end{equation}
			where $f_{\beta_i}(\rho-w_i-\beta_i)$, $i=1,2$, are bounded continuous functions 
			which are equal to $0$ unless \eqref{bd2.0} holds. Since $|w|\geq 1$, we obtain 
			similarly to \eqref{bd2}, \eqref{bd3} that 
			\begin{equation}\label{bd5}
			\begin{split}
				|\ell\beta_1 - \ell\beta_2|_{\T^{2d}_{\alpha}} &\leq 2C_0 + \sqrt{2}|w|_{\T^{2d}_{\alpha}} \\
				&\leq \mO(1)|w|,
			\end{split}	
			\end{equation}
			and
			\begin{equation}\label{bd6}
			\begin{split}
				\left|\rho - \frac{\ell\beta_1+\ell\beta_2}{2}\right|_{\T^{2d}_{\alpha}} &\leq 
				C_0 +
				\frac{\sqrt{2}}{2}|w|_{\T^{2d}_{\alpha}} \\
				& \leq \mO(1)|w|.
			\end{split}	
			\end{equation}
			Here we used that $|w|_{\T^{2d}_{\alpha}}\leq |w|$. Hence, taking $T\in \N$ in 
			\eqref{bd4.1} sufficiently large, we obtain that for any 
			$K\in\N$, $\gamma \in \N^{2d}$, $\beta_1,\beta_2 \in \Z^{2d}_{M}$, 
			\begin{equation}\label{bd7}
				|\partial_{\rho}^{\gamma}I_2(\rho)|
				 \leq \mO_{d,\gamma,K} (1)
				 \langle \ell\beta_1-\ell\beta_2\rangle_{\T^{2d}_{\alpha}}^{-K}
				\left\langle\rho-\frac{\ell\beta_1+\ell\beta_2}{2}\right\rangle_{\T^{2d}_{\alpha}}^{-K}.
			\end{equation}
			4. Recall the definition \eqref{qf1} and notice that for any $\beta_1,\beta_2\in \Z^{2d}_{M}$, 
			$K \in \N\backslash\{0\}$
			\begin{equation}\label{bd8}
				 m_{\beta_1,\beta_2}(\rho) \defeq 
				 \langle \ell\beta_1-\ell\beta_2\rangle_{\T^{2d}_{\alpha}}^{-K}
				\left\langle\rho-\frac{\ell\beta_1+\ell\beta_2}{2}\right\rangle_{\T^{2d}_{\alpha}}^{-K}.
			\end{equation}
			is an order function on $\T^{2d}_{\alpha}$. Since the constants in the estimates \eqref{bd4} 
			and \eqref{bd7} are independent of $N$ and $\alpha$, it follows that 
			$\overline{p}_{\beta_1}\#_{\tth} p_{\beta_2}\in S(m_{\beta_1,\beta_2},\alpha)$. By 
			\eqref{qf5}, we then see that 
			\begin{equation}\label{bd9}
			q_{\beta_1,\beta_2} \defeq
			 p_{\beta_1}\#_{\tth} \overline{p}_{\beta_2}\#_{\tth}  \overline{p}_{\beta_1}\#_{\tth} p_{\beta_2}
			 \in S(m_{\beta_1,\beta_2}^2,\alpha).
			\end{equation}
			Using the calculus \eqref{qf5} and Proposition \eqref{prop:qf1}, we get that 
			\begin{equation}\label{bd10}
			\begin{split}
			\| (p_{\beta_1})^*_N&(p_{\beta_2})_N\|_{\mathrm{HS}}^2 \\
			&= \tr (p_{\beta_2})^*_N(p_{\beta_1})_N(p_{\beta_1})^*_N(p_{\beta_2})_N \\
			& = (N\alpha)^d \left( \int_{\T^{2d}_{\alpha}} q_{\beta_1,\beta_2}(\rho) d\rho 
			+ \mO(N\alpha^{1/2})^{-1}
			\sum_{|\gamma| \leq 2d+1}
			\| \partial^{\gamma} q_{\beta_1,\beta_2}\|_{L^1(\T^{2d}_{\alpha})}
			\right).
			\end{split}
			\end{equation}
			Taking $K\in \N$ in \eqref{bd8} sufficiently large, we can estimate the first integral by 
			\begin{equation}\label{bd11}
			\begin{split}
			\left| \int_{\T^{2d}_{\alpha}} q_{\beta_1,\beta_2}(\rho) d\rho \right|
			&\leq 
			 \mO(1)\int_{\T^{2d}_{\alpha}} m_{\beta_1,\beta_2}(\rho) d\rho \\
			& \leq \mO(1)  \langle \ell\beta_1- \ell\beta_2\rangle_{\T^{2d}_{\alpha}}^{-K}
				\int_{\R^{2d}}\left\langle\rho\right\rangle^{-K}d\rho \\ 
			& =  \mO(1)  \langle \ell\beta_1-\ell\beta_2\rangle_{\T^{2d}_{\alpha}}^{-K}.
			\end{split}
			\end{equation}
			Here, to see the second inequality, we split $\T^{2d}_{\alpha}\simeq [0,\alpha^{-1/2}]^d$ into 
			$2^d$ translates of the cube $[0,\alpha^{-1/2}/2]^d$ and used the translation invariance 
			of the Lebesgue measure.
			\par
			Similarly, we have that 
			\begin{equation}\label{bd12}
			\| \partial^{\gamma} q_{\beta_1,\beta_2}\|_{L^1(\T^{2d}_{\alpha})}
			=  \mO(1)  \langle \ell\beta_1-\ell\beta_2\rangle_{\T^{2d}_{\alpha}}^{-K}.
			\end{equation}
			Using that $\alpha =  CN^{-1}$, $C\gg1$, we get from \eqref{bd10}, \eqref{bd11}, \eqref{bd12}, 
			that
			\begin{equation}\label{bd13}
			\| (p_{\beta_1})^*_N(p_{\beta_2})_N\|_{\ell^2\to \ell^2} 
			\leq \mO(1)\langle \ell\beta_1-\ell\beta_2\rangle_{\T^{2d}_{\alpha}}^{-K}. 
			\end{equation}
			Using that $\alpha^{-1/2} = \ell M$, we get that for $K\in\N$ sufficiently large, 
			there exists a constant $C>0$, independent of $N$, such that 
			\begin{equation}\label{bd14}
				\sup\limits_{\beta_1\in \Z^{2d}_{M}} \sum_{\beta_2 \in \Z^{2d}_{M}} 
			\| (p_{\beta_1})^*_N(p_{\beta_2})_N\|_{\ell^2\to \ell^2} \leq C,
			\end{equation}
			and similarly  
			\begin{equation}\label{bd15}
				\sup\limits_{\beta_1\in \Z^{2d}_{M}} \sum_{\beta_2 \in \Z^{2d}_{M}} 
			\| (p_{\beta_1})_N(p_{\beta_2})^*_N\|_{\ell^2\to \ell^2} \leq C.
			\end{equation}
			Hence, by the Cotlar-Stein Lemma, see for instance \cite[Lemma 7.2]{DiSj99}, 
			$p_N = \sum_{\beta \in \Z^{2d}_{M}} (p_{\beta})_N$ converges strongly and 
			$\| p_N\|_{\ell^2\to \ell^2} \leq C$.
\end{proof}
\section{Functional calculus}\label{sec:FC}
We begin by recalling the functional calculus for pseudo-differential operators, 
as presented in \cite[Section 8]{DiSj99}, adapted to symbols in the class $S(m,\alpha)$ 
\eqref{qf2}. The following Proposition \ref{prop:fc1} in the case when $N^{-\rho}C \leq \alpha \ll1$, 
$\rho\in]0,1[$, has been proven in \cite[Lemma 2.8]{ZwChrist10}. The following result gives 
and extension including the critical case when $\alpha = CN^{-1}$, $C\gg1$.
\begin{prop}\label{prop:fc1}
	Let $N^{-1}\ll \alpha \ll1$. 
	Let $m\geq 1$ be an order function satisfying \eqref{qf1}, and let $p^w(x,\tth D_x;h,\alpha)$ 
	be a selfadjoint operator with $0 \leq p\in S(m,\alpha)$, and $p \sim \sum_0^{\infty}\tth^{\nu} p_{\nu}$ 
	in $S(m,\alpha)$, and with $p+i$ elliptic. Then, for every $\psi\in C^{\infty}_c(\R)$
	\begin{equation}\label{prop:fc1.1}
		\psi(p_{N,\alpha}) = f_{N,\alpha}, 
		\quad f \in S(1/m,\alpha)
	\end{equation}
	and
	\begin{equation}\label{prop:fc1.2}
		f \sim \sum_0^{\infty} \tth^{\nu} f_{\nu}(\rho;\alpha) \text{ in } S(1/m,\alpha), 
		\quad f_{\nu} \in S(1/m,\alpha).
	\end{equation}
	In particular, $f_0(\rho;\alpha) = \psi(p_0(\rho))$ and 
	\begin{equation}\label{prop:fc1.3}
		 f_{\nu}(\rho;\alpha) = \sum_{j=1}^{2\nu}g_j(\rho,\alpha)\psi^{(j)}(p_0(\rho)),
		 \quad g_j \in S(1,\alpha).
	\end{equation}
\end{prop}
\begin{rem}	
	We recall that in the above Proposition, although not denoted explicitly, 
	$p, p_{\nu}\in S(m,\alpha)$ may depend on $\alpha$, 
	however with the constants in the symbol estimates \eqref{qf2} 
	being independent of $\alpha$. 
\end{rem}
\begin{proof}[Proof of Proposition \ref{prop:fc1}]
	We employ the approach to the functional calculus of pseudo-differential operators 
	via the Helffer-Sj\"ostrand formula: For self-adjoint operators $A$, it follows 
	from the spectral theorem and the fact that $(\pi z)^{-1}$ is a fundamental solution 
	to $\partial_{\bar{z}}$ on $\C$, that 
	\begin{equation}\label{qf10.1}
		\psi(A) = - \frac{1}{\pi} \int (z-A)^{-1} \partial_{\bar{z}} \widetilde{\psi}(z) L(dz).
	\end{equation}
	Here $L(dz)$ denotes the Lebesgue measure on $\C$ and 
	$\widetilde{\psi} \in C_c^{\infty}(\C)$ is an almost holomorphic extension of $\psi$, 
	satisfying $\widetilde{\psi}\!\upharpoonright_{\R} = \psi$ and 
	$\partial_{\bar{z}}\widetilde{\psi}(z) = \mO(|\Ima z|^{\infty})$, see \cite[Chapter 8]{DiSj99} 
	for more details.
	\\
	\par
	Since $p+i$ is elliptic, so is $p-z$ for $|z|\leq C$ and $\Ima z \neq 0$. By Beals' Lemma, 
	see for instance \cite[Proposition 8.3]{DiSj99}, 
	it follows that $(z-p^w)^{-1} = r^w$ with $r\in S(1)$. In fact $r\in S(1,\alpha)$. To see this 
	notice first that conjugating $a^w(x,\tth D_x)$, $a\in S(m')$ for some order function $m'$, 
	with the unitary operator $\exp(i L(x,\tth D_x)/\tth)$, $L(x,\tth D_x) = x^*\cdot x + \xi^*\cdot \tth D_x$, 
	$\xi^*,x^*\in \R^d$, 	we get 
	\begin{equation*}
		\e^{\frac{i}{\tth} L(x,\tth D_x)}a^w(x,\tth D_x) \e^{-\frac{i}{\tth} L(x,\tth D_x)} 
		= (a\circ\phi_1)^w(x,\tth D_x).
	\end{equation*}
	Here $\phi_t$ is the flow at time $t$ associated with the Hamilton vector field $H_L = (\xi^*, -x^*)$ generated by $L$. 
	In particular $a\circ\phi_1(x,\xi)  =  a(x +\xi^*,\xi - x^*) \in S(m')$. Setting $x^* = \alpha^{-1/2}m$, 
	$\xi^* = -\alpha^{-1/2} n$, for $n,m\in \Z^d$, and using the $\alpha^{-1/2}\Z^{2d}$-translation invariance 
	of $p$, we get 
	\begin{equation*}
		(r\circ \phi_1)^w = \e^{\frac{i}{\tth} L}r^w\e^{-\frac{i}{\tth} L} = (z- \e^{\frac{i}{\tth} L}p^w\e^{-\frac{i}{\tth} L} )^{-1} 
		= (z- p^w)^{-1} 
		= r^w. 
	\end{equation*}
	This equality holds in the space of linear continuous maps $\cS\to \cS'$, so by the Schwartz' kernel theorem  
	$r\circ \phi_1 = r$ in $\cS'$, and therefore $r\circ \phi_1 = r$ point-wise since both are smooth functions. 
	Since $n,m$ where chosen arbitrarily, it follows that $r\in C^{\infty}(\T^{2d}_{\alpha})$, and in particular that 
	$r\in S(1,\alpha)$.
	\par
	Next, applying the modified Beals' symbol estimates as in \cite[Proposition 8.6]{DiSj99} and proceeding as in 
	the proof of Proposition 8.7 and of (8.19) \cite{DiSj99}, it follows that $\psi(p^w) = f^w$ with $f\in S(1/m,\alpha)$, 
	satisfying \eqref{prop:fc1.2}, \eqref{prop:fc1.3}.
	\\
	\par
	The above discussion shows that $r_N = (z-p_N)^{-1}$, so by \eqref{qf10.1}, we see that 
	\begin{equation*}
		\psi(p_{N,\alpha}) = \psi(p^w\!\upharpoonright_{\mathcal{H}_{\tth,\alpha}^d }) 
		= \psi(p^w)\!\upharpoonright_{\mathcal{H}_{\tth,\alpha}^d }
	\end{equation*}
	and \eqref{prop:fc1.1} follows.
\end{proof}
\subsection{Phase space dilation}
Until further notice we let $h=1/(2\pi N)$, $0<N\in \N$ and $h \ll \alpha \ll 1$. Let $p\in C^{\infty}(\T^{2d})$ 
with $p \sim p_0 + hp_1 + \dots$ in $S(1,1)$. Then, $P = p^w(x,hD_x;h)$ is a bounded 
operator $L^2(\R^{d})\to L^2(\R^{d})$. Setting $Q=P^*P$, we see, by the semiclassical calculus reviewed in Section \ref{sec:SC},  that the selfadjoint operator 
\begin{equation}\label{qf11.0}
\begin{split}
	&Q=q^w(x,hD_x;h) \text{ has the symbol } \\ 
	&q=\overline{p}\#_h p \sim q_0 + h q_1 + \dots \text{ in } S(1,1), 
	\quad \text{with }q_0 = |p_0|^2.
\end{split}
\end{equation} 
\\
\par 
The transformation 
\begin{equation}\label{qf11}
	(U_{\alpha} \phi)(x) = \alpha^{d/4}\phi (\alpha^{1/2}x), \quad \phi \in \cS(\R^d),
\end{equation}
is a continuous bijection $\cS \to \cS$, and $\cS' \to \cS'$ by duality. Notice that the 
factor $\alpha^{d/4}$ has been chosen so that $U_{\alpha} : L^2 \to L^2$ unitarily 
with $U_{\alpha}^* = U_{\alpha}^{-1} = U_{\alpha^{-1}}$. Moreover, $U_{\alpha}$ maps 
$C^{\infty}(\T^{2d}) \to C^{\infty}(\T^{2d}_{\alpha})$ continuously, 
and $ \mathcal{H}_{h,1}^d \to \mathcal{H}_{\tth,\alpha}^d$ unitarily 
with respect to the inner products introduced in Lemma \ref{lem:qf1}. In particular, we see by  
\eqref{qt3} that 
\begin{equation}\label{qf12}
	U_{\alpha} Q_j^1 = Q_j^{\alpha}. 
\end{equation}
Using $U_{\alpha}$ we perform the phase space 
dilation $\T^{2d} \ni (x,\xi) =\alpha^{1/2}(\widetilde{x},\widetilde{\xi})$, 
$(\widetilde{x},\widetilde{\xi}) \in \T^{2d}_{\alpha}$, and get 
\begin{equation}\label{qf13}
	q^w(x,hD_x;h) = U_{\alpha}^{-1} q^w(\alpha^{1/2}(\widetilde{x},\tth D_{\widetilde{x}});h)U_{\alpha}, 
	\quad \tth = \frac{h}{\alpha}.
\end{equation}
\par
Writing $\widetilde{q}(\widetilde{\rho};h) := q(\alpha^{1/2}\widetilde{\rho};h)\in S(1,\alpha)$, we conclude 
from the mapping properties discussed after \eqref{qf11}, that

\begin{equation}\label{qf14}
	 q_N = q^w(x,hD_x;h)\!\upharpoonright_{\mathcal{H}_{h,1}^d }
	=
	U_{\alpha}^{-1}
	\left( \widetilde{q}^w(\alpha^{1/2}(\widetilde{x},\tth D_{\widetilde{x}});h)
	\!\upharpoonright_{\mathcal{H}	_{\tth,\alpha}^d }\right) 
	U_{\alpha}
	=
	U_{\alpha}^{-1}
	\widetilde{q}_{N,\alpha}
	U_{\alpha}.
\end{equation}
Next, we follow the ideas of \cite[Section 4]{HaSj08}, and 
pass to a new order function adapted to the rescaled 
symbol $\alpha^{-1}\widetilde{q}$. Furthermore, we drop the tilde on the variable $\rho$  
to ease the notation. Since $q_0\in S(1,1)$, we see that 
\begin{equation}\label{psd4}
	 m(\rho) \defeq 1 + \frac{q_0(\alpha^{1/2}\rho)}{\alpha}\geq 1, \quad \rho \in \T^{2d}_{\alpha},
\end{equation}
is a function in $C^{\infty}(\T^{2d}_{\alpha})$. We check that it 
satisfies the estimate for an order function \eqref{qf1}. Indeed, for $|\beta|=1$
\begin{equation*}
	\partial_{\rho}^{\beta} m(\rho)  = \frac{(\partial_{\rho}^{\beta} q_0)(\alpha^{1/2}\rho)}{\alpha^{1/2}} 
	\leq C \frac{q_0^{1/2}(\alpha^{1/2}\rho)}{\alpha^{1/2}} 
	\leq  C \,m(\rho)^{1/2},
\end{equation*}
and for $|\beta|=2$ 
\begin{equation*}
	\partial_{\rho}^{\beta} m(\rho)  = (\partial_{\rho}^{\beta} q_0)(\alpha^{1/2}\rho)
	\leq C,
\end{equation*}
where the constant $C>0$ (not necessarily the same in both inequalities) is 
independent of $\alpha$. Hence, we get by Taylor expansion that for $\rho, \mu \in \R^{2d}$ 
\begin{equation*}
	m(\rho) \leq m(\mu) + C m(\rho)^{1/2}| \rho - \mu| + C| \rho - \mu| ^2,
\end{equation*}
and since $m \geq 1 $ that 
\begin{equation*}
	m(\rho) \leq C \langle \rho - \mu \rangle^2 m(\mu).
\end{equation*}
Using the $\alpha^{-1/2}\Z^{2d}$-translation invariance of $m$, we see that for any $\rho',\mu'\in \T^{2d}_{\alpha}$ and 
any $\gamma \in \Z^{2d}$
\begin{equation*}
	m(\rho') \leq C \langle \rho' - \mu' +\alpha^{-1/2}\gamma \rangle^2 m(\mu).
\end{equation*}
Since this holds for any $\gamma$ it also holds for the infimum over $\gamma\in \Z^{2d}$, and we 
deduce that $m$ satisfies \eqref{qf1}. 
\par
Similarly, we get the following symbol 
estimates (with respect to the new order function $m$ \eqref{psd4})
\begin{equation}\label{psd5}
\begin{split}
		&\widetilde{q}(\rho) \leq m(\rho) \\
		&\partial^{\beta}_{\rho} \widetilde{q}(\rho) =\mO(1) m(\rho)^{1/2}, \quad |\beta| =1 \\ 
		&\partial^{\beta}_{\rho} \widetilde{q}(\rho) = \mO(1) , \quad |\beta| =2 \\ 
		&\partial^{\beta}_{\rho} \widetilde{q}(\rho) = \mO(1) \alpha^{|\beta|/2-1}, \quad |\beta|  \geq 2.
\end{split}
\end{equation}
and in general 
\begin{equation}\label{psd6}
	\partial^{\beta}_{\rho} \widetilde{q}(\rho) = \mO_{\beta}(1)m(\rho).
\end{equation}
We note that in the above equations the constants $\mO(1)$ are independent of $\alpha$. Hence, 
$\widetilde{q} \in S(m,\alpha)$, and setting 
$\widetilde{q}_{\nu}(\rho) = \alpha^{\nu-1} q(\alpha^{1/2}\rho)$, $\nu \in \N$, we see that 
$\widetilde{q} \sim \sum_0^{\infty} \tth^\nu \widetilde{q}_{\nu}$ in $S(m,\alpha)$.
\\
\par
Since $\widetilde{q}_0+i$ is elliptic in $S(m,\alpha)$, we may applying the functional calculus given in  
Proposition \ref{prop:fc1}, and we get for $\psi\in C_c^{\infty}(\R)$ that 
\begin{equation}\label{psd7}
\begin{split}
		&\psi(\alpha^{-1}\widetilde{q}_{N,\alpha}) = f_{N,\alpha}, 	\quad f \in S(1/m,\alpha),  \\
		&f \sim \sum_0^{\infty} \tth^{\nu} f_{\nu}(\rho;\alpha) \text{ in } S(1/m,\alpha), 
		\quad f_{\nu} \in S(1/m,\alpha),
	\end{split}
\end{equation}
with $f_0(\rho;\alpha) = \psi(q_0(\alpha^{1/2}\rho)/\alpha)$ and 
\begin{equation}\label{psd8}
		 f_{\nu}(\rho;\alpha) = \sum_{j=1}^{2\nu}g_j(\rho,\alpha)\psi^{(j)}(q_0(\alpha^{1/2}\rho)/\alpha),
		 \quad g_j \in S(1,\alpha).
\end{equation}
Next, we recall \cite[Proposition 4.1]{HaSj08} translated to our calculus.
\begin{prop}\label{prop:fc2}
	Let $\widetilde{m}\in C^{\infty}(\T^{2d}_{\alpha},]0,\infty[)$ be an order function satisfying \eqref{sc1.1}, so that 
	$\widetilde{m}(\rho)= 1$ when $\alpha^{-1}q_0(\alpha^{1/2}\rho) \leq \sup \supp \psi +1/C$, for some $C>0$ that is  
	independent of $\alpha$. Then, \eqref{psd7} holds in $S(\widetilde{m},\alpha)$, for $h$,$\tth$ sufficiently small. 
\end{prop}
\begin{proof}
	The proof of Proposition 4.1 in \cite{HaSj08} is based on the functional calculus using the Helffer-Sjöstrand 
	formula and on standard semiclassical calculus. It translates directly to our case using the notions 
	discussed in Section \ref{sec:SQ}.
\end{proof}
\subsection{Log-determinant estimates}
The following result is an adaptation of the results of \cite[Section 4]{HaSj08} to our situtation.
\begin{prop}\label{prop:ld1}
Let $0< N\in \N$ and $N^{-1} \ll \alpha \ll 1$, and let $q$ be as in \eqref{qf11.0}. Suppose 
that there exists a $\kappa \in ]0,1]$ such that 
\begin{equation}\label{prop:ld1.1}
		V(t) = \mathrm{Vol}\{\rho \in \T^{2d}; q_0(\rho) \leq t \} = \mO(t^{\kappa}), 
		\quad 0 \leq t \ll1.
\end{equation}
Then, for $\psi\in C^{\infty}_c(\R)$ 
\begin{equation}\label{prop:ld1.2}
\tr \psi\left( \frac{\widetilde{q}_{N,\alpha}}{\alpha}\right)
		 = 
		N^d \left( \int\psi\left(\frac{q_0}{\alpha}\right) d V(q_0)
		+ 
		\mO(N\alpha)^{-1}\alpha^{\kappa} \right),
\end{equation}
Moreover, for $\chi \in C^{\infty}_c([0,\infty[,[0,\infty[)$ with $\chi(0)>0$, 
\begin{equation}\label{prop:ld1.3}
		\log \det \left( q_N + \alpha \chi \left(\frac{q_N}{\alpha}\right )\right) 
		 = N^d \left( \int_{\T^{2d}}  \log q_0(\rho)  \, d\rho + 
		 	 \mO\!\left(\alpha^{\kappa}\log \frac{1}{\alpha}\right)\right).
\end{equation}
\end{prop}
Before, we turn to the proof, let us make the following remark: since the trace 
is invariant under unitary conjugation, \eqref{prop:ld1.2} and \eqref{qf14}  
imply that the number $N(q_N,\alpha)$ of eigenvalues of $q_N$ in the interval $[0,\alpha]$ 
is 
\begin{equation}\label{ld10}
	N(q_N,\alpha) = \mO(N^d \alpha^{\kappa}).
\end{equation}
\begin{proof}[Proof of Proposition \ref{prop:ld1}]
	We essentially follow the proof given in \cite[Section 4]{HaSj08} with some modifications 
	to suit our setting. 
	\\
	\par
	1. Extend $\chi \in C^{\infty}_c(\R;\C)$ in such a way that $\chi(x)>0$ near $0$ and 
	$x+\chi(x) \neq 0$ for all $x\in\R$. Suppose that $\widehat{q}_N >0$, with 
	$0 < 1/C\leq \widehat{q}\in S(1,1)$. 
	Setting $\widehat{q}^t := t\widehat{q} + (1-t) \in S(1,1)$, $t\in [0,1]$, so that $\widehat{q}^0_N 
	= \mathbf{1}_{\mathcal{H}^d_h}$ and $\widehat{q}^1_N =\widehat{q}_N $, we see that 
	$\widehat{q}^t_N \geq 1/C$, uniformly in $t$. Hence, $(\widehat{q}^t_N)^{-1}$ exists and 
	is given by $r^t_N$, for some $r^t\in S(1,1)$, by Beals' Lemma \cite[Proposition 8.3]{DiSj99}. 
	Indeed the periodicity 
	of the symbol $r^t$ follows by an argument similar to the one in the proof of Proposition \ref{prop:fc1}. 
	Moreover, since $\widehat{q}^t \in S(1,1)$ with symbol estimates uniformly with respect to $t$,  
	and since $r^t_N$ is bounded uniformly with respect to $t$, Beals' Lemma shows that 
	$r^t \in S(1,1)$ uniformly with respect to $t$. 
	\par
	Hence, we get by the symbolic calculus \eqref{qf5}, \eqref{sc4} and by 
	Proposition \ref{prop:qf1} that 
	\begin{equation}\label{ld.0}
	\begin{split}
		\frac{d}{dt} \log \det \left( \widehat{q}^t_N \right) 
		&= \tr (\widehat{q}^t_N )^{-1}\frac{d}{dt}  \widehat{q}^t_N \\ 
		&=\tr \left((\widehat{q}_0^t)_N )^{-1}\frac{d}{dt}  (\widehat{q}_0^t)_N + \mO(N^{-1})\right) \\
		&=N^d \left( \int_{\T^{2d}} \frac{d}{dt} \log \widehat{q}_0^t (\rho) \,d\rho + \mO(N^{-1})\right)
	\end{split}
	\end{equation}
	Notice that the $\mO(N^{-1})$ term in the second line means a symbol $b\in N^{-1}S(1,1)$ 
	with symbol estimates 
	uniform with respect to $t\in [0,1]$. Integrating \eqref{ld.0} from $t=0$ to $t=1$, we get 
	\begin{equation}\label{ld.0.1}
		 \log \det \left( \widehat{q}_N \right)  
		 =N^d \left( \int_{\T^{2d}}  \log \widehat{q}_0 (\rho) \,d\rho + \mO(N^{-1})\right).
	\end{equation}
	For fixed $0<\alpha_1\ll1$, the above discussion applies to 
	$q_N + \alpha_1 \chi(\alpha_1^{-1}q_N)$. Hence 
	\begin{equation}\label{ld.0.2}
		\log \det \left( q_N + \alpha_1\chi \left(\frac{q_N}{\alpha_1}\right )\right) 
		 =N^d \left( \int_{\T^{2d}}  \log \left(q_0 (\rho)
		 + \alpha_1\chi \left(\frac{q_0 (\rho)}{\alpha_1}\right ) \right)\, d\rho + \mO(N^{-1})\right).
	\end{equation}
	2. Next, let $ 0< h \ll \alpha \leq t \leq  \alpha_1$. For $t>0$, $E\geq 0$ we have 
	\begin{equation}\label{ld.0.3}
		\frac{d}{dt} \log (E + t\chi(E/t)) = \frac{1}{t} \frac{\chi(E/t)-(E/t)\chi'(E/t)}{E/t + \chi(E/t)}
		\defeq \frac{1}{t}\psi(E/t),
	\end{equation}
	where $\psi \in C^{\infty}_c(\R)$. As in \eqref{qf14}, we have  
	$\widetilde{q}_{N,t}  =  \widetilde{q}^w(t^{1/2}(x,\tth D_{x});h)
	\!\upharpoonright_{\mathcal{H}	_{\tth,t}^d }$, $\tth = h/t$, with 
	$\widetilde{q}(\rho;h) = t^{-1}q(t^{1/2}\rho;h) \in S(m,t)$ 
	and $m$ as in \eqref{psd4}. 
	\par
	Then, we get by standard self-adjoint functional calculus that 
	\begin{equation}\label{ld.0.4}
	\begin{split}
	\frac{d}{dt} \log \det \left( \widetilde{q}_{N,t}  + t\chi(t^{-1} \widetilde{q}_{N,t} ) \right)
	&= 
	\tr ( \widetilde{q}_{N,t}  + t\chi(t^{-1} \widetilde{q}_{N,t} ) )^{-1} 
	\frac{d}{dt}( \widetilde{q}_{N,t}  + t\chi(t^{-1} \widetilde{q}_{N,t} ) \\ 
	& = \tr \frac{1}{t} \psi\left( \frac{\widetilde{q}_{N,t}}{t}\right).
	\end{split}
	\end{equation}
	From Proposition \ref{prop:fc2} in combination with \eqref{psd7}, \eqref{psd8}, we get that 
	\begin{equation*}
	\begin{split}
		&\psi\left( \frac{\widetilde{q}_{N,t}}{t}\right) = f_{N,t}, 	\quad f \in S(\widetilde{m},t),  \\
		&f \sim \frac{q_0(t^{1/2}\rho)}{t} + 
		\tth f_1(\rho;t) +\dots  \text{ in } S(\widetilde{m},t), \quad \tth = \frac{h}{t} = \frac{1}{2\pi N t},
	\end{split}
	\end{equation*}
	with $f_{\nu} \in S(\widetilde{m},t)$ as in \eqref{psd8}  (with $\alpha$ replaced by $t$). 
	Using Proposition \ref{prop:qf1} and \eqref{ld.0.4}, we get for any $\N \ni M,k \geq 1$ 
	\begin{equation}\label{ld.1}
	\begin{split}
		\frac{d}{dt} \log \det \left( \widetilde{q}_{N,t}  +
		 t\chi\left( \frac{\widetilde{q}_{N,t}}{t}\right) 
		 \right) 
		 = &
		\frac{N^d t^{d}}{t}\int_{\T^{2d}_t}\psi\left(\frac{q_0(t^{1/2}\rho)}{t}\right) d\rho  \\
		&+ 
		\mO_M(1)\frac{N^d t^{d}}{t}\tth
		\int_{\T^{2d}_t}\widetilde{\chi}\left(\frac{q_0(t^{1/2}\rho)}{t}\right) d\rho  \\ 
		&+ 
		\mO_M(1)\frac{N^d t^{d}}{t}\tth^M
		\int_{\T^{2d}_t}\widetilde{m}(\rho)d\rho  \\ 
		&+ 
		\mO_k(1)\frac{N^d t^{d}}{t}(Nt^{1/2})^{-k}
		\int_{\T^{2d}_t}\widetilde{m}(\rho) d\rho,
	\end{split}
	\end{equation}
where $\widetilde{\chi}\in C^{\infty}_c(\R;[0,1])$ is equal to $1$ on a small neighbourhood of $[0,\sup\supp \psi]$. 
Here, the third term 
is the error term from the asymptotic expansion of $f$, and 
the fourth term is the error term stated in Proposition \ref{prop:qf1}. Taking, $M=k$, 
we see that the last term can be absorbed in the third term. Next, we set  
\begin{equation*}
	\widetilde{m}(\rho)  = \left( 1 +
	\dist_{\T^{2d}_t}\left( \rho, \supp \widetilde{\chi}
	\left(\frac{q_0(t^{1/2}\bullet)}{t}\right)
	\right)^2\right)^{-M'}
\end{equation*}
for some $1\leq M'\in \N$ to be chosen later on. Here the distance $\dist_{\T^{2d}_t}$ is defined 
similar to \eqref{qf1}.
\\ 
\\
3. We begin with the leading in term on the right hand side in \eqref{ld.1}. The change of variables 
$t^{1/2} \rho = \rho '$, $d\rho = t^{-d}d\rho'$, and \eqref{ld.0.3} yield 
\begin{equation*}
	\frac{N^d t^{d}}{t}\int_{\T^{2d}_t}\psi\left(\frac{q_0(t^{1/2}\rho)}{t}\right) d\rho 
	= 
	N^d\int_{\T^{2d}} \frac{d}{dt}\log (q_0(\rho) + t\chi(q_0(\rho)/t)) d\rho.
\end{equation*}
Integrating this quantity from $t=\alpha$ to $t=\alpha_1$, we find 
 \begin{equation}\label{ld.2}
 	\left. N^d\int_{\T^{2d}}\log (q_0(\rho) + t\chi(q_0(\rho)/t)) d\rho\right|_{t=\alpha}^{\alpha_1}.
\end{equation}
Next, we treat the second term on the right hand side of \eqref{ld.1}. Performing the 
same change of variables as above yields that
\begin{equation*}
\begin{split}
	\mO_M(1)\frac{N^d t^{d}}{t}\tth
		\int_{\T^{2d}_t}\widetilde{\chi}\left(\frac{q_0(t^{1/2}\rho)}{t}\right) d\rho 
	&= 
	\mO_M(1) N^{d-1}t^{-2}
		\int_{\T^{2d}}\widetilde{\chi}\left(\frac{q_0(\rho)}{t}\right) d\rho \\
	&\leq \mO_M(1)N^{d-1}t^{-2} \int_{ \T^{2d}} \mathbf{1}_{q_0 \leq Ct}\,d\rho,
\end{split}
\end{equation*}
for some $C>0$ depending only on $\widetilde{\chi}$. 
The integral from $t=\alpha$ to $t=\alpha_1$ of this quantity, is bounded by 
\begin{equation}\label{ld.3}
\begin{split}
	&\mO_M(1) N^{d-1}
		\int \int_{\T^{2d}}t^{-2} \mathbf{1}_{ \max(q_0/C,\alpha) \leq t \leq \alpha_1} \,d\rho \,dt \\
	& = \mO_M(1) N^{d-1}
	      \int_{\T^{2d}} \left( \frac{1}{\max(q_0/C,\alpha)} - \frac{1}{\alpha_1}\right)
	       \mathbf{1}_{ q_0 \leq C\alpha_1} \,d\rho \\
	& \leq \mO_M(1) N^{d-1}
		 \int_{\T^{2d}}  \frac{1}{\alpha + q_0(\rho)} 
	       \mathbf{1}_{ q_0 \leq C\alpha_1} \,d\rho \\
	& = \mO_M(1) N^{d-1} \int_0^{C\alpha_1} \frac{1}{\alpha + t } \, dV(t),
\end{split}
\end{equation}
where $dV = (q_0)_*(d\rho)$ is the push-forward measure of the Lebesgues measure $d\rho $ 
on $\T^{2d}$ by $q_0$, with distribution function 
\begin{equation*}
	V(t) = \int_{\T^{2d}}  \mathbf{1}_{ q_0(\rho) \leq t} \,d\rho, \quad 0 \leq t \leq C\alpha_1 \ll1,
\end{equation*}
which is an increasing and right-continuous function. 
\\
\par
4. We turn to the third term on the right hand side of \eqref{ld.1}. Set  
\begin{equation*}
	d_t(\rho) = \dist_{\T^{2d}_t}\left( \rho, \{ \rho'\in \T^{2d}_t ; t^{-1}q_0(t^{1/2}\rho')\leq C\}\right) 
\end{equation*}
for some sufficiently large $C>0$. For $\rho' \in \T^{2d}_t $ with 
$q_0(t^{1/2}\rho') \leq Ct$, it follows by \eqref{psd5}, \eqref{psd4} that 
$\nabla q_0(t^{1/2}\rho') /t= \mO(1)$ uniformly in $t$. Similarly, 
$\nabla^2 q_0(t^{1/2}\rho') /t= \mO(1)$. Hence, by Taylor expansion we 
get for any $\rho \in \T^{2d}_t$ 
\begin{equation*}
	t^{-1} q_0(t^{1/2}\rho) 
	\leq \mO(1) ( 1 +d_t(\rho) + d_t(\rho)^2)	
	\leq \mO(1) ( 1 +d_t(\rho)^2).
\end{equation*}
We split the integral in the third term in \eqref{ld.1} into two parts: one where $q_0(t^{1/2}\rho) \leq C\alpha_1$, 
and one where $q_0(t^{1/2}\rho) > C\alpha_1$, for some $C>0$. The first part is bounded by 
\begin{equation}\label{ld.5.1}
\begin{split}
	\mO_M(1)&\frac{N^d t^{d}}{t}\tth^M
		\int_{\T^{2d}_t } \widetilde{m}(\rho) 
		\mathbf{1}_{q_0(t^{1/2}\rho) \leq C\alpha_1 }d\rho \\
		&\leq \frac{\mO_{M,M'}(1) (Nt)^{d-M}}{t} 
		\int_{\T^{2d}_t} ( 1+t^{-1} q_0(t^{1/2}\rho))^{-M'}\mathbf{1}_{q_0(t^{1/2}\rho) \leq C\alpha_1 }d\rho\\
		&
		\leq \frac{\mO_{M,M'}(1)N^d }{t(Nt)^{M}} 
		\int_{\T^{2d}} ( 1+t^{-1} q_0(\rho))^{-M'}\mathbf{1}_{q_0(\rho) \leq C\alpha_1 }d\rho.
\end{split}
\end{equation}
The second part is 
\begin{equation}\label{ld.5.2}
\begin{split}
	\mO_M(1)&\frac{N^d t^{d}}{t}\tth^M
		\int_{\T^{2d}_t } \widetilde{m}(\rho) 
		\mathbf{1}_{q_0(t^{1/2}\rho) > C\alpha_1 }d\rho \\
		&
		\leq \frac{\mO_{M,M'}(1)N^d }{t(Nt)^{M}} 
		\int_{\T^{2d}} ( t^{-1} q_0(\rho))^{-M'}\mathbf{1}_{q_0(\rho) > C\alpha_1 }d\rho \\ 
		& \leq \frac{\mO_{M,M'}(1)N^d }{t(Nt)^{M}} t^{M'} \\ 
		& =\mO_M(N^{d-M}),
\end{split}
\end{equation}
where we chose $M'=M+1$ and the last estimate is uniform in $t$. Going back to \eqref{ld.5.1}, 
we integrate it from $t=\alpha$ to $t=\alpha_1$, exchange the integrals, 
and, keeping in mind that $q_0(\rho) \leq C\alpha_1$,  we estimate
\begin{equation}\label{ld.5.3}
		\int_{\alpha}^{\alpha_1} \frac{1}{t^{M+1}( 1+t^{-1} q_0(\rho))^{M'}}dt,
\end{equation}
with $M'=M+1$. 
When $q_0\leq \alpha$, then 
\begin{equation*}
		\int_{\alpha}^{\alpha_1} 
		 \frac{1}{t^{M+1}( 1+t^{-1} q_0)^{M+1}}dt 
		 = \mO_M(1)\int_{\alpha}^{\alpha_1} t^{-M-1}dt = \mO_M(\alpha^{-M}).
\end{equation*}
When $\alpha \leq q_0 \leq \alpha_1$,, then 
\begin{equation*}
\begin{split}
		\left(\int_{\alpha}^{q_0} + \int_{q_0}^{\alpha_1} \right)
		 \frac{1}{t^{M+1}( 1+t^{-1} q_0)^{M+1}}dt 
		 = \mO_M(1)q_0^{-M}.
\end{split}
\end{equation*}
When $\alpha_1 \leq q_0 $,  then \eqref{ld.5.3} is bounded from above by 
 \begin{equation*}
\begin{split}
		\int_{\alpha}^{\alpha_1} 
		 \frac{1}{q_0^{M+1}}dt 
		 = \mO_M(1)
\end{split}
\end{equation*}
In conclusion the integral from $t=\alpha$ to $t=\alpha_1$ of \eqref{ld.5.1} 
is 
\begin{equation}\label{ld.5.4}
	        \mO_{M}(1)N^d
		\int_0^{C\alpha_1}
		 \left( \frac{N^{-1}}{\alpha+ q_0}\right)^{M}dV(q_0).
\end{equation}
Summing up what he have shown so far, we get by \eqref{ld.1},  \eqref{ld.0.2},  \eqref{ld.2}, as well as 
\eqref{ld.3}, \eqref{ld.5.2} and \eqref{ld.5.4} that for any $2\leq M \in \N$ 
\begin{equation}\label{ld.6}
\begin{split}
		\log \det \left( q_N + \alpha \chi \left(\frac{q_N}{\alpha}\right )\right) 
		 = &N^d \left( \int_{\T^{2d}}  \log \left( q_0(\rho) + 
		 	\alpha \chi \left(\frac{q_0(\rho)}{\alpha}\right )\right)  \, d\rho + \mO(N^{-1})\right. \\
		    &+\mO_M(1)  \int_0^{C\alpha_1} \frac{N^{-1}}{\alpha + q_0 } \, dV(q_0), \\ 
		    &+   \left.     \mO_{M}(1)
		\int_0^{C\alpha_1} 
		 \left( \frac{N^{-1}}{\alpha+ q_0}\right)^{M}dV(q_0)+ \mO_M(N^{-M})\right).
\end{split}
\end{equation}
Let us remark at this point that most of the above discussion 
applies to general $\psi\in C^{\infty}_c(\R)$. Thus, using \eqref{ld.1}, \eqref{ld.0.4}, 
\eqref{ld.5.1}, \eqref{ld.5.2}, we get for any $M,M'\geq 1$ that 
\begin{equation}\label{ld.5}
\begin{split}
\tr \psi\left( \frac{\widetilde{q}_{N,\alpha}}{\alpha}\right)
		 = &
		N^d \left( \int\psi\left(\frac{q_0}{\alpha}\right) d V(q_0)
		+ 
		\mO\!\left((N\alpha)^{-1}\right)
		\int\widetilde{\chi}\left(\frac{q_0}{\alpha}\right) d V(q_0)  \right.\\ 
		&+\left. 
		\mO_{M,M'}\!\left((N\alpha)^{-M}\right)
		\int_0^{C\alpha_1}\left(1 +\frac{q_0}{\alpha} \right)^{-M'}\,dV(q_0)
		+\mO_{M,M'}((N\alpha)^{-M}\alpha^{M'}) \right)
	\end{split}
\end{equation}
5. Splitting the first integral in \eqref{ld.6} into one where $q_0 \leq C\alpha$ and one 
where $q_0 > C\alpha$, we see that 
\begin{equation}\label{ld.6.1}
\begin{split}
\int_{\T^{2d}}  \log \left( q_0(\rho) + 
		 	\alpha \chi \left(\frac{q_0(\rho)}{\alpha}\right )\right)  \, d\rho 
			=\int_{\T^{2d}}  \log q_0(\rho)   \, d\rho  + 
			\mO(1) \int_{0}^{C\alpha}  \log  \frac{1}{q_0} dV(q_0).
\end{split}
\end{equation}
To estimate the second term on the right hand side of \eqref{ld.6.1}, 
we use \eqref{prop:ld1.1} and integration by parts, and get that 
\begin{equation}\label{ld.7}
\begin{split}
	 \int_0^{C\alpha} \log x\,dV(x) &=
	V(C\alpha) \log C\alpha
	  + \int_0^{C\alpha} \frac{1}{x} V(x)dx \\ 
	  & \leq \mO(\alpha^{\kappa}\log \frac{1}{\alpha}).
\end{split}
\end{equation}
Similarly, we get that the second integral in \eqref{ld.6} is
\begin{equation}\label{ld.8}
\begin{split}
	 \int_0^{C\alpha_1} \frac{N^{-1}}{\alpha + t } \, dV(t) &=
	 \frac{N^{-1}}{\alpha + C\alpha_1} V(C\alpha_1) 
	  + \int_0^{C\alpha_1} \frac{N^{-1}}{(\alpha + t)^2 } V(t)dt \\ 
	  & \leq \mO(N^{-1}) \left( \ \frac{(C\alpha_1)^{\kappa}}{\alpha + C\alpha_1}
	  +\alpha^{\kappa -1}\int_0^{C\alpha_1/\alpha} \frac{t^{\kappa}}{(1 + t)^2 }dt \right) \\
	&  =\begin{cases}
		\mO(N^{-1}\log\frac{1}{\alpha}), \quad \kappa =1 \\ 
		\mO(N^{-1}\alpha^{\kappa -1}), \quad  0<\kappa < 1. 
	 	\end{cases}
\end{split}
\end{equation}
The last integral in \eqref{ld.6} gives
\begin{equation}\label{ld.9}
\begin{split}
		\int_0^{C\alpha_1}
		 \left(\frac{N^{-1}}{\alpha+ q_0}\right)^{M}dV(q_0) 
		& =  \left.\left(\frac{N^{-1}}{\alpha+ q_0}\right)^{M}V(q_0) \right|_{q_0 =0}^{C\alpha_1} 
		    + N^{-M} \int_0^{C\alpha_1}
		 \left(\alpha+ q_0\right)^{-M-1} V(q_0) dq_0 \\ 
		 & = \mO_M(N^{-M})  + \mO(1) (N\alpha)^{-M}\alpha ^{\kappa} \int_0^{C\alpha_1/\alpha}
		 \left(1+ x\right)^{-M-1} x^{\kappa} dx \\ 
		 & = \mO_M(1)\alpha^{\kappa},
\end{split}
\end{equation}
where in the last line we used that $N^{-1} \ll \alpha \ll 1$. Combining \eqref{ld.6} with (\ref{ld.6.1}-\ref{ld.9}) we 
obtain \eqref{prop:ld1.3}.
\\
\par
6. Finally, we turn back to \eqref{ld.5}. Using that 
\begin{equation*}
	\int\widetilde{\chi}\left(\frac{q_0}{\alpha}\right) d V(q_0)  
	\leq \mO(1)
	\int_0^{C\alpha} d V(q_0)  
	\leq 
	\mO(1) \alpha^{\kappa},
\end{equation*}
and taking $M=1$, $M'=2$, we deduce from \eqref{ld.5} in 
combination with \eqref{ld.9} that 
\begin{equation*}
\tr \psi\left( \frac{\widetilde{q}_{N,\alpha}}{\alpha}\right)
		 = 
		N^d \left( \int\psi\left(\frac{q_0}{\alpha}\right) d V(q_0)
		+ 
		\mO(N\alpha)^{-1}\alpha^{\kappa} \right),
\end{equation*}
which completes the proof of the proposition.
\end{proof}
\section{Grushin problem}\label{sec:GP}
We begin by giving a short overview on Grushin problems. For more 
general details see for instance the review \cite{SjZw07}. The central 
idea is to set up a problem of the form 
\begin{equation*}
 \begin{pmatrix}
  P(z) & R_- \\ 
  R_+ & 0\\
 \end{pmatrix}
 :
 \mathcal{H}_1\oplus \mathcal{H}_- 
 \longrightarrow \mathcal{H}_2\oplus \mathcal{H}_+,
\end{equation*}
where $P(z)$ is the operator under investigation and $R_{\pm}$ are 
suitably chosen so that the above matrix of operators is bijective. If 
$\dim\mathcal{H}_-  = \dim\mathcal{H}_+ < \infty$, one typically 
writes 
\begin{equation*}
 \begin{pmatrix}
  P(z) & R_- \\ 
  R_+ & 0\\
 \end{pmatrix}^{-1}
 =
 \begin{pmatrix}
  E(z) & E_{+}(z) \\ 
  E_{-}(z) & E_{-+}(z) \\
 \end{pmatrix}.
\end{equation*}
The key observation goes back to the Schur complement formula or, 
equivalently, the Lyapunov-Schmidt bifurcation method: 
the operator $P(z): \mathcal{H}_1 \rightarrow \mathcal{H}_2$ 
is invertible if and only if the finite dimensional matrix 
$E_{-+}(z)$ is invertible and when this is the case, 
we have 
\begin{equation*}
  P^{-1}(z) = E(z) - E_{+}(z) E_{-+}^{-1}(z) E_{-}(z), 
  \quad 
  E_{-+}^{-1} = - R_+ P^{-1}(z)R_-.
\end{equation*}
\subsection{Grushin Problem for the unperturbed operator}\label{sec:GP_s1}
In this section we will set up a Grushin problem as in  \cite{HaSj08}, using the left and 
right singular vectors of the unperturbed operator. 
We keep $h=1/(2\pi N)$, $0<N\in \N$ and $N^{-1} \ll \alpha \ll 1$. Let $p\in C^{\infty}(\T^{2d})$ with 
\begin{equation}\label{gp1}
	p \sim p_0 + hp_1 + \dots \text{ in } S(1,1),
\end{equation}
Until further notice we identify $\mathcal{H}_h^d \simeq \C^{N^d}$, as in \eqref{qf10}.  
By the discussion in Section \ref{sec:qf}, we have that $P=p_N$ is a bounded 
operator $\ell^2 \to\ell^2$. For $z\in \C$ let 
 \begin{equation}\label{gp2}
	0\leq t_1^2 \leq \dots \leq t_{N^d}^2 
\end{equation}
denote the eigenvalues of $Q=(P-z)^*(P-z)$ with an associated orthonormal basis 
of eigenfunctions $e_1,\dots,e_{N^d}\in \mathcal{H}_h^d$. 
\\
\par
Since $P$ is Fredholm 
of index $0$, the spectra of $Q$ and $Q' =(P-z)(P-z)^*$ are equal, and we can find 
an orthonormal basis $f_1,\dots,f_{N^d}$ of $\mathcal{H}_h^d$ comprised of 
eigenfunctions of $Q'$ associated with the eigenvalues \eqref{gp2}, such that 
 \begin{equation}\label{gp3}
	(P-z)^* f_i = t_i e_i, \quad (P-z)e_i = t_i f_i, \quad i=1,\dots, N^d.
\end{equation}
Indeed, let $f_1,\dots,f_{N_0}$ denote an orthonormal basis of the kernel 
$\mathcal{N}(P-z)^*$, and set $f_i = t_i^{-1}(P-z)e_i$, for $N_0 < i \leq N^d$. 
For $i> N_0$ the $f_i$ are well-defined due to the equality of the spectra 
of $Q$ and $Q'$ and since $t_i >0$. Moreover, one easily checks that they are orthonormal. Furthermore, 
for $0\leq i \leq N_0$ and $j>N_0$, we have that 
$( f_i| f_j) = ( (P-z)^*f_i | t_j^{-1} e_j) = (0| t_j^{-1} e_j) =0$. 
\\
\par
 Let $M>0$ be so that $0\leq t_1^2 \leq \dots \leq t_M^2 \leq \alpha$, and let 
 $\delta_i$, $1\leq i \leq M$, denote and orthonormal basis of $\C^M$. We suppose 
 that $Q$ satisfies assumption \eqref{prop:ld1.1}. Then, we know from \eqref{ld10} that 
  \begin{equation}\label{gp4.0}
		M = \mO(N^d\alpha^{\kappa}). 
\end{equation}
 It is clear from the proof of Proposition \ref{prop:ld1} that when $z$ varies in some 
 compact set $K\Subset \C$, and condition \eqref{prop:ld1.1} is assumed to be 
 uniform in $z\in K$ (compare with \eqref{i3}), then the estimate \eqref{gp4.0} is 
 uniform in $z\in K$. 
 \\
 \par 
 Continuing, we put
 \begin{equation}\label{gp4}
		R_+:  ~\mathcal{H}_h^d \longrightarrow \C^M : ~~u \longmapsto  \sum_1^M (u|e_i)\,\delta_i,
\end{equation}
and 
\begin{equation}\label{gp5}
		R_-:  ~\C^M  \longrightarrow  \mathcal{H}_h^d:~~u_- \longmapsto  \sum_1^M u_-(i)f_i,
\end{equation}
 where $u_-(i)=(u_-|\delta_i)$. The Grushin problem 
 \begin{equation}\label{gp6}
		\mathcal{P}(z) \defeq \begin{pmatrix} 
		P-z  & R_- \\ 
		R_+ & 0\\
		\end{pmatrix} :  \mathcal{H}_h^d \times \C^M \longrightarrow \mathcal{H}_h^d \times \C^M,
\end{equation}
is bijective with inverse $\mathcal{E}(z)$. Indeed, for given 
$(v,v_+)\in \mathcal{H}_h^d \times \C^M$ we want to solve 
 \begin{equation}\label{gp7}
		\mathcal{P}(z) 
		 \begin{pmatrix} u \\ u_- 
		\end{pmatrix} =
		 \begin{pmatrix} v \\ v_+
		\end{pmatrix}.
\end{equation}
Write $u= \sum_1^{N^d} u(j) e_j$, $v= \sum_1^{N^d} v(j) f_j$. Then we see by \eqref{gp3} 
that \eqref{gp7} is equivalent to 
\begin{equation*}
		\begin{cases}
			\sum_1^{N^d} t_i u_i f_i + \sum_1^M u_-(j)f_j = \sum_1^{N^d} v_j f_j \\ 
			u_j = v_+(j), \quad j =1, \dots, M,
		\end{cases}
\end{equation*}
which is equivalent to 
\begin{equation}\label{gp7.1}
		\begin{cases}
			\sum_{M+1}^{N^d} t_i u_i f_i  = \sum_{M+1}^{N^d} v_j f_j, \\[1.5ex] 
			 \begin{pmatrix} t_i & 1 \\ 1 & 0 \\
			\end{pmatrix}
			 \begin{pmatrix} u_i \\ u_-(i)\\
			\end{pmatrix}
			= 
			 \begin{pmatrix} v_i \\ v_+(i)
			\end{pmatrix}, \quad i=1,\dots, M.
		\end{cases}
\end{equation}
Since 
\begin{equation*}
	 \begin{pmatrix} t_i & 1 \\ 1 & 0 \\
	\end{pmatrix}^{-1} 
	=
	 \begin{pmatrix} 0& 1 \\ 1 & -t_i  \\
	\end{pmatrix},
\end{equation*}
it follows that
\begin{equation}\label{gp8}
		\mathcal{P}^{-1} (z)= \mathcal{E}(z) =
		\begin{pmatrix} 
		E(z)  & E_+(z)\\ 
		E_-(z) & E_{-+}(z)\\
		\end{pmatrix} 
\end{equation}
where 
\begin{equation}\label{gp9}
\begin{split}
		&E(z) = \sum_{M+1}^{N^d} \frac{1}{t_i} e_i \circ f_i, \quad 
	         E_+(z) = \sum_1^M e_i \circ \delta_i^*,  \\
		&E_-(z) =\sum_1^M  \delta_i\circ f_i^*, \quad
		 E_{-+}(z) = - \sum_1^M t_j \delta_j\circ\delta_j^*.
\end{split}
\end{equation}
Furthermore, 
\begin{equation}\label{gp10}
	\|E(z) \| \leq \frac{1}{\sqrt{\alpha}}, \quad \| E_{\pm } \| =1, \quad 
	\| E_{-+}\| \leq \sqrt{\alpha}.
\end{equation}
It follows from \eqref{gp7.1} that 
\begin{equation}\label{gp11}
	|\det \mathcal{P}(z) |^2= \prod_{M+1}^{N^d} t^2_i 
	= \alpha^{-M} \prod_1^{N^d} \mathbf{1}_{\alpha}(t^2_i), \quad 
	\mathbf{1}_{\alpha}(x)=\max(x,\alpha),
\end{equation}
which can be written as 
\begin{equation}\label{gp11.1}
	|\det \mathcal{P}(z) |^2 =
	\alpha^{-M} 
	 \det 	\mathbf{1}_{\alpha}(Q).
\end{equation}
The Schur complement formula applied to $\mathcal{P}$ and $\mathcal{E}$ yields
\begin{equation}\label{gp12}
	\log | \det (P-z)| = \log |\det \mathcal{P}(z)| + \log |\det E_{-+}(z)|.
\end{equation}
Next, we estimate $\log |\det \mathcal{P}(z)| $. Let $\chi\in C^{\infty}_c(\R)$ 
be supported in $[0,2]$, $0\leq \chi \leq 1$ and equal to $1$ on $[0,1]$. 
Then, for $x\geq 0$
\begin{equation}\label{gp13}
	x+ \frac{\alpha}{4}\chi\left( \frac{x}{\alpha}\right ) 
	\leq 
	 \mathbf{1}_{\alpha}(x)
	\leq 
	x + \alpha\chi\left( \frac{x}{\alpha}\right ).
\end{equation}
Similar to \eqref{qf11.0}, the principal symbol of $Q$ is $|p_0-z|^2$. Then, 
Proposition \ref{prop:ld1} together with \eqref{gp11.1}, \eqref{gp13}, \eqref{gp4.0}, yields that 
\begin{equation}\label{gp14}
\begin{split}
	\log |\det \mathcal{P}(z)| ^2 
	&= \log \det \left(\mathbf{1}_{\alpha}(Q) \right ) + M\log\frac{1}{\alpha} \\
	&= N^d \left( \int_{\T^{2d}}  \log |p_0(\rho)-z|^2  \, d\rho + 
		 	 \mO\!\left(\alpha^{\kappa}\log \frac{1}{\alpha}\right)\right).
\end{split}
\end{equation}
\subsection{Grushin Problem for the perturbed operator}\label{sec:GPP}
Let $Q_{\omega}:\mathcal{H}_h^d \to \mathcal{H}_h^d $ be a linear operator 
(i.e. an $N^d\times N^d$ matrix since $\mathcal{H}_h^d\simeq \C^{N^d}$). In this section $Q_{\omega}$ 
can be considered to be deterministic, its randomness will only be important 
later on. We are interested in studying the eigenvalues of 
\begin{equation}\label{gpp2}
	P^{\delta}  = P + \delta Q_{\omega}, \quad 0\leq \delta \ll 1,
\end{equation}
where $P=p_N$ as in Section \ref{sec:GP_s1}. For this purpose 
we set up a Grushin problem for the perturbed operator. 
\\
\par
Let $R_{\pm}$ be as in \eqref{gp4}, \eqref{gp5}, and for $z\in\C$ put
 \begin{equation}\label{gpp3}
		\mathcal{P}^{\delta}(z) \defeq \begin{pmatrix} 
		P^{\delta}-z  & R_- \\ 
		R_+ & 0\\
		\end{pmatrix} :  \mathcal{H}_h^d \times \C^M \longrightarrow \mathcal{H}_h^d \times \C^M.
\end{equation}
Suppose that for some $q>0$ and sufficiently large $C_1>0$ 
\begin{equation}\label{gpp3.1}
	\| Q_{\omega} \| \leq C_1 N^q,
\end{equation}
and suppose that $\delta \geq 0$ is such that
\begin{equation}\label{gpp4}
	\delta \| Q_{\omega} \| \alpha^{-1/2}\leq \frac{1}{2}.
\end{equation}
Using \eqref{gp8} and a Neumann series argument, we see that 
\begin{equation}\label{gpp5}
	\mathcal{P}^{\delta}\mathcal{E} = 1 + 
	\begin{pmatrix} 
		\delta Q_{\omega} E & \delta Q_{\omega}E_+ \\ 
		0 & 0\\
	\end{pmatrix}
\end{equation}
is bijective with inverse 
\begin{equation}\label{gpp6}
	(\mathcal{P}^{\delta}\mathcal{E} )^{-1}
	= 1 + \sum_1^{\infty}(-\delta)^n
	\begin{pmatrix} 
		(Q_{\omega} E)^n & (Q_{\omega}E)^{n-1} Q_{\omega}E_+ \\ 
		0 & 0\\
	\end{pmatrix},
\end{equation}
of norm $\leq \mO(1)$. Thus, 
$\mathcal{P}^{\delta}$ is bijective with inverse 
 \begin{equation}\label{gpp7}
 \begin{split}
		\mathcal{E}^{\delta}(z)
		=\mathcal{E} (\mathcal{P}^{\delta}\mathcal{E} )^{-1}
		&=\mathcal{E} + 
		 \sum_1^{\infty}(-1)^n 
		\begin{pmatrix} 
		E(\delta Q_{\omega} E)^n & (\delta EQ_{\omega})^{n}E_+ \\ 
		E_-(\delta Q_{\omega}E)^{n}  & E_-(\delta Q_{\omega}E)^{n-1}(\delta Q_{\omega})E_+\\
		\end{pmatrix}
		\\
		& \defeq \begin{pmatrix} 
		E^{\delta}(z)  & E^{\delta}_+(z)\\ 
		E^{\delta}_-(z)& E^{\delta}_{-+}(z)\\
		\end{pmatrix},
\end{split}
\end{equation}
where by \eqref{gpp4}, \eqref{gp10}, 
 \begin{equation}\label{gpp8}
 \begin{split}
		&\| E^{\delta} \| = \|  E( 1+ \delta Q_{\omega}E)^{-1} \| \leq 2 \|E\| \leq 2 \alpha^{-1/2}, \\
		&\| E_+^{\delta} \| = \|  ( 1+ \delta Q_{\omega}E)^{-1}E_+ \| \leq 2 \|E_+\| \leq 2,  \\
		&\| E_-^{\delta} \| = \| E_- ( 1+ \delta Q_{\omega}E)^{-1} \| \leq 2\|E_-\| \leq 2,  \\
		&\| E_{-+}^{\delta} -E_{-+}\| = 
		\| E_- ( 1+ \delta Q_{\omega}E)^{-1}\delta Q_{\omega}E_+ \| \leq 2 \|\delta Q_{\omega}\| \ll 1.  \\
\end{split}
\end{equation}
By the Schur complement formula applied to $\mathcal{P}^{\delta}$ and $\mathcal{E}^{\delta}$, 
we see that 
 \begin{equation}\label{gpp8.2a} 
	\log |\det (P^{\delta}-z) | = \log |\det \mathcal{P}^{\delta}(z)| + \log |\det E_{-+}^{\delta}(z)|.
\end{equation}
Since $\frac{d}{d\delta} \log \det \mathcal{P}^{\delta} 
= \tr \mathcal{E}^{\delta}\frac{d}{d\delta} \mathcal{P}^{\delta}$, we see by 
\eqref{gpp3}, \eqref{gpp7}, \eqref{gpp8}, that
 \begin{equation}\label{gpp9}
 \begin{split}
 	\left| \log |\det \mathcal{E}^{\delta}| - \log |\det \mathcal{E}| \right| 
	&=
	\left| \log |\det \mathcal{P}^{\delta}| - \log |\det \mathcal{P}| \right| \\
	&= \left| \Rea \int_0^{\delta} \tr\left({E}^{\tau}Q_{\omega}\right)d\tau \right| 
	\leq \mO( \delta \alpha^{-1/2} \| Q\|_{\tr}).
\end{split}
\end{equation}
Since $\| Q\|_{\tr} \leq N^{d/2} \| Q\|_{\mathrm{HS}}$, we get by combining \eqref{gpp9} with \eqref{gp14}, 
that 
 \begin{equation}\label{gpp10}
	\log |\det \mathcal{P}^{\delta}(z)| 
	= N^d \left( \int_{\T^{2d}}  \log |p_0(\rho)-z|  \, d\rho + 
		 	 \mO(\alpha^{\kappa}\log \frac{1}{\alpha})+\mO( \delta N^{-d/2}\alpha^{-1/2} \| Q\|_{\mathrm{HS}})\right).
\end{equation}
Under the assumption \eqref{gpp4} 
we have by \eqref{gp10}, \eqref{gpp8}, that 
$\| E^{\delta}_{-+}\| \leq \alpha^{1/2} + \mO(\| \delta Q_{\omega}\| )$, which in view of 
\eqref{gp4.0} and \eqref{gpp4} yields the following upper bound 
 \begin{equation}\label{gpp11}
	\log |\det E_{-+}^{\delta}(z)| \leq \mO( N^d\alpha^{\kappa}) | \log\alpha|.
\end{equation}
We end this section with a general result on the singular values of Grushin problems. 
\begin{lem}\label{gpp:lem1}
	Let $\mathcal{H}$ be an $N$-dimensional complex Hilbert spaces, and 
	let $N\geq M>0$. Suppose that 	
	 \begin{equation*}
		\mathcal{P} = \begin{pmatrix} 
		P & R_- \\ 
		R_+ & 0\\
		\end{pmatrix} :  \mathcal{H} \times \C^M \longrightarrow  \mathcal{H} \times \C^M
	\end{equation*}
	is a bijective matrix of linear operators, with inverse 
	 \begin{equation*}
		\mathcal{E}
		=
		 \begin{pmatrix} 
		E  & E_+\\ 
		E_-& E_{-+}\\
		\end{pmatrix}.
	\end{equation*}
	Let $0 \leq t_1(P) \leq \dots \leq t_N(P)$ denote the eigenvalues of $(P^*P)^{1/2}$, and 
	let $0 \leq t_1(E_{-+}) \leq \dots \leq t_M(E_{-+})$ denote the eigenvalues of $(E_{-+}^*E_{-+})^{1/2}$. 
	Then, 
	\begin{equation*}
			\frac{t_n(E_{-+})}{\| E\| t_n(E_{-+}) + \|E_-\| \|E_+\|}
			\leq  t_n(P) \leq \|R_+\| \|R_-\| t_n(E_{-+}), \quad 1\leq n \leq M.
	\end{equation*}
\end{lem}
\begin{rem}\label{rem:NotSV}
	Before we present the proof of Lemma \ref{gpp:lem1}, let us 
	comment on some notation. Let $A$ be a trace-class operator 
	and Fredholm of index $0$ acting on a complex separable Hilbert 
	space $\mathcal{H}$. We shall denote by $0\leq t_1(A) \leq t_2(A) \leq \dots$ the 
	increasing sequence of eigenvalues of $(A^*A)^{1/2}$ and 
	by $s_1(A) \geq s_2(A) \geq \dots $ the decreasing sequence 
	of eigenvalues of $(A^*A)^{1/2}$. The latter are called the 
	singular values of $A$. We have that$s_n(A) = s_n(A^*)$  and 
	 $t_n(A) = t_n(A^*)$. When $\dim \mathcal{H} = D < \infty$, 
	then these sequences are finite, and we have that $s_{D-n+1} = t_n$. 
\end{rem}

\begin{proof}[Proof of Lemma \ref{gpp:lem1}]
	We know from the Schur complement formula applied to $\mathcal{P}$
	and $\mathcal{E}$, that $P$ is invertible if and only if 
	$E_{-+}$ is invertible, and that in this case 
	 \begin{equation}\label{gpp12}
		P^{-1} = E - E_+ E_{-+}^{-1} E_-, \quad 
		E_{-+}^{-1} = - R_+ P^{-1} R_-. 
	\end{equation}
	Let $0 \leq s_N(P) \leq \dots \leq s_1(P)$ denote the singular values of $P$, and 
	let $0 \leq s_M(E_{-+}) \leq \dots \leq s_1(E_{-+})$ denote the singular values of $E_{-+}$. 
	Notice that $s_n(P) = t_{N-n+1}(P)$ and similarly, $s_n(E_{-+}) = t_{M-n+1}(E_{-+})$. 
	Suppose first that $P$ is invertible. Then, since $s_n(P)  = s_n(P^*) $, we have that 
	\begin{equation}\label{gpp13}
			s_n(P^{-1}) = \frac{1}{t_n(P)}, \quad 1 \leq n \leq N.
	\end{equation}
	and similarly, 
	\begin{equation}\label{gpp14}
			s_n(E_{-+}^{-1}) = \frac{1}{t_n(E_{-+})}, \quad 1 \leq n \leq M. 
	\end{equation}
	We recall from \cite{GoKr69} that if $A,B$ are trace-class operators, then we have the following general 
	estimates 
	\begin{equation}\label{gpp15}
	\begin{split}
			&s_{n+k-1}(A+B) \leq s_n(A) + s_k(B) ,\\ 
			&s_{n+k-1}(AB) \leq s_n(A)s_k(B).
	\end{split}
	\end{equation}
	Using that $s_1(A) = \| A\|$, it follows from the first equation in \eqref{gpp12} 
	in combination with \eqref{gpp15} that 
	\begin{equation}\label{gpp16}
		s_n(P^{-1}) \leq \| E\| + \| E_+ \| \|E_-\| s_n(E_{-+}^{-1}), 
		\quad 1 \leq n \leq M.
	\end{equation}
	By \eqref{gpp13}, \eqref{gpp14} we then get 
	\begin{equation}\label{gpp17}
		t_n(P) \geq  \frac{t_n(E_{-+})}{\| E\| t_n(E_{-+})+ \| E_+ \| \|E_-\| }, 
		\quad 1 \leq n \leq M.
	\end{equation}
	Similarly, we get from the second equation in \eqref{gpp12} that 
	\begin{equation}\label{gpp18}
		t_n(P) \leq   \| R_+ \| \|R_-\| t_n(E_{-+}), 
		\quad 1 \leq n \leq M.
	\end{equation}
	When replacing $P$ in $\mathcal{P}$ by $P_{\varepsilon} = P+ \varepsilon X$, $\|X\| \leq 1$, 
	$0\leq \varepsilon \ll 1$, a small perturbation of $P$, we see by a Neumann series argument 
	that the perturbed Grushin problem $\mathcal{P}^{\varepsilon}$ remains invertible with inverse 
	$\mathcal{E}^{\varepsilon}$. Since the  singular values of $P_{\varepsilon}$ and $E^{\varepsilon}_{-+}$ 
	depend continuously on $\varepsilon$, we see that \eqref{gpp17} and \eqref{gpp18} hold 
	even when $P$ is not invertible. 
\end{proof}

\section{Perturbation by a random matrix}\label{sec:PRM}
In this section we consider perturbations of $p_N$ by two 
random matrix ensembles $Q_{\omega}$. We will begin with 
the case when $Q_{\omega}$ is of the complex Ginibre ensemble, 
i.e. a complex Gaussian random matrix whose entries are independent and 
identically distributed (iid). Then, we will 
consider a more general ensemble of random matrices 
with iid entries with mean $0$, variance $1$ and bounded fourth moment. 
The aim of this section is to estimate the probability 
that $|\det E_{-+}^{\delta}|$ is small. 
\\ 
\par 
In this section we let $K\Subset \C$ be an open connected 
relatively compact set. Let $z\in K$, let $V_z(t)$ be as in 
\eqref{i2} and suppose that 
\begin{equation}\label{gc0}
		\exists \kappa \in ]0,1], \text{ such that } V_z(t) = \mO(t^{\kappa}), 
		\text{ uniformly for } z \in K, ~ 0 \leq t \ll 1.
\end{equation}
\par
Our principal aim is to find probabilistic lower bounds on $\log |\det E_{-+}^{\delta}(z)|$. 
\subsection{The Gaussian case}
We continue to identify $\mathcal{H}_h^d\simeq \C^{N^d}$, as in Section \ref{sec:GP}. 
Let $Q_{\omega}$ be a complex Gaussian random matrix with independent 
and identically distributed (iid) entries, i.e. 
\begin{equation}\label{gc1}
	Q_{\omega} = (q_{i,j}(\omega))_{1\leq i,j\leq N^d}, \quad 
	q_{i,j}(\omega) \sim \mathcal{N}_{\C}(0,1) ~(iid).
\end{equation}
In other words, let $\mathcal{H}_N \defeq (\C^{N^d\times N^d}, \| \cdot \|_{\mathrm{HS}})$ 
of $N^d\times N^d$ be the space of complex valued matrices equipped with the Hilbert-Schmidt norm, 
which we equip with the Gaussian probability measure 
\begin{equation}\label{gc2}
	\mu_N (dQ) \defeq \pi^{-N^2} \e^{- \| Q\|_{\mathrm{HS}}^2} L(dQ), 
\end{equation} 
where $L(dQ)$ denotes the Lebesgue measure on $\mathcal{H}_N$. 
For $C_1>0$, let $\mathcal{Q}_{C_1N}\subset \mathcal{H}_N$ be the 
subset where 
\begin{equation}\label{gc3}
	\| Q_{\omega}\| \leq C_1 N^{\frac{d}{2}}.
\end{equation} 
Since $Q_{\omega}$ is Gaussian, we know (see e.g. \cite{Ta12,TaVu10}) that 
for $C_1>0$ is sufficiently large, 
\begin{equation}\label{gc4}
\prob\left(
\Vert Q_{\omega}\Vert\le C_1N^{d/2}
\right)
= \mu_N(\mathcal{Q}_{C_1N}) \ge 1-e^{-N^d}.
\end{equation}
We restrict our attention to $\|Q_{\omega}\|\leq C_1 N^{d/2}$ 
and we assume that 
\begin{equation}\label{r4}
0< \delta \ll N^{-d/2}\alpha^{1/2}.
\end{equation}
Then \eqref{gpp4} is satisfied and it follows from the discussion in Section \ref{sec:GPP} that 
the Grushin problem $\mathcal{P}^{\delta}$ \eqref{gpp3} 
is bijective with inverse $\mathcal{E}^{\delta}$ \eqref{gpp7}, and the estimates of Section \ref{sec:GPP} 
apply. In particular, we have by \eqref{gpp10}, \eqref{gpp11} in combination with \eqref{gc4}, that 
 \begin{equation}\label{r5}
	\log |\det \mathcal{P}^{\delta}(z)| 
	= N^d \left( \int_{\T^{2d}}  \log |p_0(\rho)-z|  \, d\rho + 
		 	 \mO\!\left(\alpha^{\kappa}\log \frac{1}{\alpha}\right)+\mO(\delta  N^{\frac{d}{2}}\alpha^{-\frac{1}{2}})\right),
\end{equation}
and that 
 \begin{equation}\label{r6}
	\log |\det E_{-+}^{\delta}(z)| \leq \mO( N^d\alpha^{\kappa}|\log(\alpha )|),
\end{equation}
with probability $\geq 1 - \e^{-N^d}$. Thus, with the same probability, we have 
in view \eqref{gpp8.2a} that 
 \begin{equation}\label{r6.1}
	\log |\det (P^{\delta}-z)| \leq N^d \left( \int_{\T^{2d}}  \log |p_0(\rho)-z|  \, d\rho + 
		 	 \mO\!\left(\alpha^{\kappa}\log \frac{1}{\alpha}\right)+\mO(\delta N^{\frac{d}{2}}\alpha^{-\frac{1}{2}})\right).
\end{equation}
From Theorem \ref{thm:SSV} below (a complex version of \cite[Lemma 3.2]{SaTeSp06}), 
we know that there exists a constant $C>0$ such that every $N\geq 2$, and all $t> 0$
	\begin{equation}\label{r7}
	\textbf{P}\left( s_{N^d}( P + \delta Q_{\omega} -z) \leq \delta t\right) 
	\leq  C N^{d} t^2.
	\end{equation}	
Notice that here the constant $C>0$ is uniform in $z\in K$. 
If $\|Q_{\omega}\|\leq C_1 N^{d/2}$ and \eqref{gc4} holds, 
then the Grushin Problem \eqref{gpp3} is bijective, and we then know from 
Lemma \ref{gpp:lem1} that 
\begin{equation}\label{r7.00}
	s_{N^d}(P-z+\delta Q_{\omega})= t_{1}(P - z+ \delta Q_{\omega}) \leq t_1(E_{-+}^{\delta}(z)) 
	=s_{M}(E_{-+}^{\delta}(z)) 
\end{equation}	
Hence, by combining \eqref{r7} and \eqref{gc4}, we get that 
	\begin{equation}\label{r7.0}
	\textbf{P}\left( s_{M}( E_{-+}^{\delta}(z)) > \delta t 
	 \text{ and } \Vert Q_{\omega}\Vert \leq C_1 N^{d/2} \right) 
	\geq 1 - C N^{d} t^2- \e^{-N^d}.
	\end{equation}	
We are interested in the regime when $0 < \delta, t \ll 1$. Therefore, 
supposing that the event \eqref{r7.0} holds, we have 
\begin{equation}\label{r7.2}
\begin{split}
\log |\det E_{-+}^{\delta} | & = \sum_1^M \log s_j(E_{-+}^{\delta}) \\ 
& \geq M \log s_M(E_{-+}^{\delta})  \\
& \geq - M \log (\delta t ) ^{-1} \\
&\geq -C N^d \alpha^{\kappa} \log (t \delta)^{-1},
\end{split}
\end{equation}
here we used as well \eqref{gp4.0}. 
Summing everything up so far, we have proven
\begin{prop}\label{prop:gc3}
Let $N\geq 2$, let $K\Subset \C$ be an open connected relatively compact set, let $z\in K$, 
let $N^{-1} \ll \alpha \ll 1$ and suppose that \eqref{gc0} and \eqref{r4} hold. Then, the 
following holds uniformly for $z\in K$ : There exists a positive constant $C>0$ such that for any 
$0 < t \ll 1$ 
\begin{equation*}
\textbf{P}\left( \log |\det E_{-+}^{\delta}(z) |  \geq  
	-C N^d \alpha^{\kappa}  \log (t \delta )^{-1}
	\text{ and } \Vert Q_{\omega}\Vert \leq C_1 N^{d/2} \right) 
	\geq  1 - C N^{d} t^2- \e^{-N^d}.
\end{equation*}
\end{prop}
We recall that $M=\mO(N^{d}\alpha^{-\kappa})$. 
Combining Proposition \ref{prop:gc3} with \eqref{gpp8.2a} and \eqref{r5}, we obtain 
that 
 \begin{equation}\label{gc12.1}
 \begin{split}
	\log |\det &(P^{\delta}-z)| \\
	&\geq N^d \left( \int_{\T^{2d}}  \log |p_0(\rho)-z|  \, d\rho + 
		 	 \mO\!\left(\alpha^{\kappa}\log \frac{1}{\alpha}\right)+\mO(\delta N^{\frac{d}{2}}\alpha^{-1/2}) 
			 -C\alpha^{\kappa}  \log (t \delta )^{-1}\right).
\end{split}
\end{equation}
with probability 
 \begin{equation}\label{gc12.2}
\geq 1 - C N^{d} t^2- \e^{-N^d}.
\end{equation}
when
\begin{equation}\label{gc12.3}
	0 < t \ll1, \quad  0<\delta \ll N^{-d/2}\alpha^{1/2}.
\end{equation}
\subsection{The universal case}
We continue to identify $\mathcal{H}_h^d\simeq \C^{N^d}$, as in Section \ref{sec:GP}. 
Now, we consider the random matrix 
\begin{equation}\label{gc2.1}
	Q_{\omega} = (q_{i,j}(\omega))_{1\leq i,j\leq N^d}
\end{equation}
whose entries $q_{i,j}(\omega)$ are independent copies of a random variable 
$q$ satisfying the moment conditions 
\begin{equation}\label{gc2.1.0.1}
	\erw[ q ] =0, \quad  \erw[|q|^2] =1, \quad \erw[ |q|^4] < +\infty.
\end{equation}
We are interested in the eigenvalues of 
\begin{equation}\label{gc2.1.0}
	P^{\delta} = P + \delta Q_{\omega}, \quad 0 < \delta \ll1.
\end{equation}
In this section we assume  that for some sufficiently large constant $C>0$  
\begin{equation}\label{gc2.1.1}
	\delta = \frac{1}{C}N^{ -d/2 - \delta_0 }, \quad \text{ with some fixed } \delta_0 >0.
\end{equation}
Furthermore, we set for some arbitrary but fixed $ \tau \in ]0,1[$
\begin{equation}\label{gc2.1.2}
	\alpha = N^{- \min( \delta_0, 1) \tau }.
\end{equation}
Form \cite{La05} we know that \eqref{gc2.1.0.1} implies that $\erw [ \| Q_{\omega}\| ] \leq C N^{d/2}$, which   
using Markov's inequality, yields that 
\begin{equation}\label{gc2.2}
\prob\left[ 
\Vert Q_{\omega}\Vert \geq C N^{d/2+( 1-\tau) \delta_0}
\right] 
\leq C^{-1}N^{-d/2- ( 1-\tau) \delta_0}\erw [ \| Q_{\omega}\| ] 
\leq N^{-( 1-\tau) \delta_0}.
\end{equation}
Suppose that $\|Q_{\omega}\| \leq CN^{d/2+(1-\tau)\delta_0}$. 
Then, 
\begin{equation}\label{gc2.2.1}
	\delta \alpha^{-1/2} \|Q\| \ll N^{-\tau \delta_0 /2}.
\end{equation}
Then \eqref{gpp4} is satisfied and it follows from the discussion in Section \ref{sec:GPP} that 
the Grushin problem $\mathcal{P}^{\delta}$ \eqref{gpp3} 
is bijective with inverse $\mathcal{E}^{\delta}$ \eqref{gpp7}, and the estimates of Section \ref{sec:GPP} 
apply. In particular, we have by \eqref{gpp10}, \eqref{gpp11} in combination with \eqref{gc4}, that 
 \begin{equation}\label{gc2.3}
	\log |\det \mathcal{P}^{\delta}(z)| 
	= N^d \left( \int_{\T^{2d}}  \log |p_0(\rho)-z|  \, d\rho + 
		 	 \mO(N^{-\min(\delta_0,1)\tau \kappa}\log N)+\mO(N^{-\tau \delta_0 /2})\right),
\end{equation}
and
 \begin{equation}\label{gc2.4}
	\log |\det E_{-+}^{\delta}(z)| \leq \mO( N^{d-\tau\kappa\min(\delta_0,1)}\log N),
\end{equation}
with probability $\geq 1 -N^{-( 1-\tau) \delta_0}$. Here we also used that 
$\| Q_{\omega}\|_{\mathrm{HS}} \leq N^{d/2}\| Q_{\omega}\|$. Using \eqref{gpp8.2a}, we have that
 with the same probability,  
 \begin{equation}\label{gc2.4.1}
	\log |\det (P^{\delta}-z)| \leq N^d \left( \int_{\T^{2d}}  \log |p_0(\rho)-z|  \, d\rho + 
		 	 \mO(N^{-\min(\delta_0,1)\tau \kappa}\log N)+\mO(N^{-\tau \delta_0 /2})\right).
\end{equation}
\begin{prop}\label{prop:gc2}
Let $K\Subset \C$ be an open connected relatively compact set, let $z\in K$, 
suppose that \eqref{gc0}, \eqref{gc2.1.1} and \eqref{gc2.1.2} hold. Then, the 
following holds uniformly for $z\in K$: 
There exist a positive constants $\beta,C>0$ such that for all $\tau \in ]0,1[$ 
\begin{equation*}
\textbf{P}\left( \log |\det E_{-+}^{\delta} (z)|  \geq  
	-\beta N^{d- \kappa\tau \min( \delta_0, 1) }\log N
	\text{ and } \Vert Q_{\omega}\Vert \leq C N^{d/2+( 1-\tau) \delta_0} \right) 
	\geq  1 - N^{-( 1-\tau) \delta_0}.
\end{equation*}
\end{prop}
\begin{proof}
	By \cite[Theorem 3.2]{TaVu10} we have that for any $\gamma \geq 1/2$ and $A\geq 0$ 
	there exists a constant $C>0$ such that for any deterministic $n\times n$ matrix $M$ 
	with $\| M_n\| \leq n^{\gamma}$, 
	\begin{equation}\label{gc2.4.2}
	\textbf{P}\left( s_n(M_n + N_n) \leq n^{\gamma (2A+2) +1/2} \right) 
	\leq  C \left( n^{-A + o(1)} + \prob( \| N_n\| \geq n^{\gamma})\right).
	\end{equation}
	where $N_n$ is a random matrix of size $n$ whose entries are iid copies of $q$ \eqref{gc2.1.0.1}, 
	and $s_n(M_n + N_n)$ denotes $n$-th (i.e. the smallest) singular value of $M_n + N_n$.
	\\
	\par
	We apply this result to $\delta^{-1}(P_{\delta}-z )= \delta^{-1}( P-z) + Q_{\omega}$, cf. \eqref{gc2.1.0}. 
	By Proposition \ref{prop:qf2} and \eqref{gc2.1.1} we have that for all $z\in K$
	$$\| \delta^{-1} (P-z )\| \leq  N^{d/2 + \delta_0} = (N^d)^{\gamma}, \quad \gamma = 1/2+\delta_0/d.$$
	Similarly to \eqref{gc2.2}, we have that 
	\begin{equation}\label{gc2.5}
	\prob\left(
	\Vert Q_{\omega}\Vert \geq N^{d\gamma}
	\right) 
	\leq C N^{-\delta_0}.
	\end{equation}
	Thus, applying \eqref{gc2.4.2} with $n=N^d$, $M_n = \delta^{-1}( P-z)$, $N_n = Q_{\omega}$, 
	$A=1/2+\delta_0$, we conclude that there exists a constant $C>0$, independent of $z\in K$, 
	such that for $N\geq 1$ 
	\begin{equation}\label{gc2.6}
	\textbf{P}\left( s_{N^d}(\delta^{-1} (P-z) + Q_{\omega}) \leq (N^d)^{-\gamma (2A+2) +1/2} \right) 
	\leq  C N^{-\delta_0}.
	\end{equation}
	Since $s_{N^d}((P-z) + \delta Q_{\omega})  = \delta s_{N^d}(\delta^{-1} (P-z) + Q_{\omega})$, we get  
	that  
	\begin{equation}\label{gc2.7}
	\textbf{P}\left( s_{N^d}(P-z + \delta Q_{\omega}) \leq C^{-1} N^{-(d/2 + \delta_0) (2\delta_0 +3) -\delta_0} \right) 
	\leq  C N^{-\delta_0},
	\end{equation}
	after potentially increasing the constant $C>0$ in \eqref{gc2.6}, uniformly in $z\in K$.
	\\
	\par
	Let $0 \leq t_1(P_{\delta}-z) \leq \dots \leq t_{N^d}(P_{\delta}-z)$ denote the eigenvalues of 
	$((P_{\delta}-z)^*(P_{\delta}-z))^{1/2}$, $P_{\delta} =P + \delta Q_{\omega}$, and recall 
	from \eqref{gp4}, \eqref{gp5}, that $\|R_{\pm}\| =1$. Suppose 
	$\Vert Q_{\omega}\Vert \leq C N^{d/2+( 1-\tau) \delta_0}$, with $\tau$ as in \eqref{gc2.2}. 
	Recall the estimates given in \eqref{gpp8}, 
	\eqref{gp10}, which together with \eqref{gc2.2}, \eqref{gc2.1.1}, \eqref{gc2.1.2} yield that 
	\begin{equation}\label{gc2.7.1}
		 \| E^{\delta}\| \|E^{\delta}_{-+}\| \leq 2 \alpha^{-1/2} ( \alpha^{1/2}  + N^{-\tau\delta_0}) 
		 \leq  2(1+ N^{-\tau \delta_0/2}).
	\end{equation}
	Lemma \ref{gpp:lem1} then shows that
	\begin{equation}\label{gc2.8}
	\frac{ t_1(E_{-+}^{\delta}(z))}{C} \leq   t_1(P_{ \delta}-z) \leq  t_1(E_{-+}^{\delta}(z)),
	\end{equation}
	 for all $z\in K$. Since $s_{N^d}(P_{ \delta}-z) = t_1(P_{ \delta}-z)$ 
	and $s_{M}(E_{-+}^{\delta}(z)) = t_1(E_{-+}^{\delta}(z))$, we deduce from  \eqref{gc2.7}, \eqref{gc2.2} 
	and \eqref{gc2.8} that for $N\geq 1 $ 
	\begin{equation}\label{gc2.9}
	\begin{split}
	\textbf{P}\big( s_{M}(E_{-+}^{\delta}(z)) > C^{-1} N^{-(d/2 + \delta_0) (2\delta_0 +3) -\delta_0} 
	\text{ and } \Vert Q_{\omega}\Vert& \leq C N^{d/2+( 1-\tau) \delta_0} \big) \\
	&\geq  1 - C N^{-\delta_0}- N^{-( 1-\tau) \delta_0}.
	\end{split}
	\end{equation}
	Notice that the constants here are uniform in $z\in K$. 
	Supposing that the above event holds, we get by \eqref{gp4.0}, \eqref{gc2.1.2}, that 
	\begin{equation}\label{gc2.10}
	\begin{split}
		\log |\det E_{-+}^{\delta} (z)| &= \sum_1^M \log s_j(E_{-+}^{\delta}(z) )\\ 
			&\geq M \log s_M(E_{-+}^{\delta} (z)) \\ 
			& \geq -\beta N^{d- \kappa 
			\tau \min( \delta_0, 1) }\log N
	\end{split}
	\end{equation}
	for some $\beta >0$ uniform in $z\in K$, and the claim of the proposition follows. 
\end{proof}
By \eqref{gpp8.2a}, it follows from Proposition \ref{prop:gc2} and \eqref{gc2.3} that there 
exists a constant $C>0$ such that for all $z\in K$ 
 \begin{equation}\label{gc2.11}
	\log |\det (P^{\delta}-z)| \geq N^d \left( \int_{\T^{2d}}  \log |p_0(\rho)-z|  \, d\rho 
		 	-CN^{-\min(\delta_0,1)\tau \kappa}\log N -CN^{-\tau \delta_0 /2}
			\right).
\end{equation}
with probability $\geq  1 - N^{-( 1-\tau) \delta_0}$.
\section{Counting eigenvalues}
\subsection{Counting zeros of holomorphic functions of exponential growth}
\label{sec:c1}
We recall Theorem $1.1$ in \cite{Sj09b} (in a form somewhat adapted to our formalism) which 
gives an estimate on the number of zeros of holomorphic functions with exponential growth 
in certain domains with Lipschitz boundary. Different versions of this theorem have been 
been proven also in \cite{Ha06,HaSj08}. This presentation has been taken from \cite{SjVo19a}.
\\
\par
\subsubsection{Domains with associated Lipschitz weight}\label{sec:LipBD}
Let $N\geq1$ be a large parameter, and let $\Omega \Subset \C$ be an open 
simply connected set with Lipschitz boundary $\omega =\partial\Omega$ which may 
depend on $N$. More precisely, we assume that $\partial\Omega$ is Lipschitz 
with an associated Lipschitz weight $r:\omega \to ]0,+\infty[$, which is 
a Lipschitz function of modulus $\leq 1/2$, in the following way : 
\par
There exists a constant $C_0>0$ such that for every $x\in\omega$ 
there exist new affine coordinates $\widetilde{y} = (\widetilde{y}_1,\widetilde{y}_2)$ 
of the form $\widetilde{y} = U(y-x)$, $y\in \C \simeq \R^2$ being the old coordinates, 
where $U=U_x$ is orthogonal, such that the intersection of $\Omega$ and the 
rectangle $R_x := \{ y\in\C; |\widetilde{y}_1| < r(x), |\widetilde{y}_2| < C_0r(x)\}$ 
takes the form 
\begin{equation}\label{eq.LD1}
	\{ y \in R_x; ~\widetilde{y}_2 > f_x(\widetilde{y}_1), |\widetilde{y}_1| < r(x)\},
\end{equation}
where $f_x(\widetilde{y}_1)$ is Lipschitz on $[-r(x),r(x)]$, with Lipschitz modulus 
$\leq C_0$. Notice that \eqref{eq.LD1} remains valid if we shrink the weight function $r$.
\subsubsection{Thickening of the boundary and choice of points}
Define 
\begin{equation}\label{eq.LD1.2}
	\widetilde{\omega}_r = \bigcup_{x\in \omega}D(x,r(x))
\end{equation}
and let $z_j^0 \in \omega$, $j\in \Z/ N_0\Z$, with $N_0\in \N$ which may depend on $N$, 
be distributed along the boundary in the positively oriented sense such that 
\begin{equation}\label{eq.LD1.1}
	r(z_j^0)/4 \leq |z_{j+1}^0 - z_j^0| \leq r(z_j^0)/2.
\end{equation}
\begin{thm}[Theorem 1.1 in \cite{Sj09b}]\label{thm:Count}
	Let $C_0>0$ be as in 1) above. 
	There exists a constant $C_1>0$, depending only on $C_0$, 
	such that if $z_j \in D(z_j^0, r(z_j^0)/(2C_1))$ we have the 
	following :
	\par
	Let $N\geq 1$ and let $\phi $ be a continuous subharmonic function on 
	$\widetilde{\omega}_r$ with a distributional extension to 
	$\Omega\cup \widetilde{\omega}_r$, denoted by the same symbol. Then, 
	there exists a constant $C_2>0$ such that if $u$ is a holomorphic function 
	on $\Omega\cup \widetilde{\omega}_r$ satisfying 
	\begin{equation}\label{th:UB}
		\log |u | \ \leq N^d \phi \hbox{ on } \widetilde{\omega}_r,
	\end{equation}
	\begin{equation}\label{th:LB}
		\log |u (z_j)| \ \geq N^d (\phi(z_j) - \varepsilon_j), \hbox{ for } 
		j=1, \dots , N_0,
	\end{equation}
	where $\varepsilon_j \geq 0$, then the number of zeros of $u$ in $\Omega$
	satisfies
	\begin{equation}\label{th:c}
	\begin{split}
		\bigg|\#(u^{-1}(0)\cap \Omega) &- \frac{N^d}{2\pi} \mu(\Omega)\bigg| \\
		&\leq C_2 N^d \left(
		\mu(\widetilde{\omega}_r) + \sum_{j=1}^{N_0}
		\left( \varepsilon_j 
		+ \int_{D\!\left(z_j,\frac{r(z_j)}{4C_1}\right)}\left| \log \frac{|w-z_j|}{r(z_j)} \right |
		\mu(dw)
		\right)
		\right).
	\end{split}
	\end{equation}
	Here $\mu \defeq \Delta\phi \in \mathcal{D}'(\Omega\cup \widetilde{\omega}_r)$ 
	is a positive measure on $\widetilde{\omega}_r$ so that $\mu(\Omega)$ and 
	$\mu(\widetilde{\omega}_r)$ are well-defined. Moreover, the constant 
	$C_2>0$ only depends on $C_0$. 
\end{thm}
\subsection{Counting eigenvalues - the Gaussian case}\label{sec:GC}
Let $\Omega \Subset \C$ be an open relatively compact simply connected 
set, possibly dependent on $N$, with a uniformly Lipschitz boundary $\partial \Omega$ 
with associated possibly $N$-dependent weight $0 <r\ll1$, as in Section \ref{sec:LipBD}.
We recall that we suppose that the condition \eqref{i3} holds.
\\
\par
For $p$ satisfying \eqref{gp1} and \eqref{i3} we define for $z\in \C$ 
\begin{equation}\label{c1}
	\phi(z) \defeq \int_{\T^{2d}}  \log |p_0(\rho)-z|  \, d\rho .
\end{equation}
By definition $\phi$ is the logarithmic potential of the direct image measure $(p_0)_*(d\rho)$ 
of the Lebesgue measure $d\rho$ on $\T^{2d}$ under the principal symbol $p_0$ of $p_N$.
The Fubini-Tonelli theorem shows that $\phi \in L^1_{loc}(\C,L(dz))$, and that 
\begin{equation}\label{c2}
		\Delta \phi = 2\pi (p_0)_*(d\rho) \quad \text{in } \mathcal{D}'(\C),
\end{equation}
so $\phi$ is subharmonic. Moreover, by \eqref{i3} we have the following weak form 
of regularity  
\begin{equation}\label{c3}
		\int_{D(z,t^{1/2})} \Delta \phi(w) L(dw) = \mO(t^{\kappa}), \quad 0\leq t \ll1,
\end{equation}
for all $z$ in some neighbourhood of $\partial \Omega$, which shows 
that $\phi(z)$ is continuous there.
\\
\par
Next, pick points 
\begin{equation}\label{c4.0}
		z_j^0 \in \partial\Omega, \quad j=1,\dots, N_0,
\end{equation}
satisfying 
\eqref{eq.LD1.1}. Then 
\begin{equation}\label{c4}
		 N_0 = \mO(r^{-1}),
\end{equation}
and in view of \eqref{eq.LD1.2}, 
\begin{equation}\label{c5}
		\widetilde{w}_r = \widetilde{\partial\Omega}_r  = \partial\Omega + D(0,r).
\end{equation}
Recall that  $N^{-1} \ll \alpha \ll 1$, and assume that
\begin{equation}\label{c5.1}
	0 < t \ll1, \quad  0<\delta \ll N^{-d/2}\alpha^{1/2}.
\end{equation}
Hence, we obtain from \eqref{r6.1} that there exists a constant $C>0$ such that 
with probability $\geq 1 - \e^{-N^d}$, for any $z$ in a neighbourhood of 
$\partial\Omega$ 
\begin{equation}\label{c6}
\log |\det (P^{\delta}-z)| \leq 
N^d \left( \phi(z)  + 
		 	 C\alpha^{\kappa}\log \frac{1}{\alpha}+C\delta N^{\frac{d}{2}}\alpha^{-1/2}\right).
\end{equation}
Moreover, by applying \eqref{gc12.1} 
to all $z_j^0$, $j=1,\dots,N_0$, we get by the union bound that with 
probability $\geq 1 -  \mO(r^{-1}) ( N^{d/2} t +\e^{-N^d})$,
 \begin{equation}\label{c7}
\log |\det (P^{\delta}-z_j^0)| 
\geq N^d \left(\phi(z) + C\alpha^{\kappa}\log \frac{1}{\alpha}+C\delta N^{\frac{d}{2}}\alpha^{-1/2}
		 	 -\varepsilon \right).
\end{equation}
for all $j=1,\dots, N_0$, with 
 \begin{equation}\label{c8}
	\varepsilon = 2C\alpha^{\kappa}\log \frac{1}{\alpha} + 2C \delta N^{\frac{d}{2}}\alpha^{-1/2} 
			    +C\alpha^{\kappa}  \log (t \delta )^{-1}.
\end{equation}
Expressing $t$ in terms of $\varepsilon$, we get 
 \begin{equation}\label{c8.1}
 	t = \alpha^{-2}\delta^{-1} \exp \left( \alpha^{-\kappa}(2\delta N^{\frac{d}{2}}\alpha^{-1/2} - \varepsilon/C) \right).
\end{equation}
Notice that when 
\begin{equation}\label{c8.2}
 	\varepsilon \gg  \alpha^{\kappa}\log \frac{C}{\delta\alpha^2} + \delta N^{\frac{d}{2}}\alpha^{-1/2},
\end{equation}
then $0 < t \ll1 $. Hence, by Theorem \ref{thm:Count}, \eqref{c2} and \eqref{c4} we get that 
\begin{equation}\label{c8.4}
\begin{split}
		&\bigg|\#(\sigma(P^{\delta})\cap \Omega) - {N^d} \int_{p_0^{-1}(\Omega)} d\rho \bigg|  \\
	    &\leq \mO(N^d) \left(
		\int_{p_0^{-1}(\partial\Omega+D(0,r))}d\rho+ \varepsilon N_0 +\sum_{j=1}^{N_0}
		\int_{D\!\left(z_j^0,\frac{r}{4C_1}\right)}\left| \log \frac{|w-z_j^0|}{r} \right |
		(p_0)_*(d\rho)(w)
		\right).
	\end{split}
\end{equation}
with probability 
\begin{equation}\label{c8.5}
	\geq 1 -  \mO(r^{-1}) \left( N^{d/2}  \alpha^{-2}\delta^{-1}
	 \e^{\alpha^{-\kappa}(2\delta N^{\frac{d}{2}}\alpha^{-1/2} - \varepsilon/C)} +\e^{-N^d}\right),
\end{equation}
when 
\begin{equation}\label{c8.3}
 	\varepsilon \gg  \alpha^{\kappa}\log \frac{CN^{d/2}}{\delta\alpha^2} + \delta N^{\frac{d}{2}} \alpha^{-1/2}, 
	\quad 0<\delta \ll N^{-d/2}\alpha^{1/2}.
\end{equation}
By \eqref{c4} we see that the second error term in \eqref{c8.4} is 
 \begin{equation}\label{c8.3.0}
 	 \mO( \varepsilon/r).
\end{equation}
For the last error term in \eqref{c8.4} we have that for any $j\in \{1,\dots,N_0\}$ 
\begin{equation*}
	\begin{split}
		\int_{D\!\left(z_j^0,\frac{r}{4C_1}\right)}\left| \log \frac{|w-z_j^0|}{r} \right |
		 (p_0)_*(d\rho)(w) &= \int_{|p_0 -z| \leq r/(4C_1)}\left| \log \frac{|p_0(\rho)-z_j^0|}{r} \right |d\rho 
		 \\
		 &
		 \leq \int_{0}^{r/(4C_1)}\left| \log \frac{q}{r} \right | dV(q)  \\ 
		 &
		 \leq \mO(r^{1+\kappa}) \int_{0}^{1/(4C_1)}\left| \log q \right | q^{\kappa} dq \\ 
		 & = \mO(r^{1+\kappa}).
	\end{split}
\end{equation*}
which implies that the last error term in \eqref{c8.4} is 
 \begin{equation}\label{c8.3.1}
 	\mO(N_0 r^{1+\kappa}) = \mO(r^{\kappa}).
\end{equation}
Combining this with \eqref{c8.4} yields that 
\begin{equation*}
    	\bigg|\#(\sigma(P^{\delta})\cap \Omega) - {N^d} \int_{p_0^{-1}(\Omega)} d\rho \bigg| 
	   \leq \mO(N^d) \left(
		\int_{p_0^{-1}(\partial\Omega+D(0,r))}d\rho+ \frac{\varepsilon}{r} +r^{\kappa}
		\right).
\end{equation*}
with probability \eqref{c8.5}. 
This concludes the proof of Theorem \ref{thm1}. 
\subsection{Counting Eigenvalues - the universal case}
Let $\Omega \Subset \C$ be as in Section \ref{sec:GC} and 
recall \eqref{c4.0}, \eqref{c4}, \eqref{c5} and \eqref{c1}. 
\par
By \eqref{gc2.4.1}, we have that with probability 
$\geq 1 -N^{-( 1-\tau) \delta_0}$
 \begin{equation}\label{uc1}
	\log |\det (P^{\delta}-z)| \leq N^d \left( \phi(z) + \varepsilon \right)
\end{equation}
 for all $z$ in a neighbourhood of $\partial \Omega$. Here  
 \begin{equation}\label{uc1.1}
	\varepsilon =  	C (N^{-\min(\delta_0,1)\tau \kappa}\log N+N^{-\tau \delta_0 /2}),
\end{equation}
for some sufficiently large $C>0$. 
From \eqref{gc2.11} we deduce that with probability $\geq  1 -  \mO(r^{-1})N^{-( 1-\tau) \delta_0}$, 
 \begin{equation}\label{uc2}
	\log |\det (P^{\delta}-z_j)| \geq N^d \left( \phi(z) 
		 	-\varepsilon
			\right).
\end{equation}
for all $j=1,\dots, N_0$. Thus, by Theorem \ref{thm:Count} (with $N^d$ instead of $N$), 
\eqref{c2}, \eqref{c4}, and there error estimates \eqref{c8.3.0}, \eqref{c8.3.1}, we get that 
\begin{equation}\label{uc3}
		\bigg|\#(\sigma(P^{\delta})\cap \Omega) - {N^d} \int_{p_0^{-1}(\Omega)} d\rho \bigg| 
	    \leq \mO(N^d) \left(
		\int_{p_0^{-1}(\partial\Omega+D(0,r))}d\rho+ \frac{\varepsilon}{r} +r^{\kappa}
		\right).
\end{equation}
with probability 
\begin{equation}\label{uc4}
	\geq  1 -  \mO(r^{-1})N^{-( 1-\tau) \delta_0}.
\end{equation}

\section{Weak convergence of the empirical measure}\label{sec:WC}
We work under the assumptions of Corollary \ref{cor1}. 
Recall that the empirical measure of the eigenvalues of 
$P_{\delta} =p_N + \delta Q_{\omega}$ is given by 
\begin{equation*}
		\mu_N = N^{-d} \sum_{\lambda \in \spec(P_{\delta})} \delta_{\lambda}.
\end{equation*}
Furthermore, we set 
\begin{equation*}
		\mu = (p_0)_*(d\rho), 
\end{equation*}
which has compact support. Notice that also the support of $\mu_N$ is contained 
some fixed $N$-independent compact set with probability close to $1$. Indeed, 
it follows by \eqref{gc2.2} that 
\begin{equation}\label{wc0.1}
\delta \Vert Q_{\omega}\Vert \leq \mO(1) N^{ -\tau \delta_0}, 
\text{ with probability } \geq 1 - N^{-(1-\tau)\delta_0},
\end{equation}
which in combination with Proposition \ref{prop:qf2} yields that 
\begin{equation}\label{wc0.2}
\| P^{\delta} \| \leq \mO(1), 
\text{ with probability } \geq 1 - N^{-(1-\tau)\delta_0}. 
\end{equation}

To prove Corollary \ref{cor1} one may either use directly Theorem \ref{thm2}, or, 
observing that $\log|\det(P^{\delta}-z)|$ is the logarithmic potential of $\mu_N$, 
use \eqref{uc1} and \eqref{uc1.1} in combination with \cite[Theorem 2.8.3]{Ta12}. 
In the following, we present the first approach. 
\begin{proof}[Proof of Corollary \ref{cor1}]
0. Let $\varepsilon > 0$, and set 
\begin{equation}\label{wc1.0}
	\Omega = \Omega_{\varepsilon,j,k} = [j,j+1[\varepsilon + i  [k,k+1[\varepsilon, 
	\quad j,k \in \Z.
\end{equation}
Notice that $\Omega$ is uniformly Lipschitz with respect to the constant, 
possibly $N$-dependent Lipschitz weight $0 < r \ll 1$, as defined in Section \ref{sec:c1}. 
\\
\par
1. We begin by proving Corollary \ref{cor1} in the case when $\delta_0 > 1$, and 
we follow the strategy of \cite[Section 7.1]{SjVo19a}. By choosing 
$r = N^{-\tau\kappa/2}$, $0< \tau \leq 1$, we know from Theorem \ref{thm2} that 
\begin{equation}\label{wc1}
	\left| \mu_N(\Omega)- \mu(\Omega) \right| = \mO(1)\left(
		\int_{p_0^{-1}(\partial\Omega+D(0,r))}d\rho+ 
		N^{-\tau \kappa/2}\log N+N^{-\tau (\delta_0 - \kappa)/2} +N^{-\tau\kappa^2/2}
		\right).
\end{equation}
with probability $\geq 1 - \mO(N^{\tau\kappa /2})N^{-( 1-\tau) \delta_0}$. Notice first 
that since $\delta_0>1$, we have for $\tau>0$ sufficiently small that 
\begin{equation}\label{wc1.9}
		\frac{\tau\kappa}{2} -( 1-\tau) \delta_0 = \beta < -1,
\end{equation}
which implies that 
\begin{equation}\label{wc2}
		\sum_{N=1}^{\infty} N^{\beta} < + \infty. 
\end{equation}
Since \eqref{i3} is assumed to hold uniformly for all $z\in \C$, it follows that 
$p_0$ is non-constant, so the Morse-Sard theorem implies that $p_0^{-1}(\partial\Omega)$ 
has Lebesgue measure $0$. The regularity of the Lebesgue measure then shows that
\begin{equation}\label{wc3}
	\int_{p_0^{-1}(\partial\Omega+D(0,r))}d\rho= o(1), \quad N\to \infty.
\end{equation}
In particular, this shows that the error term on the right hand side of \eqref{wc1} is 
$o(1)$ as $N\to \infty$. From \eqref{wc2}, we know that the probability that \eqref{wc1} does not hold is 
summable. Hence, the Borel-Cantelli lemma yields that, for any $\Omega$ \eqref{wc1.0}, almost surely, 
\begin{equation}\label{wc4}
	 \mu_N(\Omega) \longrightarrow \mu(\Omega), \quad N\to \infty. 
\end{equation}
Let $\varepsilon_\nu >0 $, $\nu\in \N$, be a decreasing sequence tending to $0$. 
Since the countable union of sets of probability $0$ has probability $0$, it follows that,  
almost surely, \eqref{wc4} holds for all $\Omega_{\varepsilon_\nu,j,k}$.
\par
Let $G$ be the set of all step functions of the form 
\begin{equation}\label{wc4.1}
		\psi= \sum_{j,k} g_{j,k} \mathbf{1}_{\Omega_{\varepsilon_\nu,j,k}}, \quad 
		g_{j,k} \in \mathds{Q}. 
\end{equation}
Then, almost surely, we have that for every $\psi \in G$ 
\begin{equation}\label{wc5}
		\int \psi d\mu_N \longrightarrow \int \psi d\mu, \quad N \to \infty. 
\end{equation}
Notice that since $\tau >0$ has been chosen sufficiently small so that \eqref{wc1.9} holds, 
it follows from \eqref{wc0.2} 
and the Borel-Cantelli Lemma that almost surely the support of $\mu_N$ is contained 
in some fixed compact set. Hence, to prove the weak convergence of $\mu_N$ to $\mu$ 
it is sufficient to consider compactly supported test functions. 
\par
Let $\phi\in C_c(\C;\R)$. For every $\varepsilon >0$, we can find a $\psi \in G$, such 
that $\| \psi - \phi\|_{\infty} \leq \varepsilon$. Since $\mu_N$ and $\mu$ are probability 
measures, we get that 
\begin{equation}\label{wc6}
		\left| \int \psi d\mu_N -  \int \phi d\mu_N \right | \leq \varepsilon, 
		\quad 
		\left| \int \psi d\mu -  \int \phi d\mu \right | \leq \varepsilon.
\end{equation}
 Combining \eqref{wc5} and \eqref{wc6}, it follows that almost surely, for all $\phi\in C_c(\C;\R)$ 
 and all $\varepsilon >0$ we have that 
 \begin{equation*}
	\limsup\limits_{N\to\infty} \left| \int \phi d\mu_N -  \int \phi d\mu \right | \leq 2 \varepsilon.
\end{equation*}
 Hence, almost surely 
 \begin{equation*}
	  \mu_N  \rightharpoonup \mu = (p_0)_*(d\rho),
\end{equation*}
proving the first claim of Corollary \ref{cor1}. 
\\
\par 
2. Now we turn to the case when $\delta_0 \in ]0,1]$. We recall that 
$\mu_N  \rightharpoonup \mu$ in probability, means that for all $\eta>0$ 
and all $\phi\in C_b(\C;\R)$ 
 \begin{equation}\label{wc7}
	  \prob \left( |\mu_N(\phi)-\mu(\phi)| \geq \eta \right) = o(1), \quad N\to \infty.
\end{equation}
We begin by reducing to the case of compactly supported test functions. From 
\eqref{wc0.2} we know that there exists a compact set $K\Subset \C$ such that 
 \begin{equation}\label{wc7.1}
	  \prob \left(\supp \mu_N \not\subset K \right) = o(1), \quad N\to \infty.
\end{equation}
After possibly enlarging $K$ so that also $\supp \mu \subset K$, we let $\widetilde{K}\Subset \C$ 
be an open relatively compact set such that $\overline{K+D(0,\gamma)}\subset \widetilde{K}$, 
for some $\gamma>0$. Let $\chi\in C_c^{\infty}(\C;[0,1])$ with $\chi \equiv 1$ on $\overline{K}$ 
and $\chi\equiv 0 $ on $\widetilde{K}^c$. Then, by \eqref{wc7.1} we get that 
for all $\phi\in C_b(\C;\R)$ 
 \begin{equation*}
 \begin{split}
	  \prob \left( |\mu_N(\phi)-\mu(\phi)| \geq \eta \right) & = 
	  \prob \left( |\mu_N(\phi)-\mu(\phi)| \geq \eta \text{ and } \supp \mu_N \subset K\right)  \\
	  &\phantom{=}+ \prob \left( |\mu_N(\phi)-\mu(\phi)| \geq \eta \text{ and } \supp \mu_N \not\subset K \right) \\ 
	 &=  \prob \left( |\mu_N(\phi)-\mu(\phi)| \geq \eta \text{ and } \supp \mu_N \subset K\right) +o(1) \\
         &\leq  \prob \left( |\mu_N(\phi \chi)-\mu(\phi\chi)| \geq \eta \right) +o(1) .
\end{split}
\end{equation*}
Hence, it is enough to show \eqref{wc7} for all test functions $\phi\in C_c(\widetilde{K};\R)$.
\\
\par
Let $\Omega$ be as in \eqref{wc1.0}. 
Picking $r= N^{-\tau \delta_0\kappa/4}$ and $\tau =1/5$, we get from Theorem \ref{thm2} 
and \eqref{wc3} that 
\begin{equation*}
	\left| \mu_N(\Omega)- \mu(\Omega) \right| = o(1), 
\end{equation*}
with probability $\geq 1 - \mO(1)N^{-3\delta_0/4}$. Hence, for any $\eta >0$  
\begin{equation*}
	\prob \left ( \left| \mu_N(\Omega)- \mu(\Omega) \right| \geq \eta \right) = o(1), 
	\quad 
	N\to \infty.
\end{equation*}
Let $\phi \in C_c(\widetilde{K};\R)$ and recall \eqref{wc4.1}. Then, for any $\varepsilon >0$ 
we can find a $\psi \in G$ such that $\| \phi - \psi \|_{\infty} \leq \varepsilon$ and with all but 
finitely many, say $M_{\varepsilon}>0$, $g_{j,k} =0$. Then, for any $\eta'>0$ we 
get by the union bound that 
\begin{equation*}
\begin{split}
\prob  ( | \mu_N(\Omega_{\varepsilon,j,k})- & \mu(\Omega_{\varepsilon,j,k}) | 
	\leq \eta' \text{ whenever } g_{j,k} \neq 0  )  \\ 
	&= 1 - \prob \left [\neg ( \left| \mu_N(\Omega_{\varepsilon,j,k})- \mu(\Omega_{\varepsilon,j,k}) \right| 
	\leq \eta' \text{ whenever } g_{j,k} \neq 0)  \right] \\ 
	&\geq  1 - \sum_{j,k :~ g_{j,k} \neq 0}
	 \prob \left [ \left| \mu_N(\Omega_{\varepsilon,j,k})- \mu(\Omega_{\varepsilon,j,k}) \right| 
	\geq \eta'  \right] \\ 
	&\geq  1 - M_{\varepsilon}o(1).
\end{split}
\end{equation*}
So, since $|g_{j,k}| \leq \|\phi\|_{\infty}+\varepsilon$, we get that 
\begin{equation}\label{wc8}
\left| \int\psi d\mu_N- \int \psi d\mu \right| 
\leq \sum_{j,k} |g_{j,k}| \left| \mu_N(\Omega_{\varepsilon,j,k})- \mu(\Omega_{\varepsilon,j,k}) \right| 
\leq M_{\varepsilon}(\|\phi\|_{\infty}+\varepsilon) \eta'
\end{equation}
with probability $\geq 1 - M_{\varepsilon}o(1)$. Moreover, since $\mu_N$ and $\mu$ are probability 
measures, we get that 
\begin{equation}\label{wc9}
		\left| \int \psi d\mu_N -  \int \phi d\mu_N \right | \leq \varepsilon, 
		\quad 
		\left| \int \psi d\mu -  \int \phi d\mu \right | \leq \varepsilon.
\end{equation}
Let $\eta>0$, and set $\varepsilon =\eta$ and $\eta' = \eta(M_{\varepsilon}(\|\phi\|_{\infty} + \varepsilon))^{-1}$ 
in the \eqref{wc8} and \eqref{wc9}. Then, 
\begin{equation}\label{wc10}
	\begin{split}
		\left| \int \phi d\mu_N -  \int \phi d\mu \right |  &\leq 
		\left| \int \phi d\mu_N -  \int \psi d\mu_N \right | + 
		\left| \int \psi d\mu_N -  \int \psi d\mu \right |  \\ 
		&\phantom{=}+\left| \int \psi d\mu -  \int \phi d\mu \right | \\ 
		&\leq 3\eta,
	\end{split}
\end{equation}
with probability $\geq 1 - o(1)$ as $N\to \infty$, and we conclude the second statement of 
Corollary \ref{cor1}.
\end{proof}
\appendix
\section{Estimate on the smallest singular value}\label{app1}
We present for the reader's convenience a complex version of a result 
due to Sankar, Spielmann and Teng \cite[Lemma 3.2]{SaTeSp06}, see also 
\cite[Theorem 2.2]{TaVu10}.
\begin{thm}\label{thm:SSV}
There exists a constant $C>0$ such that the following holds. 
Let $N\geq 2$, let $X_0$ be an arbitrary complex $N\times N$ matrix, 
and let $Q$ be an $N\times N$ complex 
Gaussian random matrix whose entries are 
all independent copies of a 
complex Gaussian random variable $q\sim \mathcal{N}_{\C}(0,1)$. Then, for any $\delta >0$
\begin{equation*}
	\prob \left(s_N(X_0 + \delta Q) < \delta t \right) \leq C
	N t^2.
\end{equation*}
\end{thm}
The proof, which is a straightforward modification of the proof of \cite[Lemma 3.2]{SaTeSp06}, 
is presented here for the reader's convenience. 
\begin{lem}\label{lem:a1}
There exists a $C>0$ such that for any $\nu\in\C^N$ with $\|\nu\|=1$, we have that for any $\tau>0$ 
\begin{equation*}
	\prob \left( \|(X_0 + \delta Q)^{-1}\nu\| > \tau  \right) \leq C
	\frac{1}{ \tau^2 \delta^2}.
\end{equation*}
\end{lem}
\begin{proof}
	1. 
	Since $Q$ is a Gaussian random matrix, it is clear that the zero set 
	of the map $\C^{N\times N} \ni Q\mapsto \det (X_0 + \delta Q)$ has Lebesgue measure $0$. 
	Hence, $X:=X_0 + \delta Q$ is almost surely invertible.
	\par
	Let $U$ be a $N\times N$ unitary matrix such that $U^*e_1 = \nu$, where 
	$e_1$ is the unit vector in $\C^N$ with $1$ in the first entry, and $0$ in the other entries. 
	Write $B=UX$ and $B_0 = UX_0$, then, almost surely,  
	\begin{equation*}
		\| X^{-1} \nu \| = \| X^{-1} U^*e_1\| = \| B^{-1} e_1\|.
	\end{equation*}
	We denote by $\widehat{b}$ the first column of $B^{-1}$, and by 
	$b_j$, $j=1,\dots,N$, the rows of $B$, hence $\widehat{b}\cdot b_1=1$ 
	and $\widehat{b}\cdot b_j=0$, $j=2,\dots, N$.
	Notice that the $b_i$ are linearly independent and let $t$ denote the 
	unit vector which is orthogonal to the space spanned by the 
	$\overline{b}_j$ for $j=2,\dots, N$. Then, 
	\begin{equation*}
		\widehat{b} = ( \widehat{b} | t ) \,t, \quad \text{and } \quad  ( \widehat{b} | t )  (t \cdot b_1)=1.
	\end{equation*}
	Hence, $\| B^{-1} e_1\| = | t \cdot b_1|^{-1}$, and 
	\begin{equation}\label{ap0}
	\prob \left( \|(X_0 + \delta Q)^{-1}\nu\| > \tau  \right) = 
	\prob \left( | t \cdot b_1| < \tau^{-1}  \right).
	\end{equation}
	\par
	2. 
	Since $U$ is unitary it follows that the entries of $\delta UQ$ are independent and identically 
	distributed complex Gaussian random variables $\sim \mathcal{N}_{\C}(0,\delta^2)$. Since $t$ 
	is a unit vector depending only on $b_i$, $i=2,\dots, N$, it follows that when fixing these row vectors, 
	$t \cdot b_1-\lambda$, with $\lambda = ( B_0 t)_1$, is a complex Gaussian random variables 
	$\sim \mathcal{N}_{\C}(0, \delta^2)$. Then, 
	\begin{equation*}
	\prob \left( | t \cdot b_1| < \tau^{-1}  \right) 
	\leq 	\prob \left( | t \cdot b_1 - \lambda| < \tau^{-1}  \right) 
	\leq \frac{1}{\tau^2 \delta^2},
	\end{equation*}
	and we conclude the second statement of the Lemma. Here,
	 the second inequality follows from a straightforward calculation. To see the first 
	inequality, it is enough to show that for a complex Gaussian random variable 
	$u\sim  \mathcal{N}_{\C}(0, \delta^2)$, we have that for any $b>0$, $x \geq 0$,
	\begin{equation}\label{ap1}
	\prob \left( | x + \Rea u | < b  \right) 
	\leq 	\prob \left( | \Rea u | < b  \right). 
	\end{equation}
	The left hand side is equal to the integral 
	\begin{equation*}
	 I(x) = \frac{1}{\delta \sqrt{\pi}} \int_{-(x+b)}^{b-x} \e^{-t^2/\delta^2} dt.
	\end{equation*}
	Since
	\begin{equation*}
	 \frac{d}{dx}I(x) = \frac{1}{\delta \sqrt{\pi}} \left[  \e^{-(x+b)^2/\delta^2} -  \e^{-(b-x)^2/\delta^2} \right] \leq 0,
	\end{equation*}
	the map $x\mapsto I(x)$ is decreasing, so \eqref{ap1} holds, as it is trivially true 
	for $x=0$. 
\end{proof}
\begin{lem}\label{lem:a2}
	There exists a constant $C>1$ such that the following holds.
	Let $N\geq 2$, and let $\nu$ be a uniformly distributed random 
	unit vector in $\C^N$. Then, for any $0 < c \leq 1$
	\begin{equation*}
	\prob \left( |\nu_1| \geq \sqrt{\frac{c}{CN}} \right) \geq (1-\e^{-2}) \prob \left (|z| > \sqrt{c}\right),
	\end{equation*}
	where $z\sim \mathcal{N}_{\C}(0,1)$.
\end{lem}
\begin{proof}
	Let $z\in \C^N$ be a random vector whose entries 
	$z_j\sim \mathcal{N}_{\C}(0,1)$, $j=1,\dots,N$ are
	 independent and identically distributed 
	complex Gaussian random variables. Then, 
	\begin{equation*}
	\nu_*(d\prob) = \left(\frac{z}{\|z\|}\right)_*(d\prob)
	\end{equation*}
	Writing $z=(z_1,z')$, we get that
	\begin{equation*}
		\begin{split}
			\prob \left( |\nu_1| \geq \sqrt{\frac{c}{C N}} \right)  
			&=
			\prob \left( \frac{|z_1^2|}{\|z\|^2} \geq \frac{c}{CN} \right) \\ 
			&=
			\prob \left( \frac{(CN-1)|z_1^2|}{\|z'\|^2} \geq \frac{(CN-1)c}{CN-c} \right) \\ 
			&\geq 
			\prob \left( \frac{(CN-1)|z_1^2|}{\|z'\|^2} \geq c \right) \\ 
			&\geq 
			\prob \left( \|z'\|^2 \leq (CN-1) \right)
			\prob \left( |z_1^2|\geq c \right).
		\end{split}
	\end{equation*}
	Since $z'$ is a complex Gaussian random vector in $\C^{N-1}$ with 
	independent and identically distributed entries $\sim\mathcal{N}_{\C}(0,1)$, 
	we get from Markov's inequality that for $C>1$ large enough 
	\begin{equation*}
		\begin{split}
			\prob \left( \|z'\|^2 \leq (CN-1) \right) 
			\geq 1 -\e^{-N},
		\end{split}
	\end{equation*}
	and the statement of the Lemma follows.
\end{proof}
\begin{proof}[Proof of Theorem \ref{thm:SSV}]
	Let $\nu$ be a uniformly distributed random unit vector in $\C^N$. By Lemma 
	\ref{lem:a1} we know that for any $\tau >0$ 
	\begin{equation}\label{ap2}
	\prob \left( \|(X_0 + \delta Q)^{-1}\nu\| > \tau  \right) \leq C
	\frac{1}{ \tau^2 \delta^2}.
	\end{equation}
	Write $X:=X_0 + \delta Q$, and let $u$ be the unit eigenvector of $XX^*$ corresponding 
	to its smallest eigenvalue $t_1^2 \geq 0$, i.e. 
	\begin{equation}\label{ap3}
		XX^* u = t_1^2 u.
	\end{equation}
	Then, almost surely, 
	\begin{equation}\label{ap4}
		\| X^{-1} u\|  = t_1^{-1}\| u \| = \|X^{-1}\|.
	\end{equation}
	Writing $\nu = \nu_0 u + \nu^{\perp}$, with $\nu_0=(\nu| u)$ and $\nu^{\perp}$ orthogonal 
	to $u$, we see that, almost surely, 
	\begin{equation*}
		\begin{split}
			\| X^{-1} \nu \|^2 &= ( (XX^*)^{-1}   \nu_0 u + \nu^{\perp}|  \nu_0 u + \nu^{\perp}) \\ 
			& = |\nu_0|^2 ( (XX^*)^{-1}   u |u ) +( (XX^*)^{-1}   \nu^{\perp} |\nu^{\perp}) \\ 
			& \geq |\nu_0|^2\|X^{-1}\|^2.
		\end{split}
	\end{equation*}
	Let $C>1$ be as in Lemma \ref{lem:a2}. Then, using the above, we get that 
	for any $0 < c\leq 1$ and any $\tau>0$, 
	 \begin{equation}\label{ap5}
		\begin{split}
			\prob \left(\| X^{-1} \nu\| > \tau \sqrt{\frac{c}{C N}} \right)  
			&\geq \prob \left(|\nu_0|\|X^{-1}\| > \tau \sqrt{\frac{c}{C N}} \right)  \\
			&\geq \prob \left(\|X^{-1}\|> \tau 
			\text{ and } |\nu_0| >   \sqrt{\frac{c}{C N}}\right)  \\
			 &\geq \prob \left(\|X^{-1}\|> \tau \right) \min\limits_{ Q: \|X^{-1}\| > \tau}
			 \prob \left( |\nu_0| >   \sqrt{\frac{c}{C N}} \right).
		\end{split}
	\end{equation}
	 Since the distribution of $\nu$ is invariant a under unitary change of variables, we 
	 may express $\nu$ in an orthonormal basis of $\C^N$ which has $u$ as its first 
	 vector, wherefore the first component of $\nu $ is $(\nu|u)$. Thus, using Lemma \ref{lem:a2} 
	 we obtain from \eqref{ap5} that 
	  \begin{equation*}
			\prob \left(\| X^{-1} \nu\| > \tau \sqrt{\frac{c}{C N}} \right)  
			\geq \prob \left(\|X^{-1}\|> \tau \right)
			(1-\e^{-2}) \prob \left (|z| > \sqrt{c}\right), 
	\end{equation*}
	 where $z \sim \mathcal{N}_{\C}(0,1)$ is a complex Gaussian random variable. 
	 This, together with \eqref{ap2}, then yields that there exists a constant $C>0$ such 
	 that
	 \begin{equation*}
			\prob \left(\|X^{-1}\|> \tau \right) 
			\leq   \frac{CN}{c\,\prob \left (|z| > \sqrt{c}\right)\tau^2 \delta^2}. 
	\end{equation*}
	Since, we may choose $c \in ]0,1]$, we take $c=1$, which gives that 
	$c\,\prob \left (|z| > \sqrt{c}\right) = e^{-1}$. Recall that $\| X^{-1} \| = s_N(X)^{-1}$, 
	so taking $t =(\tau \delta)^{-1}$, we deduce that there exists a constant $C>0$ 
	such that for any $N\geq 2$ 
	 \begin{equation*}
			\prob \left( s_N(X) < t\delta \right) 
			\leq  CNt^2. \qedhere 
	\end{equation*}
\end{proof}

\providecommand{\bysame}{\leavevmode\hbox to3em{\hrulefill}\thinspace}
\providecommand{\MR}{\relax\ifhmode\unskip\space\fi MR }
\providecommand{\MRhref}[2]{%
  \href{http://www.ams.org/mathscinet-getitem?mr=#1}{#2}
}
\providecommand{\href}[2]{#2}


\begin{thebibliography}{10}

\bibitem{BPZ18b}
A.~Basak, E.~Paquette, and O.~Zeitouni, \emph{Spectrum of random perturbations
  of toeplitz matrices with finite symbols}, preprint
  https://arxiv.org/pdf/1812.06207.pdf (2018).

\bibitem{BPZ18}
\bysame, \emph{Regularization of non-normal matrices by gaussian noise - the
  banded toeplitz and twisted toeplitz cases}, Forum of Math, Sigma \textbf{7}
  (2019).

\bibitem{BoUr03}
D.~Borthwick and A.~Uribe, \emph{On the pseudospectra of berezin-toeplitz
  operators}, Meth. and Appl. of Anlaysis \textbf{10} (2003), no.~1, 031--066.

\bibitem{ZwChrist10}
T.J. Christiansen and M.~Zworski, \emph{{Probabilistic Weyl Laws for Quantized
  Tori}}, Communications in Mathematical Physics \textbf{299} (2010), 305--334.

\bibitem{Da99b}
E.B. Davies, \emph{{Semi-classical States for Non-Self-Adjoint Schr{\"o}dinger
  Operators}}, Comm. Math. Phys (1999), no.~200, 35--41.

\bibitem{DaHa09}
E.B. Davies and M.~Hager, \emph{{Perturbations of Jordan matrices}}, J. Approx.
  Theory \textbf{156} (2009), no.~1, 82--94.

\bibitem{NSjZw04}
N.~Dencker, J.~Sj{\"o}strand, and M.~Zworski, \emph{{Pseudospectra of
  semiclassical (pseudo-) differential operators}}, Communications on Pure and
  Applied Mathematics \textbf{57} (2004), no.~3, 384--415.

\bibitem{DiSj99}
M.~Dimassi and J.~Sj{\"o}strand, \emph{{Spectral Asymptotics in the
  Semi-Classical Limit}}, {London Mathematical Society Lecture Note Series
  268}, Cambridge University Press, 1999.

\bibitem{TrEm05}
M.~Embree and L.~N. Trefethen, \emph{{Spectra and Pseudospectra: The Behavior
  of Nonnormal Matrices and Operators}}, Princeton University Press, 2005.

\bibitem{Fl18}
Y.~Le Floch, \emph{A brief introduction to berezin--toeplitz operators on
  compact k{\"a}hler manifolds}, Springer, Cham, 2018.

\bibitem{GoKr69}
I.C. Gohberg and M.G. Krein, \emph{Introduction to the theory of linear
  non-selfadjoint operators}, Translations of mathematical monographs, vol.~18,
  AMS, 1969.

\bibitem{GuMaZe14}
A.~Guionnet, P.~Matchett Wood, and {0. Zeitouni}, \emph{{Convergence of the
  spectral measure of non-normal matrices}}, Proc.~AMS \textbf{142} (2014),
  no.~2, 667--679.

\bibitem{Ha06}
M.~Hager, \emph{{Instabilit{\'e} spectrale semiclassique pour des
  op{\'e}rateurs non-autoadjoints I: un mod{\`e}le}}, Annales de la facult{\'e}
  des sciences de Toulouse S{\'e}. 6 \textbf{15} (2006), no.~2, 243--280.

\bibitem{HaSj08}
M.~Hager and J.~Sj{\"o}strand, \emph{{Eigenvalue asymptotics for randomly
  perturbed non-selfadjoint operators}}, Mathematische Annalen \textbf{342}
  (2008), 177--243.

\bibitem{Ho83}
L.~H{\"o}rmander, \emph{{The Analysis of Linear Partial Differential Operators
  I}}, {Grundlehren der mathematischen Wissenschaften}, vol. 256,
  Springer-Verlag, 1983.

\bibitem{Kal97}
O.~Kallenberg, \emph{Foundations of modern probability}, {Probability and its
  Applications}, Springer, 1997.

\bibitem{La05}
R.~Latala, \emph{Some estimates of norms of random matrices}, Proc. Amer. Math.
  Soc. \textbf{133} (2005), no.~5, 1273--1282.

\bibitem{TrCh04}
S.J.~Chapman L.N.~Trefethen, \emph{Wave packet pseudomodes of twisted toeplitz
  matrices}, Comm. on Pure and Applied Mathematics \textbf{LVII} (2004),
  1233--1264.

\bibitem{Ma02}
A.~Martinez, \emph{An introduction to semiclassical and microlocal analysis},
  Springer, 2002.

\bibitem{MeSj02}
A.~Melin and J.~Sj{\"o}strand, \emph{Determinants of pseudodifferential
  operators and complex deformations of phase space}, Methods Appl. Anal.
  \textbf{9} (2002), no.~2, 177--237.

\bibitem{NoZw07}
S.~Nonnenmacher and M.~Zworski, \emph{Distribution of resonances for open
  quantum maps.}, Comm. Math. Phys (2007), no.~269, 311--365.

\bibitem{SaTeSp06}
A.~Sankar, D.A. Spielmann, and S.H. Teng, \emph{Smoothed analysis of the
  condition numbers and growth factors of matrices}, SIAM J, Matrix Anal. Appl.
  \textbf{28} (2006), no.~2, 446--476.

\bibitem{Sj09b}
J.~Sj{\"o}strand, \emph{Counting zeros of holomorphic functions of exponential
  growth}, Journal of pseudodifferential operators and applications \textbf{1}
  (2010), no.~1, 75--100.

\bibitem{SjVo19b}
J.~Sj{\"o}strand and M.~Vogel, \emph{General toeplitz matrices subject to
  gaussian perturbations}, preprint arxiv.org/abs/1905.10265 (2019).

\bibitem{SjVo19a}
\bysame, \emph{Toeplitz band matrices with small random perturbations},
  preprint arxiv.org/abs/1901.08982 (2019).

\bibitem{SjZw07}
J.~Sj{\"o}strand and M.~Zworski, \emph{{Elementary linear algebra for advanced
  spectral problems}}, Annales de l'Institute Fourier \textbf{57} (2007),
  2095--2141.

\bibitem{Ta12}
T.~Tao, \emph{{Topics in Random Matrix Theory}}, {Graduate Studies in
  Mathematics}, vol. 132, American Mathematical Society, 2012.

\bibitem{TaVu10}
T.~Tao and V.~Vu, \emph{Smooth analysis of the condition number and the least
  singular value}, Math. Comp. \textbf{79} (2010), no.~272, 2333--2352 (see
  also the Erratum arxiv.org/pdf/0805.3167v3.pdf).

\bibitem{Zw01}
M.~Zworski, \emph{{A remark on a paper of E.B. Davies}}, Proc. A.M.S. (2001),
  no.~129, 2955--2957.

\bibitem{Zw12}
\bysame, \emph{{Semiclassical Analysis}}, {Graduate Studies in Mathematics
  138}, American Mathematical Society, 2012.

\end{thebibliography}
\end{document}